\documentclass[10pt,a4paper]{article}

\usepackage{xcolor}

\usepackage{amsfonts}
\usepackage{amsmath}
\usepackage{amsbsy}
\usepackage{theorem}
\usepackage{graphicx}

\newtheorem{proposition}{Proposition}
\newtheorem{lemma}{Lemma}

\newtheorem{remark}{Remark}

\begin{document}
\sf
\title{\bfseries Local Fr\'echet  functional regression in manifolds from time-correlated bivariate curve data
}
\date{}

 \maketitle
 \author{\hspace*{-1cm} M.D. Ruiz-Medina $^{1}$ and A. Torres-Signes$^{2}$\\
$^{1}$ University of Granada\\
$^{2}$ University of M\'alaga}

 \begin{abstract}
 Under mild conditions, a least-squares local linear Fréchet curve
predictor is derived for a response and a regressor evaluated in a
separable Hilbert space. The conditions that allow the implementation of
the local linear Fréchet functional predictor in the ambient L²-space of
vector functions, with values in the time-varying tangent space of a
compact Riemannian manifold, are established. An intrinsic local linear
Fréchet curve predictor on such a manifold is then proposed, based on a
weighted Fréchet mean approach. Its asymptotic optimality is proved.
Simulations and a real-data application are considered to analyze the
finite-sample performance of the empirical versions of both predictors,
compared with a geodesic Nadaraya-Watson-type curve predictor. In the
real-data application, the functional prediction of the time-varying
spherical coordinates of the Earth's magnetic field is addressed using
observations through time of the geocentric latitude and longitude of the NASA MAGSAT
spacecraft.
\end{abstract}


\newpage
\subsection*{Highlights}

\begin{itemize}
\item Research highlight 1. The gap in least-squares local linear Fr\'echet functional regression, for responses and regressors evaluated in a   separable Hilbert space, is bridged.
\item Research highlight 2. Extrinsic and intrinsic  local linear Fr\'echet curve regression is addressed,  from time-correlated bivariate curve data evaluated in a compact Riemannian manifold.
\end{itemize}
%%%% Quito el entorno figura y el caption de lo que sigue, porque no venía en la plantilla:
%%%% \caption{Time--varying geocentric  latitude and longitude of the satellite NASA's MAGSAT  spacecraft, and the time--varying  spherical coordinates of Earth's magnetic vector  field, November, 1979 (two top-lines), and January, 1980  (two-bottom lines)}
\newpage
\subsection*{Graphical abstract}
\includegraphics[width=0.25\textwidth]{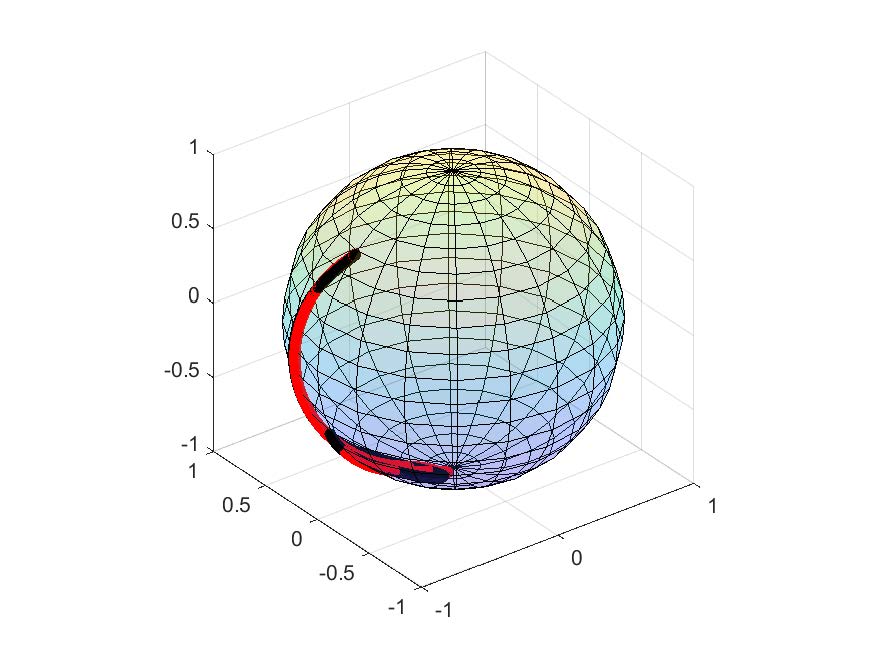}
\includegraphics[width=0.25\textwidth]{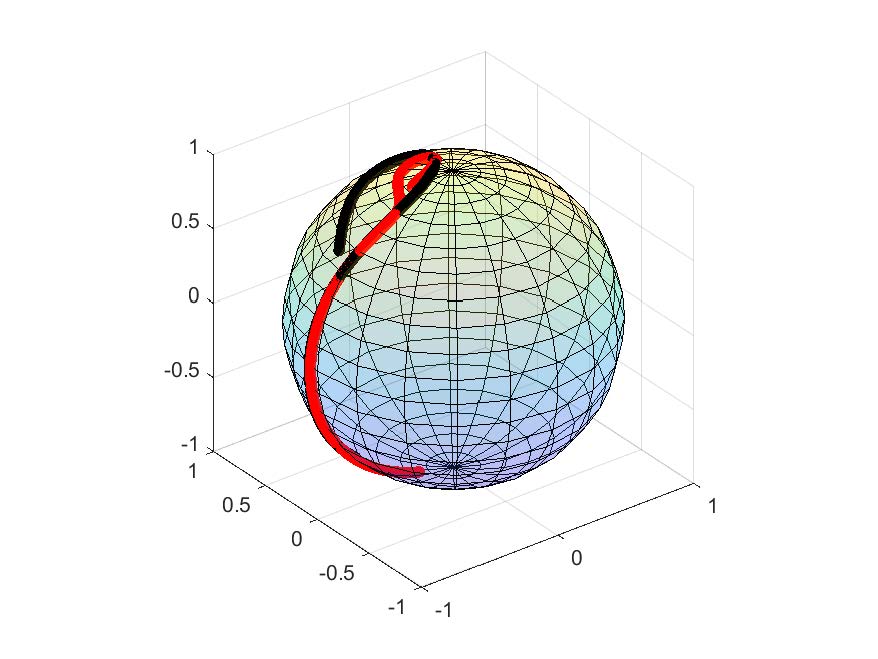}
\includegraphics[width=0.25\textwidth]{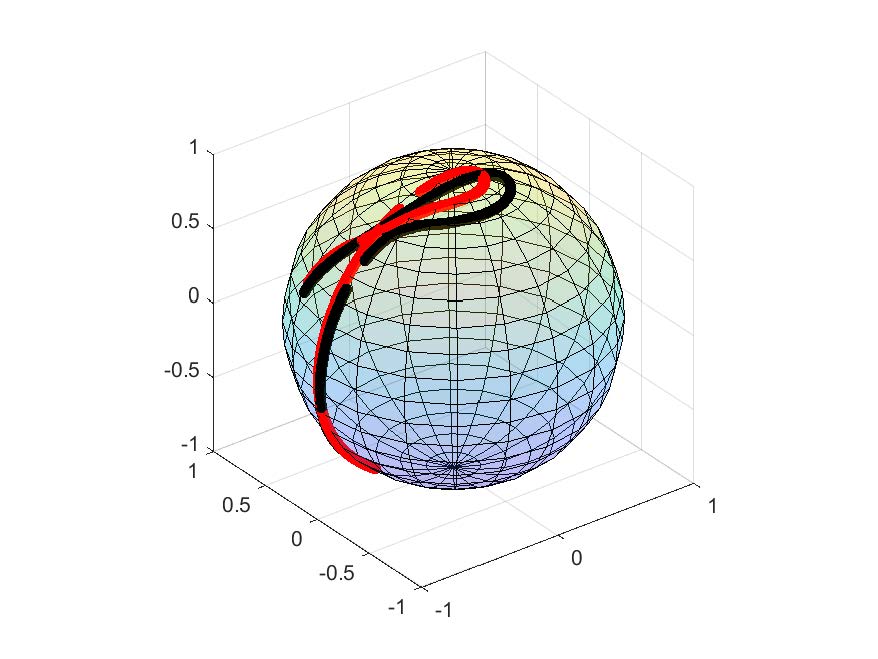}
\includegraphics[width=0.25\textwidth]{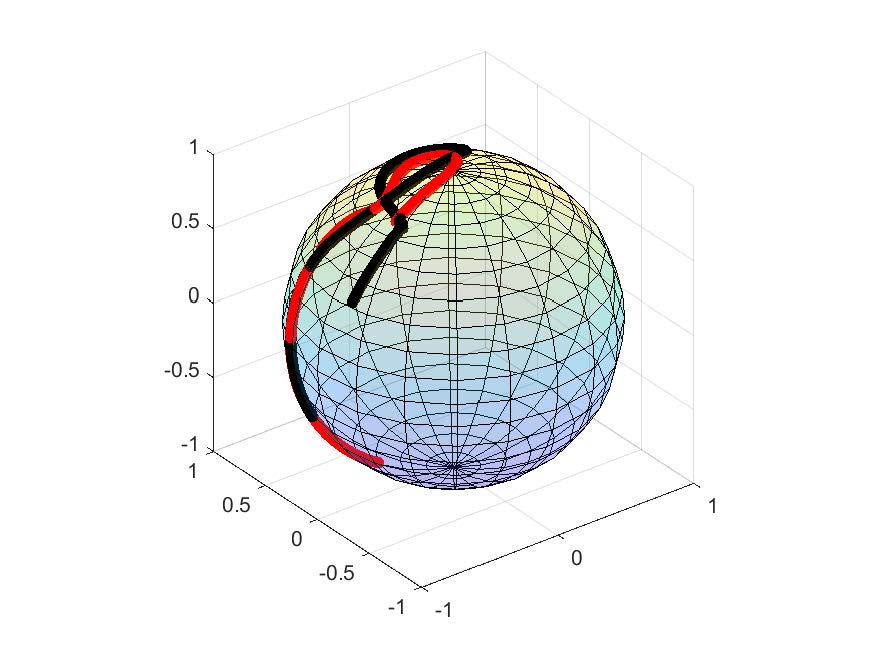}
\includegraphics[width=0.25\textwidth]{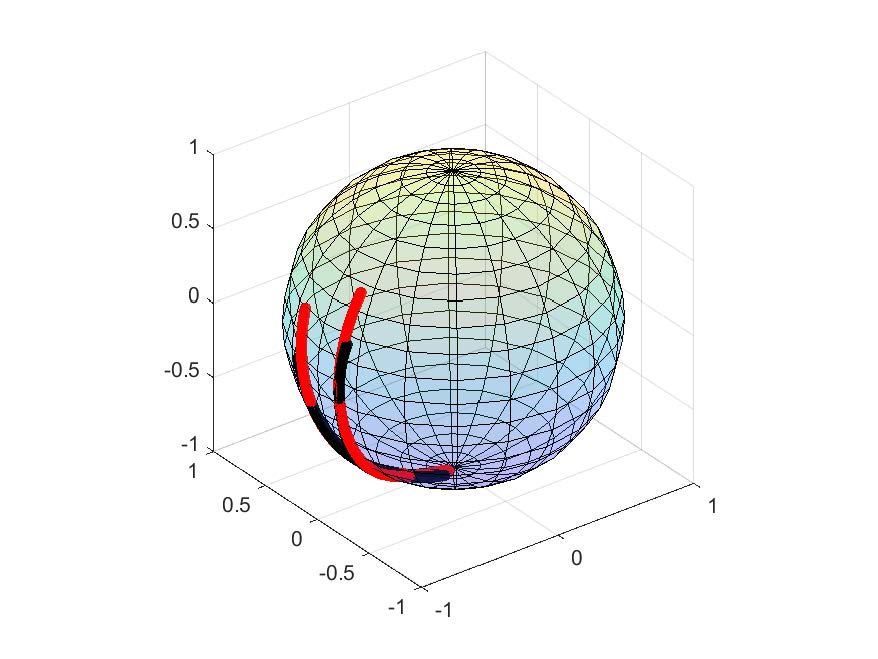}
\includegraphics[width=0.25\textwidth]{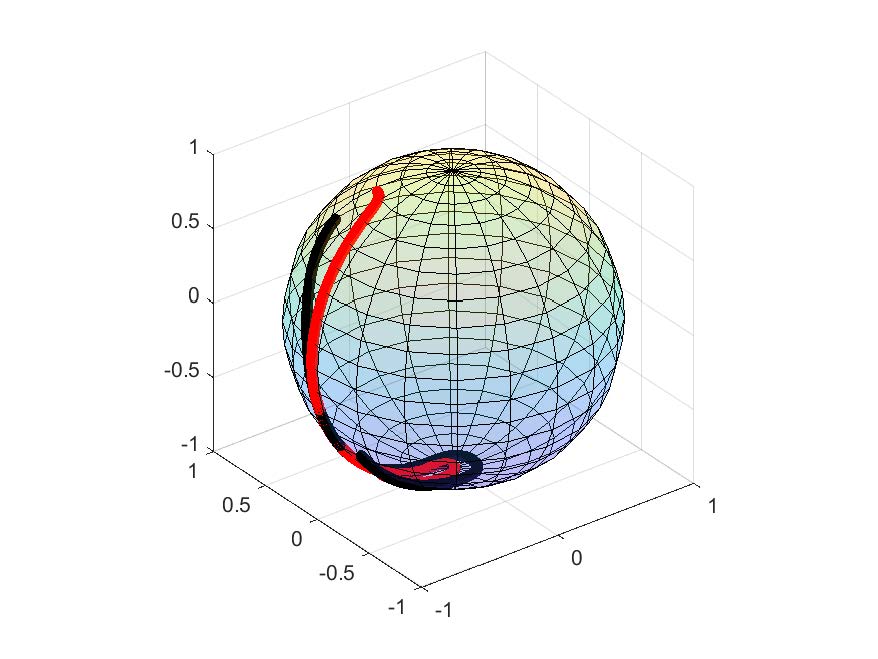}
\includegraphics[width=0.25\textwidth]{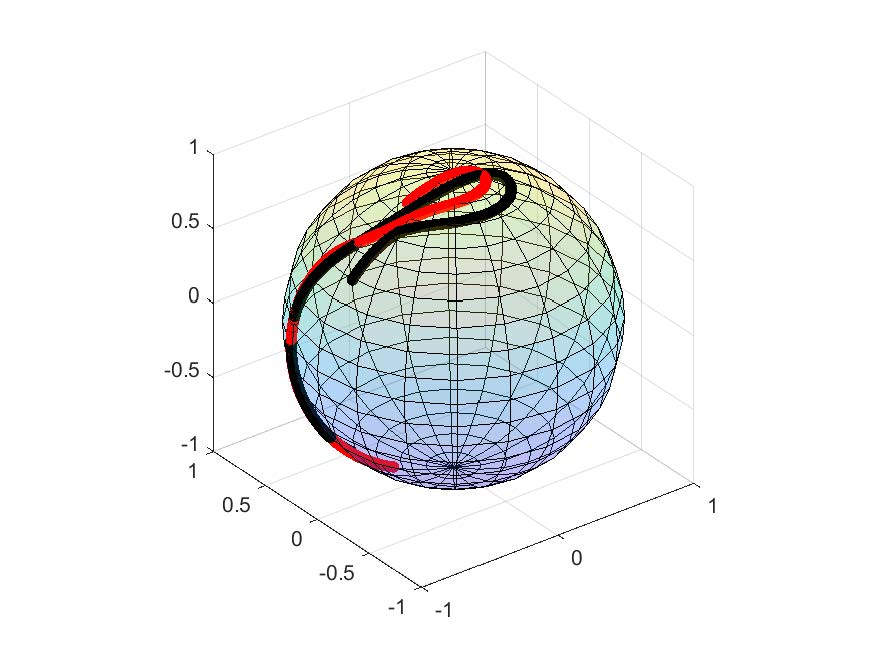}
\includegraphics[width=0.25\textwidth]{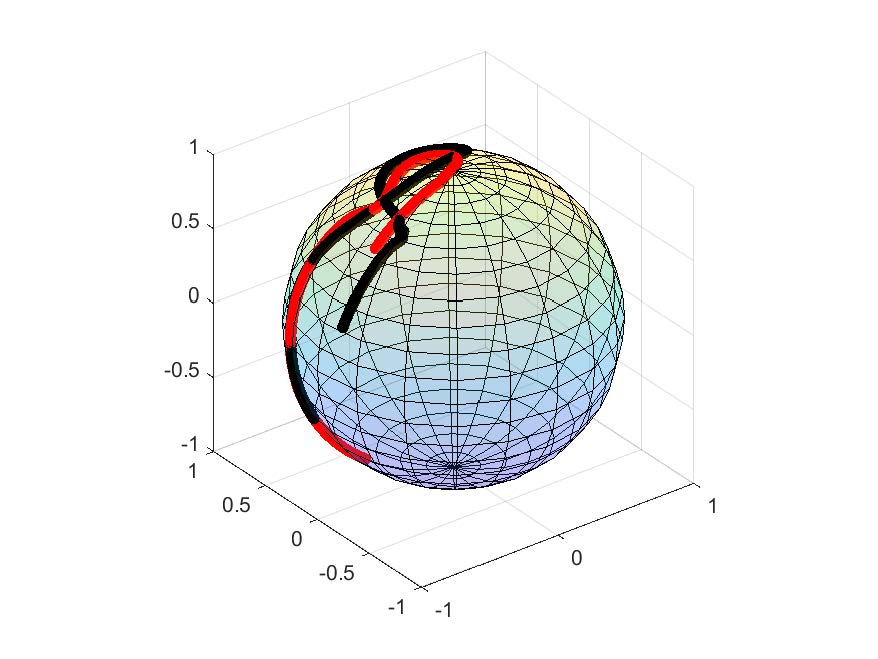}
\includegraphics[width=0.25\textwidth]{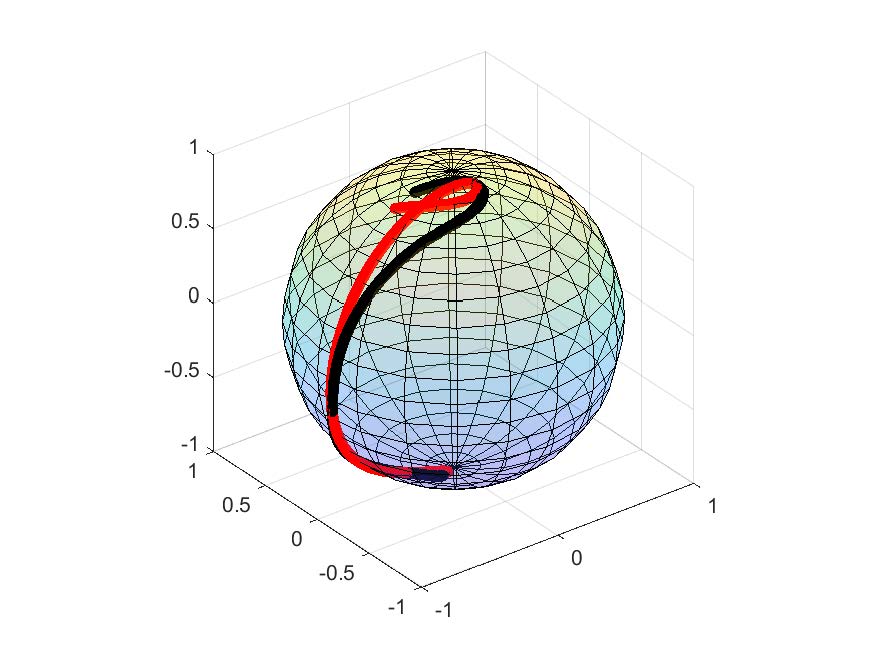}
\includegraphics[width=0.25\textwidth]{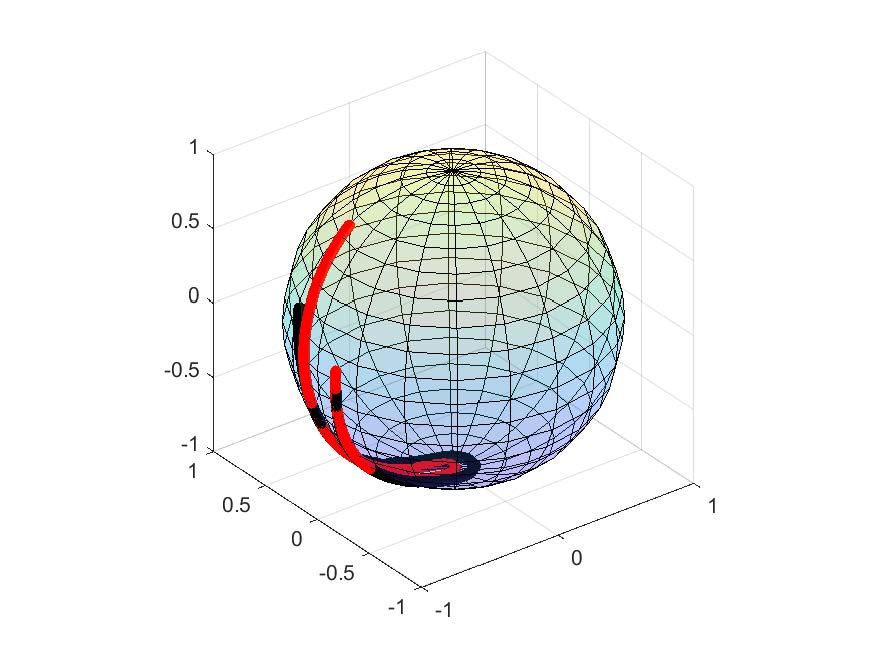}
\includegraphics[width=0.25\textwidth]{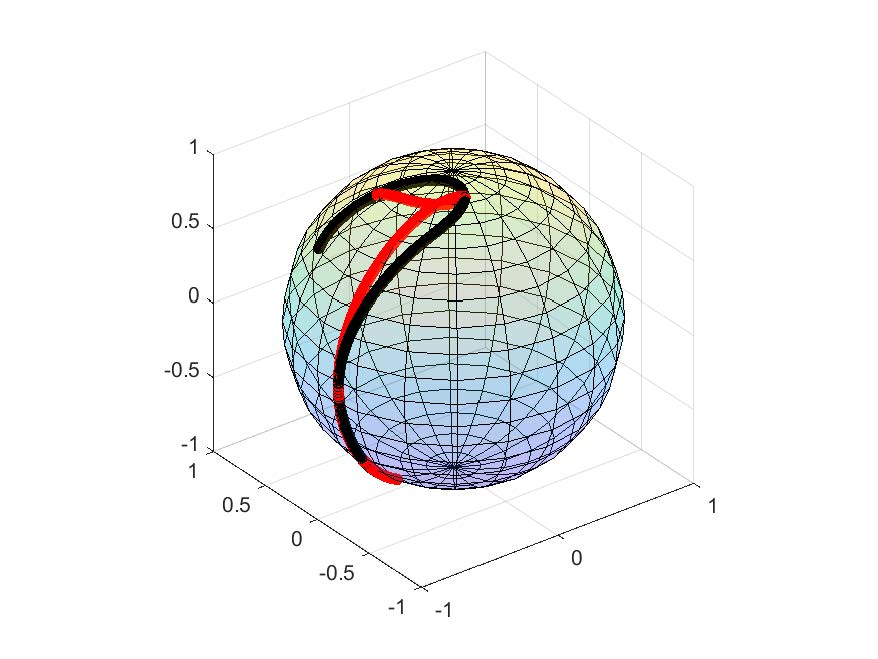}
\includegraphics[width=0.25\textwidth]{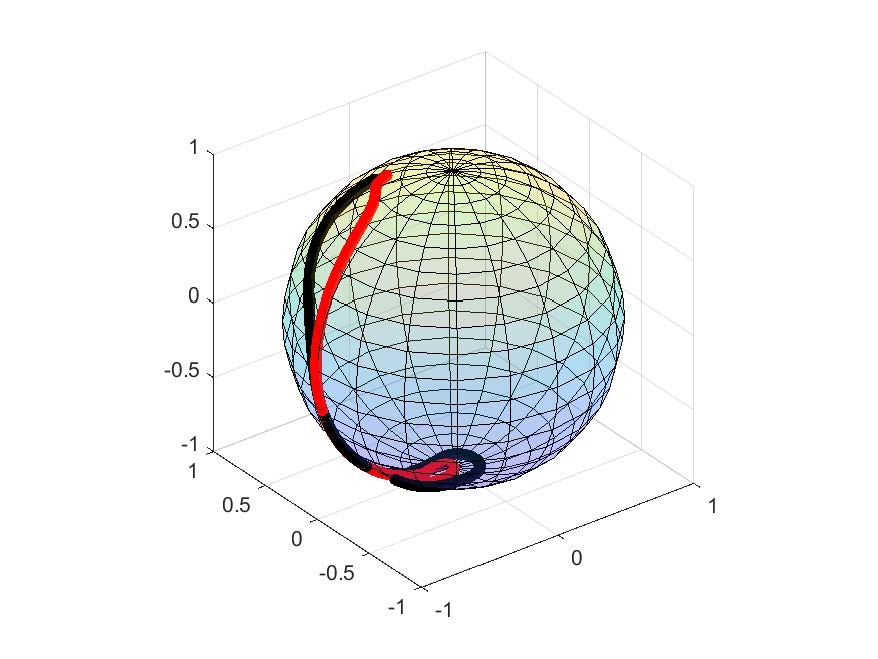}

\newpage
\section{Introduction}
The understanding of the interaction between the orientation of the solar magnetic field and the Earth's constitutes a central topic in modern economy. Measuring and managing the effects of the orientation  of the Earth's vector magnetic field allows to prevent   catastrophic losses  due to solar storms, affecting  the efficiency and  stability of  electrical and digital infrastructures supporting modern economy.
This prevention constitutes one of the main practical motivations.
 Specifically, the  orientation  of the Earth's vector magnetic field is based on the analysis of the temporal evolution of its components in a geocentric spherical coordinate system. We address the problem of functional prediction of these  components, defining our curve response,  from the observation of the  satellite NASA's MAGSAT spacecraft trajectories,  providing our  regressor sample curve values.

Note that the  Earth's magnetic field is subject to secular variation (slow internal changes), and external fluctuations, hence, a global model for this response often misses local temporal curve shifts. The time-varying kernel functional regression framework weights curve observations closer in time more heavily, thereby effectively capturing the evolution slope of the field. Particularly,   the extrinsic local linear Fr\'echet curve regression predictor proposed, involving the  Fr\'echet derivative of the projected regression operator,  allows  to estimate the velocity  changes of the  Earth's magnetic field coordinates at any given time.

The   intrinsic local linear Fr\'echet curve regression secondly proposed here is adapted to the geometry of the Earth, since the involved nonlinear weights are computed from the  geodesic distance. Thus, the intrinsic approach ensures Earth's magnetic field curve  predictions stay on the surface of the Earth.  Our choice of one of the two approaches for curve prediction depends on whether we are interested on reproducing velocity curve changes  at high resolution  levels via extrinsic local linear approach, or  curve geodesic  distance changes via the intrinsic local linear approach.

Nonparametric statistics   provides a  flexible framework for functional   regression, in particular, from data evaluated  in Riemannian manifolds, just to mention a few, see   \cite{Bhattacharya.12}, \cite{Marzio.14}, \cite{Khardani22}, \cite{KimPark13}, \cite{Linthomaszhu17} \cite{Patrangenaru.16}, \cite{Pelletier05}, \cite{Pelletier06}, \cite{Torres.21}, \cite{Torresetal2024},  \cite{Zhu09}.
Attention is given to infinite–dimensional local linear regression in a Fréchet framework. In the case of  response evaluated in a metric space and Euclidean regressors, we refer the reader to \cite{Petersen.19}. A weighted Fr\'echet mean approach is adopted to approximate the solution to the global and local linear Fr\'echet regression problem.   The global linear Fr\'echet regression problem has also been  addressed  in \cite{Torresetal2024}, when the response and regressor are curve processes evaluated in a compact Riemannian manifold.   In this response and regressor context, two important gaps are addressed regarding extrinsic and intrinsic local linear regression for functional responses and regressors evaluated on a compact Riemannian manifold. Specifically, given the local character of the exponential map, kernel curve regression can be implemented in an extrinsic way in the manifold on the time-varying tangent space. Furthermore, an intrinsic approach can also be adopted,  based on a weighted Fr\'echet mean, extending the results in \cite{Petersen.19}.  In both approaches, the optimality in the mean-square sense is proved.

The main difficulties arise in the derivation of the extrinsic local linear Fr\'echet curve predictor, since up to our knowledge, least-squares  local linear Fr\'echet curve regression has not been addressed yet, when response and regressor are evaluated in a separable Hilbert space. This problem is solved, and the conditions allowing the implementation of a local linear functional predictor are derived, in the  $L^{2}$-space of vector functions, with values in the time-varying tangent space on a compact Riemannian manifold.  The
 curve response and regressor Riemannian Functional Principal Component Analysis (RFPCA), introduced in  \cite{Dai.18},  plays a crucial role in the derivation of  our extrinsic local linear Fr\'echet curve  predictor  in manifolds.

 The finite sample performance of the  two presented regression  approaches is illustrated  in the  simulation study undertaken, and  in the real-data application, being compared with a Nadaraya-Watson type   Fr\'echet curve  predictor. The effect of time correlation, and the rules applied for fitting tuning parameters, including the bandwidth parameter, are also analyzed.

  The data set  considered is available in the  NASA's National Space Science Data Center  in the period 02/11/1979--06/05/1980. This sample curve set   provides the time--varying geocentric  latitude and longitude of the satellite NASA's MAGSAT  spacecraft, and the time--varying  spherical coordinates of the magnetic vector  field, as given in   \cite{Torresetal2024}. Data  have been   recorded every half second. The  NASA's MAGSAT spacecraft orbited the Earth  every 88 minutes during seven months at around 400 km altitude. The
$5$--fold cross validation methodology  is implemented to assess the predictive performance of the compared local  Fr\'echet functional   predictors  of the vector magnetic field through time.

The outline of the paper is summarized as follows. A projection least-squares  local linear Fr\'echet  functional regression predictor is derived in Section \ref{sec1} for response and regressor  evaluated in a separable Hilbert space. In Section \ref{geoass}, the geometrical background and assumptions are provided. The local linear  Fr\'echet functional regression methodology introduced in Section \ref{sec1} is applied in Section \ref{lefr0} to obtaining an extrinsic local linear Fr\'echet curve predictor in manifolds.
An alternative  intrinsic local linear  Fr\'echet curve predictor is proposed in  Section
 \ref{sec4},  based on a    weighted Fr\'echet mean approach. Its optimality is proved. The numerical results obtained in the simulation study   undertaken  are showed  in Section  \ref{simstud}.
 The real-data application  is addressed in Section \ref{rda}.

\section{Local linear  Fr\'echet regression for Hilbert-valued  response and regressors}
\label{sec1}
%%%%%%%%%%%%% Section 3 %%%%%%%%%%%%%%%%%%%
Local linear functional regression has previously been developed for responses evaluated in a metric space and Euclidean regressors, as well as for scalar responses and functional regressors evaluated in metric or semimetric spaces. Fundamental references in this topic are, for example,  \cite{Baillo.09, Barrientos2010, Benhenni.07, Benhenni.17, Benhenni.19, Berlinet.11}. This section addresses least-squares local linear regression for
 response and regressor  evaluated in a separable Hilbert space.

Let $\mathcal{H}$ be a separable Hilbert space. Denote by  $Y:(\Omega,\mathcal{A},P)\to$  $(\mathcal{H},\mathcal{B}(\mathcal{H}),P_{Y})$ and $X:(\Omega,\mathcal{A},P)\to (\mathcal{H},\mathcal{B}(\mathcal{H}),P_{X})$  two measurable mappings respectively defining the functional response and regressor, obeying the following equation
$$Y= m(X)+\varepsilon,$$
\noindent where $P[\varepsilon \in \mathcal{H}]=1,$ $E[\varepsilon /X]=0,$  and $E[\|\varepsilon\|_{\mathcal{H}}^{2}/X]=\sigma_{\varepsilon}^{2}<\infty .$ Here, $P_{Y}$ and $P_{X}$ denote the probability measures induced by the response $Y,$ and regressor $X,$ respectively.   $\mathcal{B}(\mathcal{H})$ is the Borel $\sigma$-- algebra generated by the open sets of $\mathcal{H}.$
  In what follows we denote the regression operator  by $m:\mathcal{H}\to \mathcal{H},$ with $m(x_{0})=E[Y/X=x_{0}],$ $x_{0}\in \mathcal{H}.$
  Assume that $m$ admits a Fr\'echet derivative given by a bounded linear operator $\mathcal{A}$ such that $$\lim_{\|h\|_{\mathcal{H}}\to 0}\frac{\|m(x+h)-m(x)-\mathcal{A}(h)\|_{\mathcal{H}}}{\|h\|_{\mathcal{H}}}=0,\quad \forall x\in \mathcal{H}.$$

Thus, in the norm $\|\cdot\|_{\mathcal{H}}$ of $\mathcal{H},$   for  $x\in \mathcal{H},$ with $\|x-x_{0}\|_{\mathcal{H}}<\epsilon,$ one can consider the local linear approximation
\begin{equation}m(x)\underset{\mathcal{H}}{\sim } m(x_{0})+\mathcal{A}(x-x_{0}),\quad x_{0}\in \mathcal{H},\label{slor}
\end{equation}
\noindent in terms of the slope operator $\mathcal{A},$ providing the Fr\'echet derivative of $m.$

\medskip  The functional value $m(x_{0})$ and operator $\mathcal{A}$ can be estimated by solving the following local linear minimization problem:

\begin{eqnarray}&&\hspace*{-1cm}
(\widehat{\beta_{0}},\widehat{\beta_{1}})=\mbox{arg min}_{\beta_{0}\in \mathcal{H},\beta_{1}\in \mathcal{L}(\mathcal{H})}E\left[
\|Y-\beta_{0}-\beta_{1}(X-x_{0})\|_{\mathcal{H}}^{2}
K_{B_{n}}\left(\|X-x_{0}\|_{\mathcal{H}}\right)\right],\nonumber\\
\label{lossef}
\end{eqnarray}
\noindent where, from equation (\ref{slor}),\begin{eqnarray}&&E\left[
\|Y-\beta_{0}-\beta_{1}(X-x_{0})\|_{\mathcal{H}}^{2}K_{B_{n}}\left(\|X-x_{0}\|_{\mathcal{H}}\right)\right]\nonumber\\
&&=\int_{\mathcal{H}\times \mathcal{H} }K_{B_{n}}\left(\|X-x_{0}\|_{\mathcal{H}}\right)\nonumber\\
&&\hspace*{1.5cm} \times\left\langle Y-\beta_{0}-\beta_{1}(X-x_{0}),Y-\beta_{0}-\beta_{1}(X-x_{0})\right\rangle_{\mathcal{H}}P(dX,dY)\label{exp}\\
&&=\int_{\mathcal{H}\times \mathcal{H} }K_{B_{n}}\left(\|X-x_{0}\|_{\mathcal{H}}\right)\nonumber\\
&&\hspace*{1.5cm} \times
\left\langle Y-m(x_{0})-\mathcal{A}(X-x_{0}),Y-m(x_{0})-\mathcal{A}(X-x_{0})\right\rangle_{\mathcal{H}}P(dX,dY).\nonumber\\
\label{PIll}
\end{eqnarray}

\noindent Here, $K$ is a probability density and $B_{n}$ is a bandwidth parameter indexed by the sample size $n,$  with $B_{n}\to 0,$ as $n\to \infty.$
 We assume that  the  Fr\'echet derivative operator  $\mathcal{A}$ is a self--adjoint  compact operator on $\mathcal{H},$ admitting the following pure point spectral diagonal expansion:
\begin{equation}
\mathcal{A}=\sum_{k\geq 1}\lambda_{k}(\mathcal{A})\phi_{k}\otimes \phi_{k},
\label{eqopfrechet}
\end{equation}
\noindent where $\{\phi_{k},\ k\geq 1\}$ and  $\{\lambda_{k}(\mathcal{A}),\ k\geq 1\}$ respectively denote the eigenfunctions and eigenvalues of $\mathcal{A}.$
\begin{remark}The existence of a bounded Fr\'echet derivative operator means that the local behavior of the regression operator is dominated by a bounded linear operator.  A typical example is the exponential operator which admits a Fr\'echet derivative,  capturing local slope changes. The additional condition of compactness of the Fr\'echet derivative  does not suppose a  strong restriction, which is understood as  the local  linear approximation of the regression operator  ignores (or dampens) variations at high frequency dimensions of the Hilbert space.
The symmetry assumption can be ignored, since,  under compactness, the singular value decomposition  holds beyond the self-adjoint condition.
\end{remark}

  Applying  Parseval identity   in equation (\ref{PIll}) in terms of the eigenfunctions $\left\{ \phi_{k}, \ k\geq 1\right\}$ of $\mathcal{A}$ in (\ref{eqopfrechet}), the following equalities hold
\begin{eqnarray}&&L(\beta_{0},\beta_{1})=E\left[
\|Y-\beta_{0}-\beta_{1}(X-x_{0})\|_{\mathcal{H}}^{2}K_{B_{n}}\left(\|X-x_{0}\|_{\mathcal{H}}\right)\right]\nonumber\\
&&=L(m(x_{0}),\mathcal{A})=\int_{\mathcal{H}\times \mathcal{H} }K_{B_{n}}\left(\|X-x_{0}\|_{\mathcal{H}}\right)
 \nonumber\\ &&\hspace*{1cm}\times \sum_{k\geq 1}\left[\left\langle Y,\phi_{k}\right\rangle_{\mathcal{H}}-\left\langle m(x_{0}), \phi_{k}\right\rangle_{\mathcal{H}}-
\left\langle\mathcal{A}(X-x_{0}),\phi_{k}\right\rangle_{\mathcal{H}}\right]^{2}P(dX,dY)\nonumber\\
&&
=\int_{\mathcal{H}\times \mathcal{H} }K_{B_{n}}\left(\|X-x_{0}\|_{\mathcal{H}}\right)\nonumber\\ &&\hspace*{1cm}
\times  \sum_{k\geq 1}\left[\left\langle Y,\phi_{k}\right\rangle_{\mathcal{H}}-\left\langle m(x_{0}), \phi_{k}\right\rangle_{\mathcal{H}}-\lambda_{k}(\mathcal{A})
\left\langle(X-x_{0}),\phi_{k}\right\rangle_{\mathcal{H}}\right]^{2}P(dX,dY).\nonumber\\
\label{exp2}
\end{eqnarray}

 By differentiation in equation (\ref{exp2}) with respect to $\beta_{0}=m(x_{0}),$ and with respect to $\beta_{1}=\mathcal{A},$
   and equalizing to zero, applying again Parseval identity, we obtain
   \begin{eqnarray} \hspace*{-1cm}
0&=&\frac{\partial L}{\partial \beta_{0}}(\beta_{0},\beta_{1})= \frac{\partial L}{\partial m(x_{0})}(m(x_{0}),\mathcal{A})
\nonumber\\ &=&\int_{\mathcal{H}\times \mathcal{H} }K_{B_{n}}\left(\|X-x_{0}\|_{\mathcal{H}}\right)
\nonumber\\ && \hspace*{1cm}\times \left[2\sum_{k\geq 1}\left[\left\langle m(x_{0}), \phi_{k}\right\rangle_{\mathcal{H}}+
\left\langle\mathcal{A}(X-x_{0}),\phi_{k}\right\rangle_{\mathcal{H}}-\left\langle Y,\phi_{k}\right\rangle_{\mathcal{H}}\right]\phi_{k}\right]P(dX,dY),
\nonumber\\
0&=&\frac{\partial L}{\partial \beta_{1}}(\beta_{0},\beta_{1})=\frac{\partial L}{\partial \mathcal{A}}(m(x_{0}),\mathcal{A})
\nonumber\\
 &&\hspace*{-0.5cm}=2\int_{\mathcal{H}\times \mathcal{H} }K_{B_{n}}\left(\|X-x_{0}\|_{\mathcal{H}}\right)\nonumber\\
&&\hspace*{1cm}\times \left[\sum_{k\geq 1}\lambda_{k}(\mathcal{A})\left\langle (X-x_{0}),\phi_{k}\right\rangle_{\mathcal{H}}\left\langle(X-x_{0}),\phi_{k}\right\rangle_{\mathcal{H}}\right.\nonumber\\
&&\hspace*{4cm}\left.-\left\langle Y-m(x_{0}),\phi_{k}\right\rangle_{\mathcal{H}}\left\langle (X-x_{0}),\phi_{k}\right\rangle_{\mathcal{H}}\right]P(dX,dY).
 \label{exp2b}
\end{eqnarray}
\noindent That is, from  (\ref{exp2b}), we  have
\begin{eqnarray} && \int_{\mathcal{H}\times \mathcal{H} }K_{B_{n}}\left(\left\|X-x_{0}\right\|_{\mathcal{H}}\right)\left\{m(x_{0})-Y+\mathcal{A}(X-x_{0})\right\}P(dX,dY)=0\label{elf0},\\
&&
\int_{\mathcal{H}\times \mathcal{H} }K_{B_{n}}\left(\|X-x_{0}\|_{\mathcal{H}}\right)
 \left[\left\langle \mathcal{A}(X-x_{0}),X-x_{0}\right\rangle_{\mathcal{H}}\right.\nonumber\\
&&\hspace*{4cm}\left.-\left\langle Y-m(x_{0}),X-x_{0}\right\rangle_{\mathcal{H}}\right]P(dX,dY)=0.\label{elf}
\end{eqnarray}

Furthermore, from  (\ref{elf0}),
\begin{eqnarray} &&
m(x_{0})\mu_{0}-r_{0}+\mathcal{A}(\mu_{1})=0,\label{effn1}\\
&&m(x_{0})=\frac{r_{0}-\mathcal{A}(\mu_{1})}{\mu_{0}},
\label{effn2}
\end{eqnarray}
\noindent where  \begin{eqnarray} \mu_{0}&=&\int_{\mathcal{H}}K_{B_{n}}\left(\left\|X-x_{0}\right\|_{\mathcal{H}}\right)P(dX),\nonumber\\ \mu_{1}&=&\int_{\mathcal{H}}K_{B_{n}}\left(\left\|X-x_{0}\right\|_{\mathcal{H}}\right)\left(X-x_{0}\right)P(dX),\nonumber\\
r_{0}&=&E\left[K_{B_{n}}\left(\left\|X-x_{0}\right\|_{\mathcal{H}}\right)Y\right].\label{eqmarginal}
\end{eqnarray}

Note that, from equation  (\ref{s2302h}) in  Appendix  \ref{app0}, \begin{eqnarray}&&\int_{\mathcal{H}\times \mathcal{H}}S^{(k)}(x,x_{0},B_{n})P(dx,dy)=
 \int_{\mathcal{H}}S^{(k)}(x,x_{0},B_{n})P(dx)=\frac{\sigma_{0}^{2}(k)}{\sigma_{0}^{2}(k)}=1,\nonumber\\\label{s2302hb}\end{eqnarray}
 \noindent with $P(dx)=\int_{\mathcal{H}}P(dx,dy).$   Hence, equation  (\ref{s2302g})  in  Appendix  \ref{app0} can be rewritten as
 \begin{equation}
  \widehat{m(x_{0})}(\phi_{k})=\mbox{arg min}_{h\in \mathcal{H}}E\left[S^{(k)}(X,x_{0},B_{n})[Y(\phi_{k})-h(\phi_{k})]^{2}\right],\ k\geq 1,
 \label{s2302i}
 \end{equation}
\noindent where, as given  in (\ref{s2302h}),  $S^{(k)}$ is defined by
 \begin{eqnarray}&& S^{(k)}(X,x_{0},B_{n})=\frac{1}{\sigma_{0}^{2}(k)}\left[K_{B_{n}}\left(\left\|X-x_{0}\right\|_{\mathcal{H}}\right)\left[\mu_{2}^{(k)}
 -\mu_{1}^{(k)}(X-x_{0})(\phi_{k})\right]\right],\nonumber\\
 \label{s2302h2}\end{eqnarray}
 \noindent with $\mu_{j}^{(k)}$ being introduced in  (\ref{eqnirmfunct}) in  Appendix  \ref{app0}, for $k\geq 1,$ and $j\geq 1,$ and
$\sigma_{0}^{2}(k)=\mu_{2}^{(k)}\mu_{0}-[\mu_{1}^{(k)}]^{2},$ $k\geq 1.$

 Thus, our local linear functional  regression predictor at the functional argument $x_{0}\in \mathcal{H},$ is given by
 \begin{equation}
 \widehat{Y}(x_{0})=\widehat{m(x_{0})}=\sum_{k\geq 1} \widehat{m(x_{0})(\phi_{k})}\phi_{k}.\label{s2302j}
 \end{equation}

\section{Background and assumptions}
\label{geoass}
Let $\mathcal{M}$ be a smooth manifold with topological dimension $d$ in an
Euclidean space $\mathbb{R}^{d_{0}},$ $d\leq d_{0}$. Denote  by $\{ \mathcal{T}_{p}\mathcal{M},\ p\in \mathcal{M} \}$  the tangent spaces at the points of $\mathcal{M}.$ A
\emph{Riemannian metric} on $\mathcal{M}$ is a family of inner products $\mathcal{G}(p):\mathcal{T}_{p}\mathcal{M}\times \mathcal{T}_{p}\mathcal{M}\longrightarrow \mathbb{R}$ that smoothly varies over $p\in \mathcal{M}.$ Hence, $\left(\mathcal{M},\mathcal{G}\right)$ endowed with this Riemann metric $\mathcal{G}$ is a Riemann manifold. Specifically, the metric on $\mathcal{M}$ induced by $\mathcal{G}$ is the geodesic distance $d_{\mathcal{M}}.$  A geodesic is a locally length minimizing
curve.
The \emph{exponential map} at $p\in \mathcal{M}$  applied to $v\in \mathcal{T}_{p}\mathcal{M}$ is given by   $\exp_{p}(v)=\gamma_{v}(1),$ where  $v\in \mathcal{T}_{p}\mathcal{M}$ is a tangent vector  at $p,$  and   $\gamma_{v}=\left\{\exp_{p}(tv), \ t\in [0,1]\right\}$ is the unique geodesic with initial location $\gamma_{v}(0)=p,$ and velocity $\gamma_{v}^{\prime }(0)=v.$

The inverse of the exponential map is called \emph{the logarithm map}, and is denoted by $\log_{p},$ $p\in \mathcal{M}.$
The injectivity radius  at $p\in \mathcal{M},$ denoted by $\mbox{inj}_{p},$ is
    the radius of the largest ball centered at the origin in the tangent space $\mathcal{T}_{p}\mathcal{M}$ on which the exponential map $\exp_{p}$ is a diffeomorphism  onto its image. If $(\mathcal{M}, d_{\mathcal{M}})$ is a complete metric space, then $\exp_{p}$ is defined on the entire tangent space, and  $\exp_{p}$ is  a diffeomorphism in a neighborhood of
the origin of $\mathcal{T}_{p}\mathcal{M}.$

Consider the  space  $$\left(\mathcal{C}_{\mathcal{M}}(\mathcal{T}),d_{\mathcal{C}_{\mathcal{M}}(\mathbb{T})}\right) =\left\{x:\mathbb{T}\to \mathcal{M}:\ x\in\mathcal{C}(\mathbb{T})\right\},$$
\noindent constituted by  $\mathcal{M}$-valued continuous functions on a compact interval $\mathbb{T} $ with the supremum geodesic distance $$d_{\mathcal{C}_{\mathcal{M}}(\mathbb{T})}(x(\cdot),y(\cdot))=\sup_{t\in \mathbb{T}}d_{\mathcal{M}}\left( x(t),y(t)\right), \quad \forall x(t),y(t)\in \left(\mathcal{C}_{\mathcal{M}}(\mathbb{T}), d_{\mathcal{C}_{\mathcal{M}}(\mathbb{T})}\right).$$

Denote, as before,  by $(\Omega ,\mathcal{A},P)$ the basic probability space.  Let $Z=\{Z_{s},\ s\in \mathbb{Z}\}$   be a family  of random elements  in  $\left(\mathcal{C}_{\mathcal{M}}(\mathbb{T}),d_{\mathcal{C}_{\mathcal{M}}(\mathbb{T})}\right) $   indexed by $\mathbb{Z}.$
    Specifically, $Z:\mathbb{Z}\times (\Omega,\mathcal{A},\mathcal{P})\to \mathcal{C}_{\mathcal{M}}(\mathbb{T})$ defines a measurable mapping,    and\linebreak 
 $\mathcal{P}\left(\xi \in \Omega;\ Z_{s}(\cdot ,\xi )\in \left(\mathcal{C}_{\mathcal{M}}(\mathbb{T}),d_{\mathcal{C}_{\mathcal{M}}(\mathbb{T})}\right)\right)=1,$ for every $s\in  \mathbb{Z}.$
Here, $Z_{s}(t)$  denotes the pointwise value (i.e., the one-dimensional time projection) at $t\in \mathbb{T}$ of the   random curve $Z_{s}$ in $\mathcal{M},$ for each  $s\in \mathbb{Z}.$

For each  $s\in \mathbb{Z},$ the
intrinsic Fr\'echet functional mean $\mu_{Z_{s},\mathcal{M}} $    is given by
\begin{eqnarray}
 \mu_{Z_{s},\mathcal{M}}(t)&=&\mbox{arg min}_{p\in \mathcal{M}}E\left([d_{\mathcal{M}}\left(Z_{s}(t),p\right)]^{2}\right)\nonumber\\
 &=&\mbox{arg min}_{p\in \mathcal{M}}\int [d_{\mathcal{M}}\left(z_{s}(t),p\right)]^{2}dP_{Z_{s}(t)}(z_{s}(t)),\ t\in \mathbb{T},\nonumber\\
 \label{FM}
\end{eqnarray}
\noindent where  $dP_{Z_{s}(t)}$ denotes the  probability measure induced by  the one-dimensional time projection $Z_{s}(t)$   of the random curve $Z_{s}\subset \mathcal{M}$ at time $s\in \mathbb{Z}.$
Thus,  for each $s\in \mathbb{Z},$ $\mu_{Z_{s},\mathcal{M}}$ is the curve in $\mathcal{M}$  providing the best  pointwise approximation of $Z_{s}$  in the mean quadratic geodesic distance sense. Since $Z_{s}\in  \left(\mathcal{C}_{\mathcal{M}}(\mathbb{T}),d_{\mathcal{C}_{\mathcal{M}}(\mathbb{T})}\right)$ almost surely (a.s.), $\mu_{Z_{s},\mathcal{M}}(t)$ is also  continuous, for every $s\in \mathbb{Z}.$ Hence,  the following equivalent  definition of  $\mu_{Z_{s},\mathcal{M}}(\cdot )$ can be considered:

\begin{eqnarray}\mu_{Z_{s},\mathcal{M}}(\cdot )&=&\mbox{arg min}_{ x(\cdot )\in \mathcal{C}_{\mathcal{M}}(\mathbb{T})}
E\left(\int_{\mathbb{T}}[d_{\mathcal{M}}\left(Z_{s}(t), x(t)\right)]^{2}dt\right)\nonumber\\ &=&\mbox{arg min}_{ x(\cdot )\in \mathcal{C}_{\mathcal{M}}(\mathbb{T})}\int_{\mathcal{C}_{\mathcal{M}}(\mathbb{T})}\int_{\mathbb{T}}[d_{\mathcal{M}}\left(z_{s}(t), x(t)\right)]^{2}dt
dP_{Z_{s}}(z_{s}),
\label{FFM}
\end{eqnarray}
\noindent where $dP_{Z_{s}}$ denotes the infinite--dimensional probability measure induced by $Z_{s}$ for every  $s\in \mathbb{Z}.$

%%%%%%%%%%%%% Section 4 %%%%%%%%%%%%%%%%%%%
\subsection{Assumptions}
\label{sec3.2}

This section provides the assumptions on sample path regularity, and the  required geometrical and probabilistic conditions,   ensuring existence and uniqueness of the proposed extrinsic and intrinsic local linear Fr\'echet functional  predictors. Specifically, the following geometrical conditions are assumed:

\begin{itemize}
\item[(i)] $\mathcal{M}$ is a $d$--dimensional compact  and connected Riemannian submanifold  of a Euclidean space  $\mathbb{R}^{d_{0}},$ $d\leq d_{0},$ with geodesic distance $d_{\mathcal{M}}$  induced by
the Euclidean metric.
\item[(ii)] The  sectional curvature of manifold  $\mathcal{M}$ is bounded, positive, and of smooth variation.
\end{itemize}
\begin{remark}
\label{rem1}
The exponential map is defined on the entire tangent space under (i), see, e.g.,  \cite{Dai.18}. Under (ii),  the geodesic distance between two points in the manifold   is upper bounded by  the  Euclidean distance of their corresponding tangent vectors, see Assumption A2,  and Proposition 1 in \cite{Dai.18}.
\end{remark}
\begin{remark}
\label{rem12}
Although our main practical motivation involves the sphere satisfying conditions (i) and (ii),
 the   existence and uniqueness of the   $L^{p},$ $1\leq p\leq \infty,$ center of mass  (minimizer of the $L^{p}$--energy function) of a probability measure  holds for more general families of complete connected  Riemannian manifolds, see Theorem 1 in
     \cite{LeBarden17}. The effect of the curvature and  topology of the manifold has  been analyzed, for example,  in \cite{Afsari13}, see also the overview
   in Section 1.1  in  \cite{Afsari11}.

\end{remark}

Let $Y=\{ Y_{s},\ s\in \mathbb{Z}\}$ and $X=\{ X_{s},\ s\in \mathbb{Z}\}$ be the response $Y$ and the regressor $X$ curve processes evaluated   in a Riemannian  manifold $\mathcal{M}.$ The following  conditions are assumed on the bivariate curve process $(X,Y):$
\begin{itemize}
\item[(iii)]
For every time $s_{i}\in \mathbb{Z},$ the  random Lipschitz constants $L_{Y}(Y_{s_{i}})$ and  $L_{X}(X_{s_{i}})$ of $Y_{s_{i}}$ and $X_{s_{i}}$
    are almost surely (a.s.) finite. The Lipschitz constants
     $L(\mu_{Y_{s_{i}},\mathcal{M}})$ and  $L(\mu_{X_{s_{i}},\mathcal{M}})$  of the Fr\'echet means  $\mu_{Y_{s_{i}},\mathcal{M}}$ and $\mu_{X_{s_{i}},\mathcal{M}}$   are also finite.  Particularly,  assume that $E\left[\left(L_{X}(X_{s_{i}})\right)^{2}\right]<\infty,$  and $E\left[\left(L_{Y}(Y_{s_{i}})\right)^{2}\right]<\infty,$  for any $s_{i}\in \mathbb{Z}.$
     Note that,  for any curve $z(\cdot),$  $L(z)=\sup_{t\neq s}\frac{d_{\mathcal{M}}(z(t),z(s))}{\left|t-s\right|}.$
\item[(iv)] The $\mathcal{M}$--valued bivariate curve process $\{(Y_{s}, X_{s}),\ s\in \mathbb{Z}\}$ is strictly stationary.
We  denote by   $\mathcal{Y}_{\mathcal{C}_{\mathcal{M}}(\mathbb{T})}\subseteq \left(\mathcal{C}_{ \mathcal{M}}(\mathbb{T}), d_{\mathcal{C}_{\mathcal{M}}(\mathbb{T})}\right),$ and   $\mathcal{X}_{\mathcal{C}_{\mathcal{M}}(\mathbb{T})}\subseteq $ $\left(\mathcal{C}_{ \mathcal{M}}(\mathbb{T}), d_{\mathcal{C}_{\mathcal{M}}(\mathbb{T})}\right)$ the respective  state spaces  of the marginals of $X$ and $Y$.
Assume also that  $\{ \log_{\mu_{X_{0},\mathcal{M}}(t)}\left(X_{s}(t)\right),\ s\in \mathbb{Z}\}$ is mean--square ergodic in the first moment  in the norm of the ambient Hilbert space $\mathbb{H},$ and in the second--order moments
in the norm of the space $\mathcal{S}(\mathbb{H})$ of Hilbert--Schmidt operators on $\mathbb{H}.$ Here, $\mathbb{H}$ denotes  the ambient  Hilbert space of  vector functions with values      in the time--varying  tangent space,  given by
   \begin{equation}\mathbb{H}=\left\{h=(h_{1}, \dots, h_{d_{0}})^{T}:\ \mathbb{T}\to \mathbb{R}^{d_{0}}:\ \int_{\mathbb{T}}h(t)^{T}h(t)dt<\infty\right\}.\label{ahs}
 \end{equation}

\item[(v)]  Curve processes
$X=\{ X_{s},\ s\in \mathbb{Z}\}$ and  $Y=\{ Y_{s},\ s\in \mathbb{Z}\}$
       have the same  Fr\'echet  functional mean. The supports of their marginal probability measures $dP_{X_{0}}(\cdot )$ and $dP_{Y_{0}}(\cdot )$ are included in the ball of the space $\left(\mathcal{C}_{ \mathcal{M}}(\mathbb{T}), d_{\mathcal{C}_{\mathcal{M}}(\mathbb{T})}\right),$ centered at the Fr\'echet functional  mean $\mu_{X_{0},\mathcal{M}}=\mu_{Y_{0},\mathcal{M}}$ with radius
$R=\inf_{t\in \mathbb{T}}\ \mbox{inj}_{\mu_{X_{0},\mathcal{M}}(t)}.$
  Here, $\mbox{inj}_{\mu_{X_{0},\mathcal{M}}(t)}$  denotes the injectivity radius of the exponential map whose  origin  is at $\mu_{X_{0},\mathcal{M}}(t)$ for each  $t\in \mathbb{T}.$
\end{itemize}

\begin{remark}
\label{rem1aa}
Concentration of curve data around the Fr\'echet curve mean in condition (v) ensures existence and uniqueness of Fr\'echet means.
 Weaker versions of condition (v) can be considered to ensure the existence of the $L^{p}$-center of mass of a probability measure. In particular, the existence and uniqueness of such a center  still hold when the support of the underlying probability measure is  contained   into  a ball, whose radius is  upper  bounded by a function of
 $p,$ the injectivity radius of the manifold,  and an upper bound on the manifold sectional curvatures, see Theorem 1 in
     \cite{LeBarden17}. Under this more general setting, in  \cite{Afsari13},    the convergence of a  constant step-size gradient descent algorithm is investigated,  for computing the $L^{p}$-center of mass of a probability measure.

 Under strictly stationarity condition in (iv),    assumption on the existence of a common marginal Fr\'echet curve mean of  the response and regressor can be weakened.  A sufficiently small supremum geodesic distance  between the Fr\'echet curve means of the regressor and response can be considered, such that  condition  (v) holds  in terms of the intersection of the  balls of radius $R_{1}=\inf_{t\in \mathbb{T}}\ \mbox{inj}_{\mu_{X_{0},\mathcal{M}}(t)}$ and $R_{2}=\inf_{t\in \mathbb{T}}\ \mbox{inj}_{\mu_{Y_{0},\mathcal{M}}(t)},$   centered at $\mu_{X_{0},\mathcal{M}}(\cdot )$ and $\mu_{Y_{0},\mathcal{M}}(\cdot ),$ respectively.

\end{remark}

\section{Extrinsic local linear  Fr\'echet   curve regression  in manifolds}
\label{lefr0}
We consider the results obtained in Section \ref{sec1} in the ambient Hilbert space  $\mathcal{H}=\mathbb{H}=L^{2}_{\mathcal{M}}(\mathbb{T})$
introduced in equation (\ref{ahs}),
  equipped with the inner product $\left\langle h,f\right\rangle_{\mathbb{H}}=\int_{\mathbb{T}}h(t)^{T}f(t)dt,$ and norm   $\|h\|_{\mathbb{H}}=\left[\left\langle h,h\right\rangle_{\mathbb{H}}\right]^{1/2},$ for every  $h, f\in \mathbb{H}.$

Under mild conditions the mean function of the log--mapped data in $\mathbb{H}$ is  zero when the logarithm map has origin at
the functional Fr\'echet mean $\mu_{Y_{0},\mathcal{M}}=\mu_{X_{0},\mathcal{M}}$ under (v), see Theorem 2.1 of \cite{Bhattacharya03}.
The trace autocovariance matrix operators $\mathcal{R}^{LY}_{0}$ and  $\mathcal{R}^{LX}_{0}$ of the log--mapped response  $\left\{ \log_{\mu_{Y_{0},\mathcal{M}}(t)}\left(Y_{s}(t)\right),\ s\in \mathbb{Z}\right\}$ and regressor $\left\{ \log_{\mu_{X_{0},\mathcal{M}}(t)}\left(X_{s}(t)\right),\ s\in \mathbb{Z}\right\}$ processes are respectively
defined as
\begin{eqnarray} &&\mathcal{R}^{LY}_{0}= E\left[\log_{\mu_{Y_{0},\mathcal{M}}(\cdot )}\left(Y_{0}(\cdot)\right)\otimes [\log_{\mu_{Y_{0},\mathcal{M}}(\cdot)}\left(Y_{0}(\cdot)\right)]^{T}\right],\nonumber\\
&&\mathcal{R}^{LX}_{0}= E\left[\log_{\mu_{X_{0},\mathcal{M}}(\cdot )}\left(X_{0}(\cdot)\right)\otimes [\log_{\mu_{X_{0},\mathcal{M}}(\cdot)}\left(X_{0}(\cdot)\right)]^{T}\right].\nonumber
\end{eqnarray}
\noindent In the next result we assume that
\begin{eqnarray}
\mathcal{R}^{LY}_{0}&=&\sum_{k\geq 1}\lambda_{k}(Y)\phi_{k}\otimes \phi_{k}^{T},\nonumber\\
\mathcal{R}^{LX}_{0}&=&\sum_{k\geq 1}\lambda_{k}(X)\phi_{k}\otimes \phi_{k}^{T},
\label{eqdiastgs}
\end{eqnarray}
\noindent  in the norm of  $\mathbb{H}\otimes \mathbb{H},$ where  $\mathcal{R}^{LY}_{0}(\phi_{k})= \lambda_{k}(Y) \phi_{k},$ and
$\mathcal{R}^{LX}_{0}(\phi_{k})= \lambda_{k}(X) \phi_{k},$
with $\lambda_{k}(Y)$ and $\lambda_{k}(X)$ respectively being the  eigenvalues of $\mathcal{R}^{LY}_{0}$ and $\mathcal{R}^{LX}_{0}$
    associated with the vector eigenfunction $\phi_{k}:\mathbb{T}\to \mathbb{R}^{d_{0}},$ for every  $k\geq 1.$

\begin{proposition}
\label{pr1bb}
Assume  that  $\{(Y_{s}, X_{s}),\ s\in \mathbb{Z}\}$ is strictly stationary, satisfying  conditions  (i)-(iii), and (v), and that
the following RFPC decompositions
  \begin{eqnarray}
\log_{\mu_{Y_{0},\mathcal{M}}(t)}\left(Y_{s}(t)\right)&=&\sum_{k=1}^{\infty}\chi_{k}(s,Y)\phi_{k}(t),\quad t\in \mathbb{T},\nonumber\\
\log_{\mu_{X_{0},\mathcal{M}}(t)}\left(X_{s}(t)\right)&=&\sum_{k=1}^{\infty}\chi_{k}(s,X)\phi_{k}(t),\quad t\in \mathbb{T},
\label{KLTSAHS}
\end{eqnarray}
\noindent  hold.  Here,  for $s\in \mathbb{Z},$
\begin{eqnarray}\chi_{k}(s,Y)&=&\int_{\mathbb{T}} \left[\log_{\mu_{Y_{0},\mathcal{M}}(t)}\left(Y_{s}(t)\right)\right]^{T}\phi_{k}(t)dt =LY_{s}(\phi_{k}),\ k\geq 1,
\nonumber\\
\chi_{k}(s,X)&=&\int_{\mathbb{T}} \left[\log_{\mu_{X_{0},\mathcal{M}}(t)}\left(X_{s}(t)\right)\right]^{T}\phi_{k}(t)dt =LX_{s}(\phi_{k}),\ k\geq 1,
\label{proyresl}
\end{eqnarray}
\noindent is the $s$-varying RFPC scores of $\log_{\mu_{Y_{0},\mathcal{M}}(t)}\left(Y_{s}(t)\right)$ and $\log_{\mu_{X_{0},\mathcal{M}}(t)}\left(X_{s}(t)\right),$
 respectively.

Assume also that the regression operator, characterizing the correlation between the log-mapped response and regressor processes in   $\mathbb{H},$  admits a compact Fr\'echet derivative operator satisfying (\ref{eqopfrechet}),
 in terms of the  eigenfunctions in (\ref{KLTSAHS})--(\ref{proyresl}).   Then, the exponential map
  \begin{equation}
 \widehat{Y_{s}}(x_{0})=\exp_{\mu_{Y_{0},\mathcal{M}}(\cdot)}\left(\sum_{k\geq 1} \widehat{m(x_{0}^{(s)})}(\phi_{k})\phi_{k}\right),\ s\in \mathbb{Z},\label{s2302jb}
 \end{equation}

  \noindent of the least-squares $\mathbb{H}$-valued local linear predictor $\sum_{k\geq 1} \widehat{m(x_{0}^{(s)})}(\phi_{k})\phi_{k}$
 stays on the  manifold surface, and its residual variability
   is upper bounded by the residual variability of $\sum_{k\geq 1} \widehat{m(x_{0}^{(s)})}(\phi_{k})\phi_{k}$ in $\mathbb{H}.$ Here, for every $k\geq 1,$
  $\widehat{m(x_{0}^{(s)})}(\phi_{k})$ is computed from   equations   (\ref{s2302i})--(\ref{s2302h2}), in terms of the log-mapped response and regressor processes.
\end{proposition}
\begin{remark}
\label{rem5}
Proposition 1 also holds  beyond the condition of the` RFPCAs  of the  log-mapped response and regressor processes are given in terms of a common eigenfunction system, see Chapter 8 in \cite{Bosq.00}.

We have also to note that  the  invariance of the kernel entries of matrix operators  on $\mathbb{H}=L^{2}_{\mathcal{M}}(\mathbb{T}),$ under translations, allows their diagonalization in terms of a vector  complex exponential basis  indexed by $\mathbb{Z}.$    It is well-known that this feature is shared by locally compact Abelian groups, see, e.g.,  \cite{Marinucci11}. In that case, the assumption of homogeneity of the Fr\'echet derivative indeed means local homogeneity of the regressor operator,  not supposing a strong restriction. In general, the assumption on the spectral diagonalization of  the Fr\'echet derivative operator on $\mathbb{H}$ can be  weakened,  leading to some extra computational burden.

\end{remark}
\begin{proof}
The RFPCAs given  in (\ref{KLTSAHS}), under the existence of a compact Fr\'echet derivative  of the regressor operator, satisfying (\ref{eqopfrechet}) in $\mathbb{H},$ allows the computation of $\widehat{\log_{\mu_{Y_{0},\mathcal{M}}(t)}}\left(Y_{s}(t)\right)=\sum_{k\geq 1} \widehat{m(x_{0}^{(s)})}(\phi_{k})\phi_{k}$
(see also Remark \ref{rem5}).
Condition (ii) implies
 the residual variability associated with the $\mathbb{H}$-valued  least-squares kernel estimator upper bounds  the residual variability of its exponential map (see Remark \ref{rem1}). Conditions (iii) and (v) ensure  $\widehat{Y_{s}}(x_{0})$ in equation (\ref{s2302jb}) stays on  the manifold.

\end{proof}
\begin{remark}
The  empirical version of $\widehat{Y_{s}}(x_{0})$ is given by  $$\widehat{Y_{s}}^{(n)}(x_{0})=\exp_{\mu_{Y_{0},\mathcal{M}}(\cdot)}\left(\sum_{k\geq 1} \widehat{m_{n}(x_{0}^{(s)})}(\phi_{k})\phi_{k}\right),\ s\in \mathbb{Z},$$
 \noindent  with
 \begin{eqnarray}&&
 \widehat{m_{n}(x_{0}^{(s)})}(\phi_{k})=\mbox{arg min}_{h\in \mathbb{H}}\frac{1}{n}\sum_{i=1}^{n}S^{(k)}(\log_{\mu_{X_{0},\mathcal{M}}(\cdot)}\left(X_{s_{i}}(\cdot )\right),x_{0},B_{n})\nonumber\\
 &&\hspace*{5.5cm}
 \times [LY_{s_{i}}(\phi_{k})-h(\phi_{k})]^{2},\ k\geq 1.
 \label{ev}
 \end{eqnarray}
 \noindent Under conditions assumed in Proposition \ref{pr1bb}, consistency of this empirical predictor follows from condition (iv).
\end{remark}
\section{Intrinsic local  Fr\'echet curve regression  in $\mathcal{M}$}
\label{sec4}

This section introduces a Nadaraya-Watson-type  (NW-type)  curve predictor   evaluated   in $\mathcal{M}$ in Section \ref{ss1}, and an intrinsic local linear
Fr\'echet curve predictor in Section \ref{seclcfp}. The asymptotic least-squares optimality of the second one is proved in Lemma  \ref{lem1} and Proposition \ref{pr1}.

\subsection{Local Fr\'echet curve  prediction based on  NW-type estimation  in $\mathcal{M}$}
\label{ss1}

As before,  $Y=\{ Y_{s},\ s\in \mathbb{Z}\}$ and $X=\{ X_{s},\ s\in \mathbb{Z}\}$ respectively denote the response and regressor curve   processes  evaluated in  $\mathcal{M}$ under conditions  (i)--(v)   in Section \ref{sec3.2}.

For each $h\in \mathbb{N}_{0},$ we consider the theoretical loss function
 \begin{eqnarray}
M_{\oplus}(x(t), h, \omega ) &=&
E\left[ K_{h}\left(d_{\mathcal{M}}\left(X_{s}(t),x(t)\right)\right) \left[d_{\mathcal{M}}\left(Y_{s}(t), \omega\right)
\right]^{2}\right]\nonumber\\
&=&
E\left[ K_{h}\left(d_{\mathcal{M}}\left(X_{0}(t),x(t)\right)\right) \left[d_{\mathcal{M}}\left(Y_{0}(t), \omega\right)
\right]^{2}\right],
\end{eqnarray}

\noindent  for $t\in \mathbb{T},$ $\omega\in \mathcal{M},$ and  $s\in \mathbb{Z},$
where  $K_{h}\left(d_{\mathcal{M}}\left( X_{0}(t), x(t)\right)\right)$ is a zonal function,  associated with the $h$th eigenspace of the Laplace Beltrami operator on $L^{2}(\mathcal{M}, d\nu ,\mathbb{R}),$  with time-varying pole at  $x(t),$   applied to the time-varying random arguments $X_{s}(t),$ $t\in \mathbb{T},$ $s\in \mathbb{Z}$.  See \cite{Gine75} for a more detailed description of zonal functions.  The NW-type Fr\'echet curve predictor is given by, for $t\in \mathbb{T},$ and  $h\in \mathbb{N}_{0},$
\begin{eqnarray}\widehat{Y}^{NW}_{s}(t)&=& m_{\oplus}(x(t),h)=\arg\min_{\omega\in \mathcal{M}}E\left[ K_{h}\left(d_{\mathcal{M}}\left(X_{s}(t),x(t)\right)\right) \left[d_{\mathcal{M}}\left(Y_{s}(t), \omega\right)
\right]^{2}\right]\nonumber\\
&=&\arg\min_{\omega\in \mathcal{M}}
E\left[ K_{h}\left(d_{\mathcal{M}}\left(X_{0}(t),x(t)\right)\right) \left[d_{\mathcal{M}}\left(Y_{0}(t), \omega\right)
\right]^{2}\right].
\label{NWE}
\end{eqnarray}

Let $(Y_{s_{1}}(\cdot),\ X_{s_{1}}(\cdot)),\dots, (Y_{s_{n}}(\cdot),\  X_{s_{n}}(\cdot))$ be
a bivariate  functional sample of size  $n$  of the  $\mathcal{M}$-valued  bivariate curve process $(Y,X).$
The empirical version of NW-type  Fr\'echet curve  predictor in (\ref{NWE}) is given by the minimizer,  in  $\omega  \in \mathcal{M},$ of
the loss function
\begin{eqnarray}&&
\widehat{M}_{\oplus}(x(t), h, \omega) =
\frac{1}{n}\sum_{i=1}^{n} K_{h}\left(d_{\mathcal{M}}\left(X_{i}(t),x(t)\right)\right) \left[d_{\mathcal{M}}\left(Y_{i}(t), \omega\right)
\right]^{2},\ t\in \mathbb{T},\nonumber
\end{eqnarray}
\noindent for each $h\in \mathbb{N}_{0}.$
Here,   $K_{h}\left(d_{\mathcal{M}}\left( X_{i}(t), x(t)\right)\right)$    is proportional to  a  Jacobi polynomial  in the case of $\mathcal{M}$ being a connected and compact  two point homogeneous space, as given in  \cite{MaMalyarenko}. The corresponding  empirical  local curve Fr\'echet  predictor is given by, for every  $t\in \mathbb{T},$ and $h\in \mathbb{N}_{0},$
\begin{eqnarray}
 \widehat{Y}^{NW}_{s,n}(t)&=&\widehat{m}_{\oplus}(x(t),h)
 =\arg\min_{\omega\in \mathcal{M}}\frac{1}{n}\sum_{i=1}^{n} K_{h}\left(d_{\mathcal{M}}\left(X_{i}(t),x(t)\right)\right) \left[d_{\mathcal{M}}\left(Y_{i}(t),\omega\right)
\right]^{2}.\nonumber\\
\label{NWF}
\end{eqnarray}

\begin{remark}
Parameter $h$ of the zonal function $K_{h}$ plays the role of the inverse of the bandwidth parameter. Specifically, when we alternatively consider  a kernel $K$ defined by a  probability density with compact support contained in $\mathcal{M},$ $h$ plays the role of the inverse of the concentration parameter.
\end{remark}
\subsection{Intrinsic local linear  Fr\'echet curve  prediction  in $\mathcal{M}$}
\label{seclcfp}

Under conditions (i)-(v), let us consider, for any $s\in \mathbb{Z},$ and $h\in \mathbb{N}_{0},$ the following intrinsic local linear  Fr\'echet curve predictor:
\begin{eqnarray}
\widehat{Y}_{s}^{LL}(t)&=&  \arg \min_{\omega \in \mathcal{M}} E\left[s(X_{s}(t),x(t), h)\left[d_{\mathcal{M}}\left(Y_{s}(t),\omega\right)\right]^{2}\right]
\nonumber\\
&=&\arg \min_{\omega\in \mathcal{M}}E\left[s(X_{0}(t),x(t), h)\left[d_{\mathcal{M}}\left(Y_{0}(t),\omega \right)
\right]^{2}\right]=m_{L,\oplus}(x(t),h) ,\ \forall t\in \mathbb{T}.\nonumber\\
\label{elfb}
\end{eqnarray}
\noindent  The non-linear weights  are given by, for every $t\in \mathbb{T},$ and $h\in \mathbb{N}_{0},$
\begin{eqnarray} &&%\hspace{-1cm}
s(X_{0}, x(t), h) = \frac{1}{\sigma_0^2}K_{h}\left(d_{\mathcal{M}}\left(X_{0}(t),x(t)\right)\right) [\mu_2(x(t), h)  - \mu_1(x(t), h) d_{\mathcal{M}}\left(X(t), x(t)\right)], \nonumber \\
&&%\hspace{-1cm}
\mu_j (x(t),h)= E\left[K_{h}\left(d_{\mathcal{M}}\left(X_{0}(t),x(t)\right)\right)
\left[d_{\mathcal{M}}\left(X_{0}(t),x(t)\right)\right]^{j}\right], \nonumber\\
&&
\sigma_0^2(x(t),h) = \mu_0(x(t), h)\mu_2(x(t), h)-[\mu_1(x(t), h)]^2. \label{glm}
\end{eqnarray}
   As before,  for each $h\in \mathbb{N}_{0},$  $K_{h}\left(d_{\mathcal{M}}\left( X_{0}(t), x(t)\right)\right)$ is a zonal function, with pole at $x(t),$  $t\in \mathbb{T},$ associated with the $h$th eigenspace of the Laplace Beltrami operator on $L^{2}(\mathcal{M}, d\nu ,\mathbb{R}),$ and
   evaluated at the time--dependent random argument  $X_{0}(t).$

The corresponding empirical version of (\ref{elfb}) is given by
\begin{eqnarray} &&
\widehat{Y}_{s,n}^{LL}(t)=\widehat{m}_{L,\oplus}(x(t),h) \nonumber\\&& = \arg \min_{\omega \in \mathcal{M}} \frac{1}{n}\sum_{i=1}^{n}
\frac{1}{\widehat{\sigma}_0^2}K_{h}\left(d_{\mathcal{M}}\left(X_{s_{i}}(t), x(t)\right)\right)\nonumber\\
&&\hspace*{2.5cm}\times 
 \left[\widehat{\mu}_2 (t)- \widehat{\mu}_1(t)d_{\mathcal{M}}\left(X_{s_{i}}(t),x(t)\right)\right]
d^{2}_{\mathcal{M}}\left(Y_{s_{i}}(t),\omega\right),\ t\in \mathbb{T},\nonumber\\
\label{empv}
\end{eqnarray}
\noindent  for each $h\in \mathbb{N}_{0},$ where, for each  $t\in \mathbb{T},$
\begin{eqnarray}
&& \widehat{\mu}_j (x(t),h)= \frac{1}{n}\sum_{i=1}^{n}K_{h}\left(d_{\mathcal{M}}\left( X_{s_{i}}(t), x(t)\right)\right)
\left[d_{\mathcal{M}}\left(X_{s_{i}}(t),x(t)\right)\right]^{j}, \nonumber\\
&& \widehat{\sigma}_0^2(x(t),h) =\widehat{\mu}_0(x(t),h)\widehat{\mu}_2(x(t),h)-\widehat{\mu}_1^2(x(t),h). \nonumber
\end{eqnarray}

Optimality of (\ref{elfb})  is proved in Lemma \ref{lem1}  and Proposition \ref{pr1}  below, in a similar way to Lemma 1 of  the Supplementary Material of \cite{Petersen.19}, and Theorem 3  in \cite{Petersen.19}. The following additional conditions are required:

\vspace*{0.5cm}

\noindent (vi) For $j=0,1,2,$ and $\widetilde{x}\in \mathcal{M},$
\begin{equation}K_{h}^{(j)}(\widetilde{x})=\int_{\mathcal{M}}K_{h}\left(d_{\mathcal{M}}\left( x, \widetilde{x}\right)\right)[d_{\mathcal{M}}\left( x, \widetilde{x}\right)]^{j}d\nu(x)<\infty,\quad h\in \mathbb{N}_{0}.\label{eqvi}
\end{equation}

\medskip

\noindent (vii) For every $t\in \mathbb{T},$ the probability distribution of the $\mathcal{M}$--valued one-dimensional time projection $X_{0}(t)$  of the random curve  $X_{0}$ is characterized by a geodesically continuously differentiable probability density $f_{X_{0}(t)}.$ The conditional probability density $g_{X_{0}(t)/Y_{0}(t)=y_{t}}$ of $X_{0}(t)$ given $Y_{0}(t)=y_{0}(t)$ exists, and it is twice geodesically continuously  differentiable, satisfying $$\sup_{(x_{0}(t),y_{0}(t))\in \mathcal{X}_{\mathcal{C}_{\mathcal{M}}(\mathbb{T})}\times  \mathcal{Y}_{\mathcal{C}_{\mathcal{M}}(\mathbb{T})}}\left|g^{\prime \prime }_{X_{0}(t)/Y_{0}(t)=y_{0}(t)}(x_{0}(t))\right|< \infty.$$  
\noindent  Also,   for  $j=0,1,$ and  $t\in \mathbb{T},$ the  local conditional moment
\begin{eqnarray}
\tau_{j}(y_{0}(t), x(t),h)&=&\int K_{h}\left(d_{\mathcal{M}}\left( x_{0}(t), x(t)\right)\right)  \nonumber\\
&&\hspace*{1cm}\times [d_{\mathcal{M}}\left( x_{0}(t), x(t)\right)]^{j}
g_{X_{0}(t)/Y_{0}(t)=y_{0}(t)}(x_{0}(t))d\nu (x_{0}(t)),\nonumber\\
\label{gclm}
\end{eqnarray}
\noindent  for  $h\in \mathbb{N}_{0},$ is finite.

\begin{lemma}
\label{lem1}
Under assumptions  (i)-(vii), the following identities hold for the local moments introduced in equation (\ref{glm}): For every $t\in \mathbb{T},$  and for $j=0,1,2,$ as $h\to \infty,$
\begin{eqnarray}\mu_{j}(x(t),h)&=&[\mathcal{D}(K_{h})/2]^{j}\left[f_{X_{0}(t)}(x(t))K_{h}^{(j)}(x(t))+f_{X_{0}(t)}^{\prime }(x(t))K_{h}^{(j+1)}(x(t))\right.\nonumber\\ &&\left. +\mathcal{O}\left([\mathcal{D}(K_{h})/2]^{2}\right)\right],\label{glm2}
\end{eqnarray}
\noindent where, for any positive natural $l,$  $K_{h}^{(l)}$ has been introduced in equation (\ref{eqvi}), and
\begin{eqnarray}f_{X_{0}(s)}^{\prime }(x(s))=\lim_{t\to 0}\frac{f_{X_{0}(s)}(\exp_{x(s)}(tv))-f_{X_{0}(s)}(x(s))}{t},
\label{gclm2d}\end{eqnarray}
\noindent for certain $v\in \mathcal{T}_{x(s)}\mathcal{M}.$  Here, $\mathcal{D}(K_{h})=\max_{x,y\in \mbox{Supp}(K_{h})}  d_{\mathcal{M}}(x,y),$ with $\mbox{Supp}(K_{h})$ denoting
the support of zonal function $K_{h}$ centered at $x(t)$ for each $t\in \mathbb{T}.$

 The conditional local moments in (\ref{gclm}) satisfy for $j=0,1,$  and for every $t\in \mathbb{T},$ as $h\to \infty,$
 \begin{eqnarray}
\tau_{j}(y_{0}(t),  x(t), h)&=& [\mathcal{D}(K_{h})/2]^{j}\left[g_{X_{0}(t)/Y_{0}(t)=y_{0}(t)}(x(t))K_{h}^{(j)}(x(t))\right.\nonumber\\ &&\hspace*{2cm}\left.+g_{X_{0}(t)/Y_{0}(t)=y_{0}(t)}^{\prime }(x(t))K_{h}^{(j+1)}(x(t))\right.\nonumber\\ &&\hspace*{2.5cm}\left.+\mathcal{O}\left([\mathcal{D}(K_{h})/2]^{2}\right)\right],\label{gclm2}
\end{eqnarray}

\noindent with \begin{eqnarray}&&g_{X_{0}(s)/Y_{0}(s)= y_{0}(s)}^{\prime }(x(s))\nonumber\\ &&\hspace*{1cm}=\lim_{t\to 0}\frac{g_{X_{0}(s)/Y_{0}(s)=y_{0}(s)}^{\prime }(\exp_{x(s)}(tv))-g_{X_{0}(s)/Y_{0}(s)=y_{0}(s)}^{\prime }(x(s))}{t},\nonumber\\
 \label{gclm3}
\end{eqnarray}
\noindent for certain $v\in \mathcal{T}_{x(s)}\mathcal{M}.$
\end{lemma}
\begin{remark}
Note that under condition (v), the derivatives in (\ref{gclm2d}) and (\ref{gclm3})  can be diffeomorphically computed   along a geodesic with initial location $x(s)$ and velocity $v,$  since, in particular,  $\exp_{x(s)}(tv),$ $t\in \mathbb{T},$  lies in the support of $dP_{X_{0}}.$
\end{remark}
\begin{proof}
Under conditions (vi)-(vii), for each $t\in \mathbb{T},$    one can consider  the following first order local approximations at point $x_{0}(t),$  along the geodesic $\exp_{x_{0}(t)}(sv),$
$s\in [0,1],$ of  $f_{X_{0}}(x)$  and $g_{X_{0}(t)/Y_{0}(t)= y_{0}(t)}(x),$
\begin{eqnarray}&&
f_{X_{0}}(x)=f_{X_{0}}(x_{0}(t))+d_{\mathcal{M}}\left( x, x_{0}(t)\right)f^{\prime }(x_{0}(t))+\mathcal{O}\left([\mathcal{D}(K_{h})/2]^{2}\right),\nonumber\\
&& g_{X_{0}(t)/Y_{0}(t)= y_{0}(t)}(x)=g_{X_{0}(t)/Y_{0}(t)= y_{0}(t)}(x_{0}(t))
\nonumber\\
&& \hspace*{2cm}
+d_{\mathcal{M}}\left( x, x_{0}(t)\right)g_{X_{0}(t)/Y_{0}(t)= y_{0}(t)}^{\prime }(x_{0}(t))+\mathcal{O}\left([\mathcal{D}(K_{h})/2]^{2}\right),\nonumber\\\label{gclm3b}
\end{eqnarray}
\noindent  for any  $x\in \mbox{Supp}(K_{h}),$  leading to
equations (\ref{glm2}) and (\ref{gclm2}), respectively, as $h\to \infty.$
\end{proof}
The following result provides the asymptotic optimality of the intrinsic local linear Fr\'echet curve predictor (\ref{elfb}).
\begin{proposition}
\label{pr1}
Under conditions (i)-(vii), the following identity holds
as \linebreak $h\to \infty$
\begin{eqnarray} &&
\int d_{\mathcal{M}}^{2}(y_{0}(t),\omega)s(z_{0}(t),x(t),h) dF_{X_{0},Y_{0}}(z_{0}(t),y_{0}(t))\nonumber\\ &&=\int d_{\mathcal{M}}^{2}(y_{0}(t),\omega)dF_{Y_{0}/X_{0}}(x(t),y_{0}(t))+\mathcal{O}\left([\mathcal{D}(K_{h})/2]^{2}\right),\nonumber\\
\label{weihtmeanapcpf}
\end{eqnarray}
\noindent where $s(z_{0}(t),x(t),h)$ has been introduced in equation (\ref{glm}).
\end{proposition}

\begin{proof}

Applying Lemma \ref{lem1}, in a similar way to the proof of Theorem 3 in the Supplementary material of \cite{Petersen.19},
\begin{eqnarray} &&
\int d_{\mathcal{M}}^{2}(y_{0}(t),\omega)s(z_{0}(t),x(t),h) dF_{X_{0},Y_{0}}(z_{0}(t),y_{0}(t))
\nonumber\\ &&=\int d_{\mathcal{M}}^{2}(y_{0}(t),\omega)s(z_{0}(t),x(t),h) dF_{X_{0}(t)/Y_{0}(t)}(z_{0}(t)/y_{0}(t))dF_{Y_{0}(t)}(y_{0}(t))\nonumber\\
 &&=\int d_{\mathcal{M}}^{2}(y_{0}(t),\omega)\left[\frac{\mu_{2} (x(t),h)\tau_{0}(y_{0}(t),  x(t),h)}{\sigma_{0}^{2}(x(t),h) } \right.\nonumber\\ &&\hspace*{3.5cm}\left.-\frac{\mu_{1}(x(t),h)\tau_{1}(y_{0}(t),  x(t),h)}{\sigma_{0}^{2}(x(t),h) }\right]
dF_{Y_{0}(t)}(y_{0}(t))\nonumber\\
&&=\int d_{\mathcal{M}}^{2}(y_{0}(t),\omega)\frac{ g_{X_{0}(t)/Y_{0}(t)= y_{0}(t)}(x(t))}{f_{X_{0}}(x(t))}dF_{Y_{0}(t)}(y_{0}(t))
+\mathcal{O}\left([\mathcal{D}(K_{h})/2]^{2}\right)\nonumber\\
&&=\int d_{\mathcal{M}}^{2}(y_{0}(t),\omega )dF_{Y_{0}(t)/X_{0}(t)}(y_{0}(t)/x(t))+\mathcal{O}\left([\mathcal{D}(K_{h})/2]^{2}\right),
\label{weihtmeanapcpf2}
\end{eqnarray}
\noindent where $\mathcal{D}(K_{h})\to 0,$ $h\to \infty.$
\end{proof}

\section{Simulation study}
\label{simstud}
We restrict here our attention to the sphere $\mathcal{M}=\mathbb{S}_2\subset  \mathbb{R}^3.$  The simulation algorithm introduced in  \cite{Torresetal2024} is implemented to generate a  time correlated  bivariate curve sample of size  $n=100 $ evaluated   in $\mathbb{S}_2.$   This algorithm pointwise applies the von Mises-Fisher  inverse transform to  the normalized,  by the  supremum norm, trajectories of
a  diffusion processes driven by vector Brownian motion with correlated components. The   regressor curve observations at 1000 temporal nodes
are displayed in Figure \ref{Fig:1.1},  for times  $s=10,20,30,40,50,60,70,80,90,100.$
\begin{figure}
	\centering
	\includegraphics[width=0.325\textwidth=0.33]{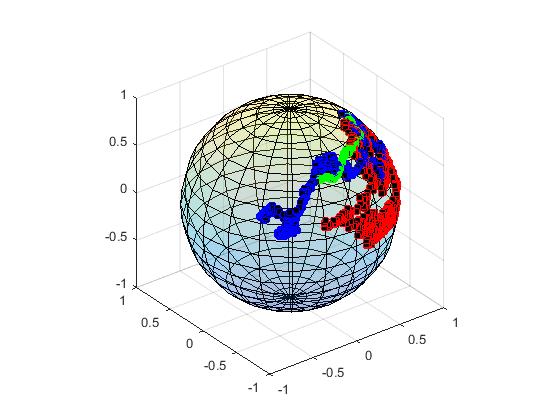}
\includegraphics[width=0.325\textwidth=0.33]{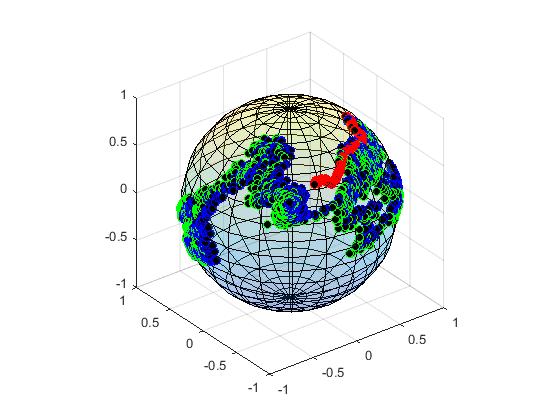}
\includegraphics[width=0.325\textwidth=0.33]{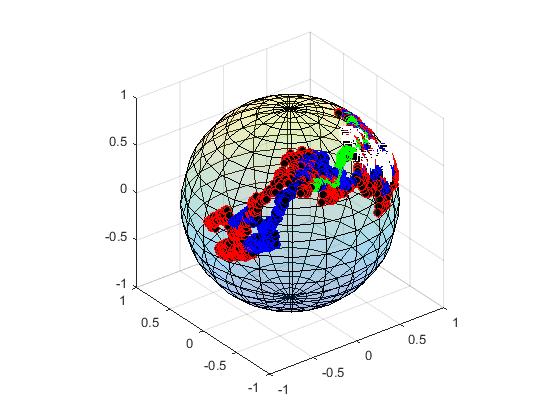}
	\caption{Spherical curve regressor observations. Left--hand side plot,  times $s=$10 (red),  20 (green),  and 30 (blue).  Center plot, time $s=40$  (red), $50$  (green), $60$  (blue), and right--hand side plot, time $s= 70$ (red), $80$ (green), $90$ (blue), $100$ (cyan)}
	\label{Fig:1.1}
\end{figure}
To compute the empirical  Fr\'echet curve mean $\widehat{\mu}_{X_{0},\mathcal{M}}$ (see right--hand side of Figure \ref{Fig:1.2}), an uniform spherical grid with  $20000\times 20000$   nodes is generated over a   region containing the  support of the regressor marginal probability measure (see left-hand side of Figure \ref{Fig:1.2}).
\begin{figure}
\centering
\includegraphics[width=0.35\textwidth]{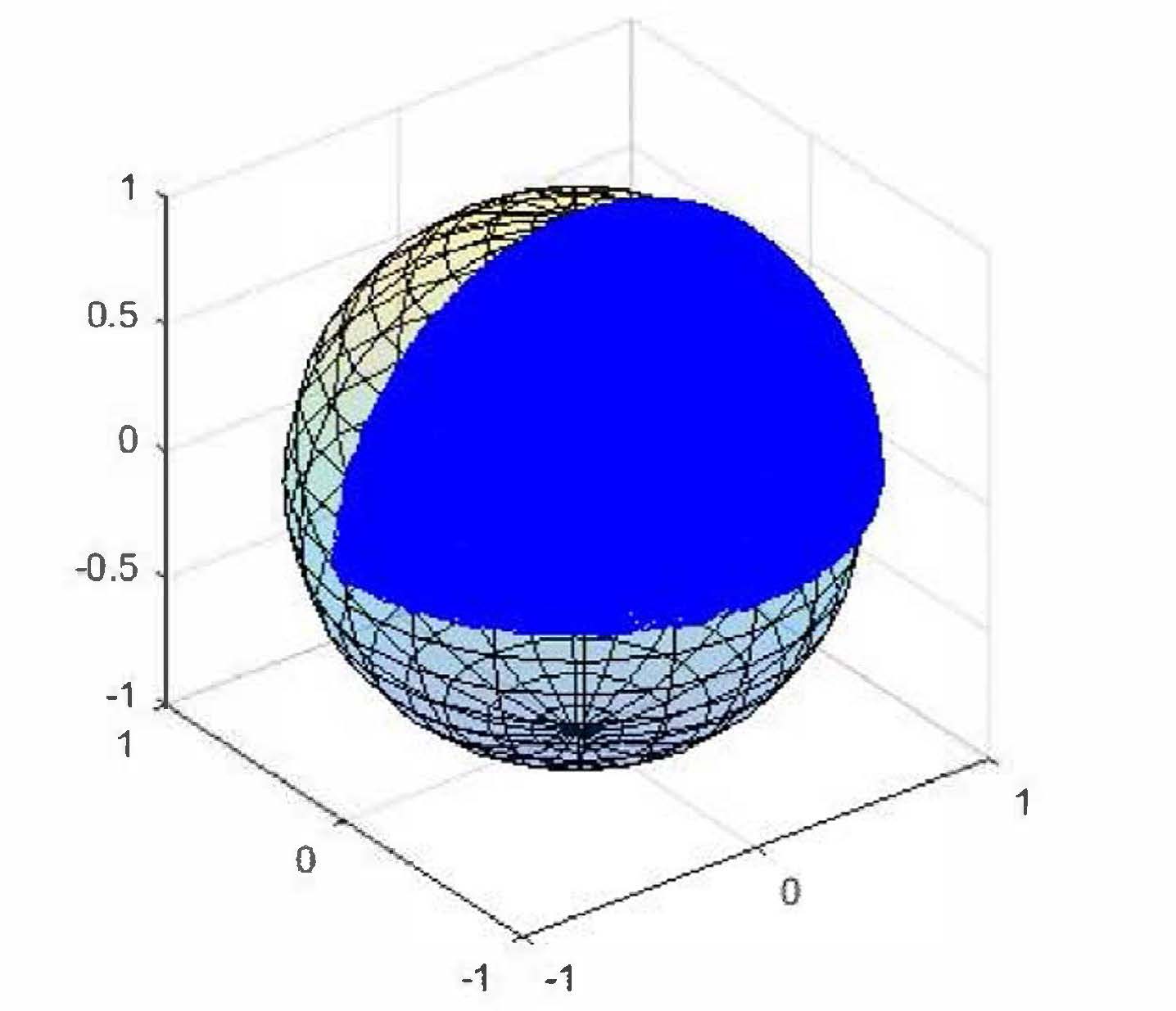}
\includegraphics[width=0.45\textwidth]{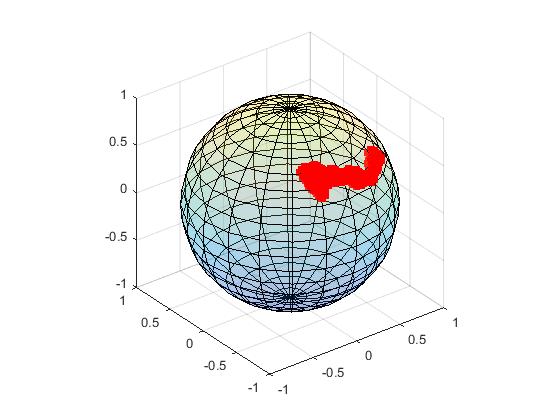}
\caption{Localized uniform grid of $20000\times 20000$ nodes at the left--hand side, and empirical  Fr\'echet curve mean at the right--hand side}\label{Fig:1.2}
\end{figure}
The regressor curve sample is log--mapped into the time--varying tangent space by applying the  time--varying  logarithm map with origin  at $\widehat{\mu}_{X_{0},\mathcal{M}}.$ That is, we compute
\begin{eqnarray}\log_{\widehat{\mu}_{X_{0},\mathcal{M}}(t)}(X_{s_{i}}(t))&=&\frac{u(t,i)}{\|u(t,i)\|}d_{\mathbb{S}_{2}}
  \left(\widehat{\mu}_{X_{0},\mathcal{M}}(t),X_{s_{i}}(t)\right),\nonumber\\ u(t,i)&=&X_{s_{i}}(t)-([\widehat{\mu}_{X_{0},\mathcal{M}}(t)]^{T}X_{s_{i}}(t))\widehat{\mu}_{X_{0},\mathcal{M}}(t),\quad t\in \mathbb{T}.\label{lms}\end{eqnarray}

 \noindent  The  curve values of the  response  process are then obtained as
$$ Y_{s_{i}}(t)= \exp_{\widehat{\mu}_{X_{0},\mathcal{M}}(t)}\left(\mathbf{\Gamma }\left(\log_{\widehat{\mu}_{X_{0},\mathcal{M}}(t)}\left(X_{s_{i}}\right)\right)(t)+\varepsilon_{s_{i}}(t)\right),\quad t\in \mathbb{T},$$
\noindent for  $i=1,\dots,n,$ where $\mathbf{\Gamma }:\mathbb{H}\to \mathbb{H}$ is a bounded linear operator, whose supremum  norm is  less than one, and $\left\{\varepsilon_{i}(\cdot),\ i\in \mathbb{Z}\right\}$  defines an $\mathbb{H}$--valued Gaussian  strong white noise, uncorrelated with the log--mapped regressors (see Figure \ref{Fig:1.3}).
\begin{figure}
\centering
\includegraphics[width=0.225\textwidth=0.25]{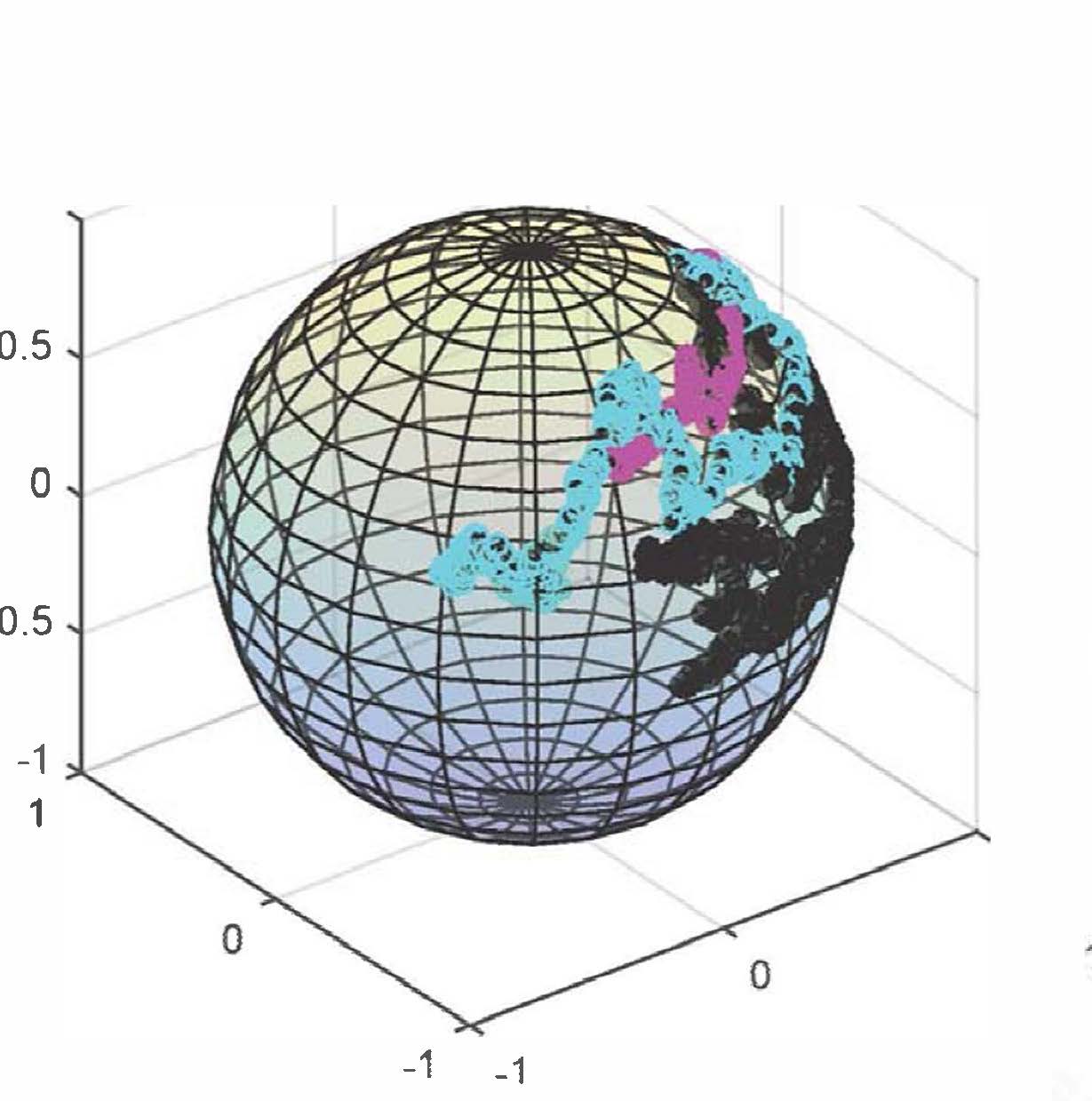}\hspace*{1.5cm}
\includegraphics[width=0.225\textwidth=0.25]{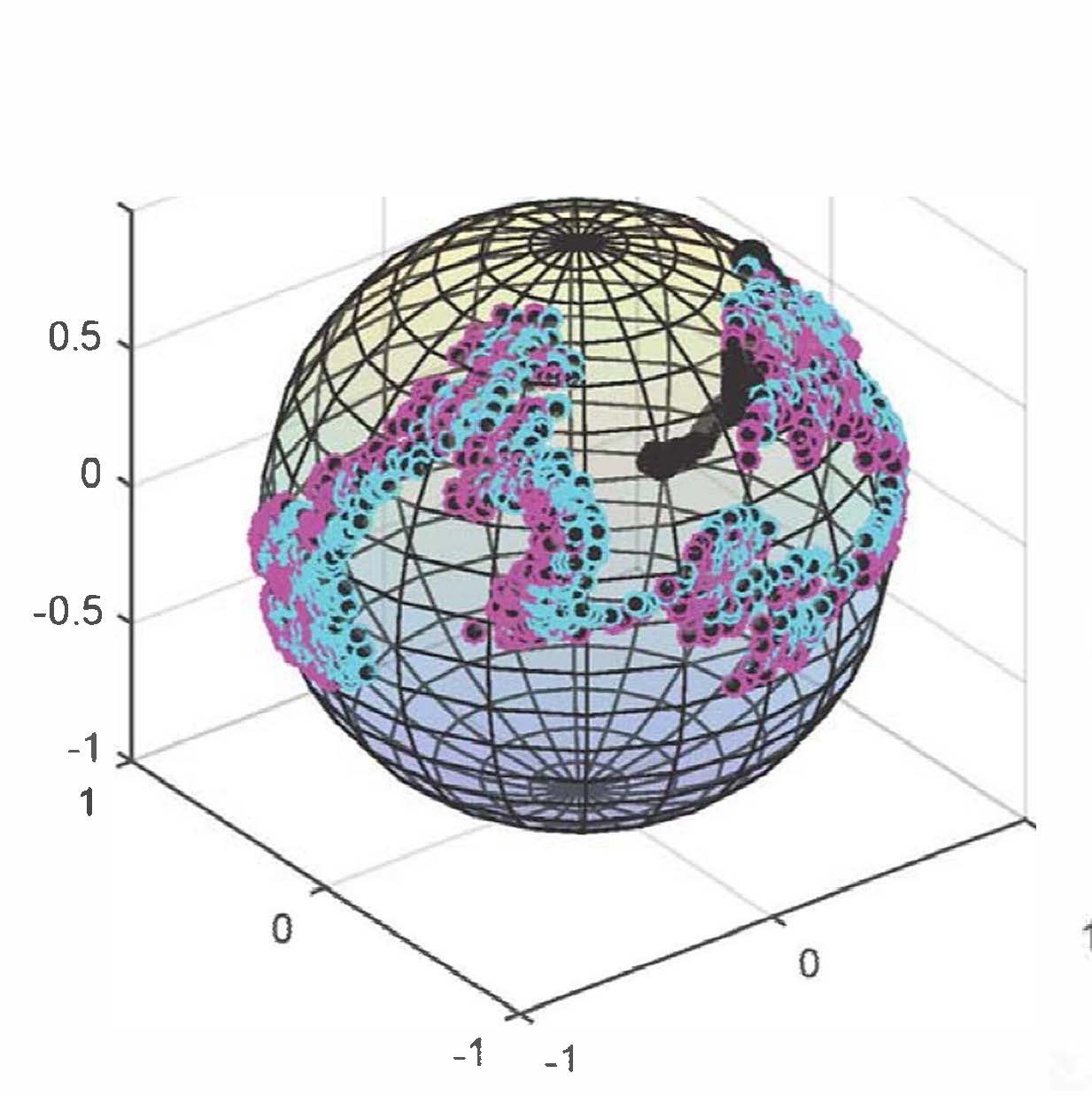}\hspace*{1.5cm}
\includegraphics[width=0.225\textwidth=0.25]{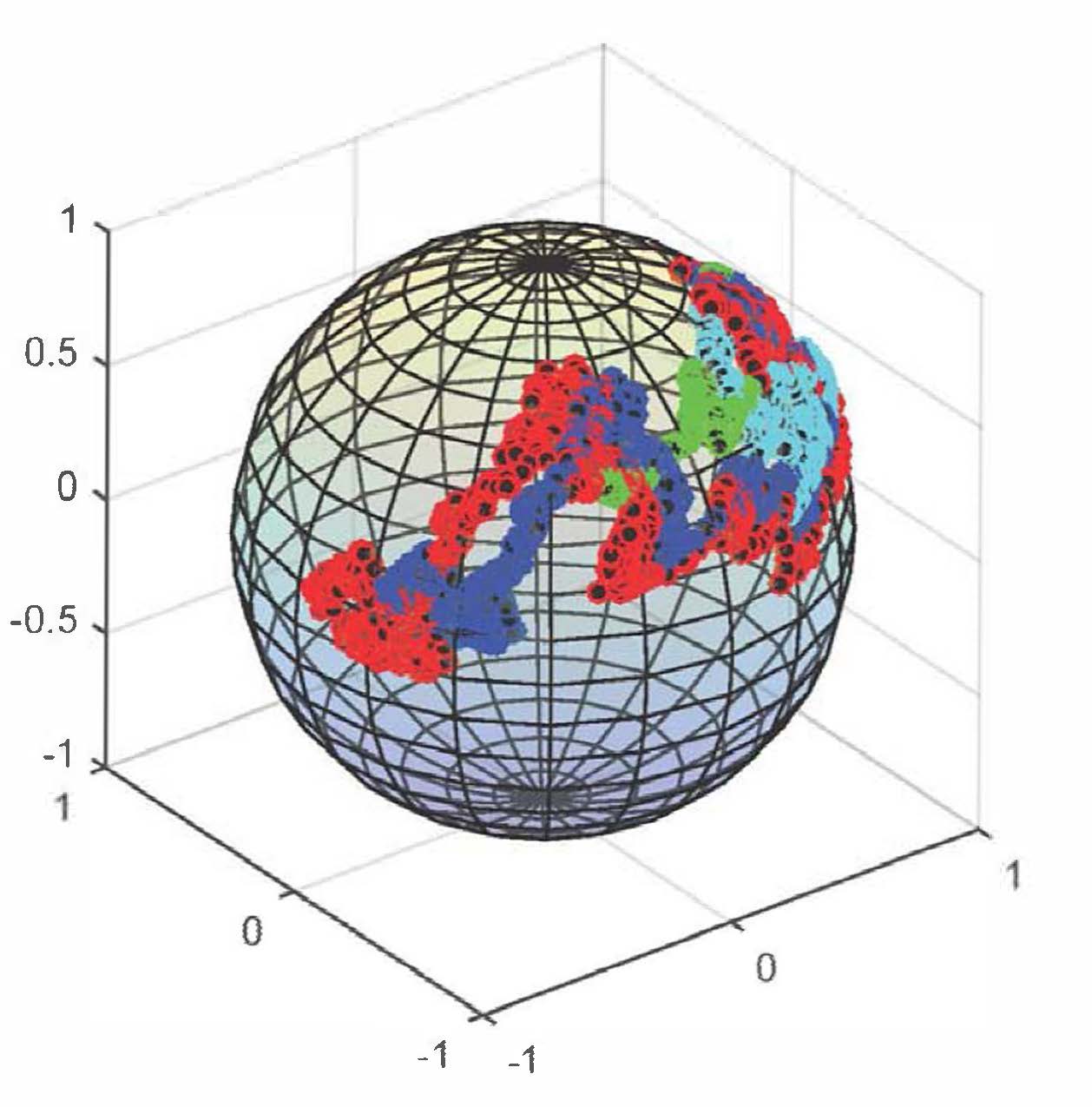}
\caption{Spherical curve response  observations generated at times $s=10$ (black), $20$ (pink), $30$ (cyan) for the left-hand-side plot, $s=40$ (black), $50$ (pink), 60 (cyan) for the center plot, $70$ (red), $80$ (green), $90$ (blue), $100$  (cyan) for the right-hand-side plot}\label{Fig:1.3}
\end{figure}
The sphere is a compact Riemannian manifold, where conditions (i)-(ii) hold. The sample path regularity and ergodicity of the underlying vector diffusion process ensure conditions (iii)-(iv) are satisfied by the generated regressor and response processes.  Correlation between the  components of the  vector driven process of such a diffusion process, and the selected high value of the concentration parameter, characterizing  inverse von Mises-Fisher transform, induce curve data concentration around the Fr\'echet curve mean, as required  in condition (v).

The  data set generated  displays stronger inter-curve  correlations than the real-data set analyzed in Section \ref{rda}. Note that the regressor curve process, and its functional linear transformation in the time-varying tangent space, defining the log-mapped response process,  inherit the underlying  temporal correlation of the transformed vector diffusion process, amplified by the von Mises-Fisher  concentration parameter.

 The empirical NW-type Fr\'echet curve predictor (\ref{NWF}) is first computed, considering bandwidth parameter values in the interval $( 0.3162,
 0.6310),$ corresponding to
   $[h(n)]^{-1}=n^{-\beta},$ $\beta=0.1000,    0.1375,    0.1750,    0.2125,    0.2500.$ For $\beta= 0.2500,$ i.e., $[h(n)]^{-1}=[0.3162]_{-},$ and
   $[h(n)]_{+}=4,$
   NW-type Fr\'echet curve values  are displayed in Figure \ref{Fig:1.4}  (red color),
at times  $s=10, 20, 30, 40, 50, 60, 70, 80, 90$, where the  corresponding response curve values (black color) are also represented for comparison.
\begin{figure}
\centering
\includegraphics[trim= 100 280 100 220, clip, width=0.32\textwidth]{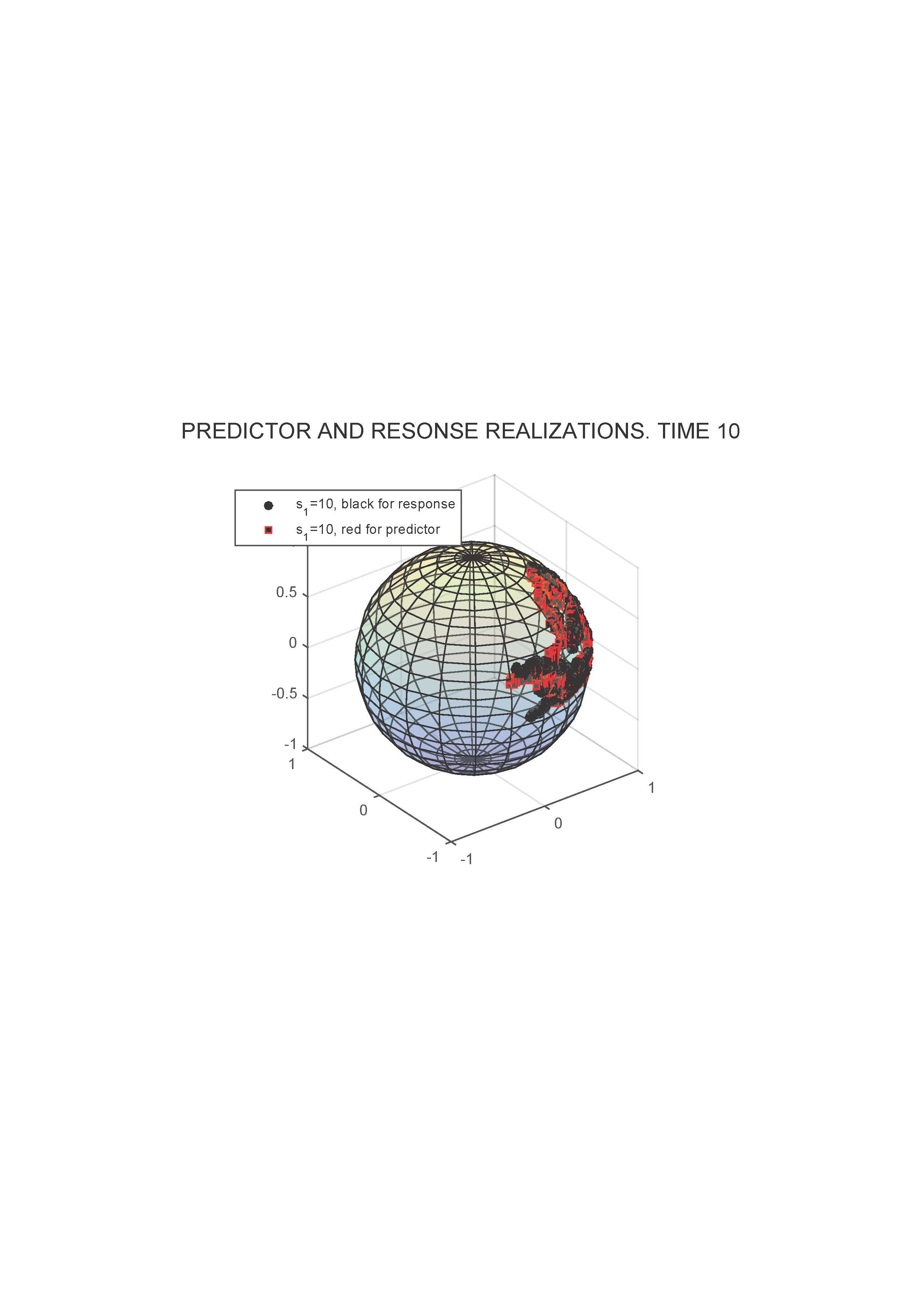}
\includegraphics[trim= 100 280 100 220, clip, width=0.32\textwidth]{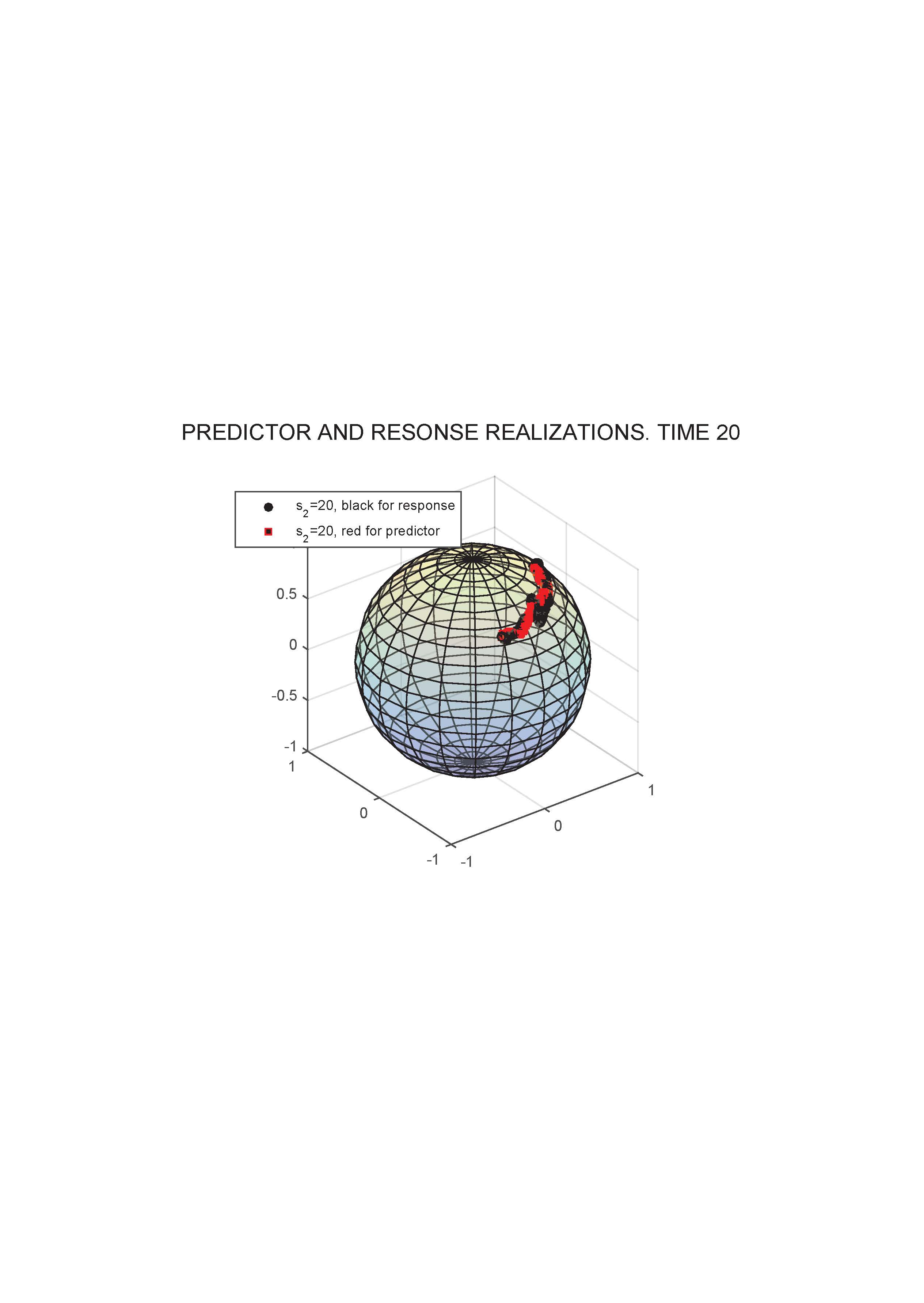}
\includegraphics[trim= 100 280 100 220, clip, width=0.32\textwidth]{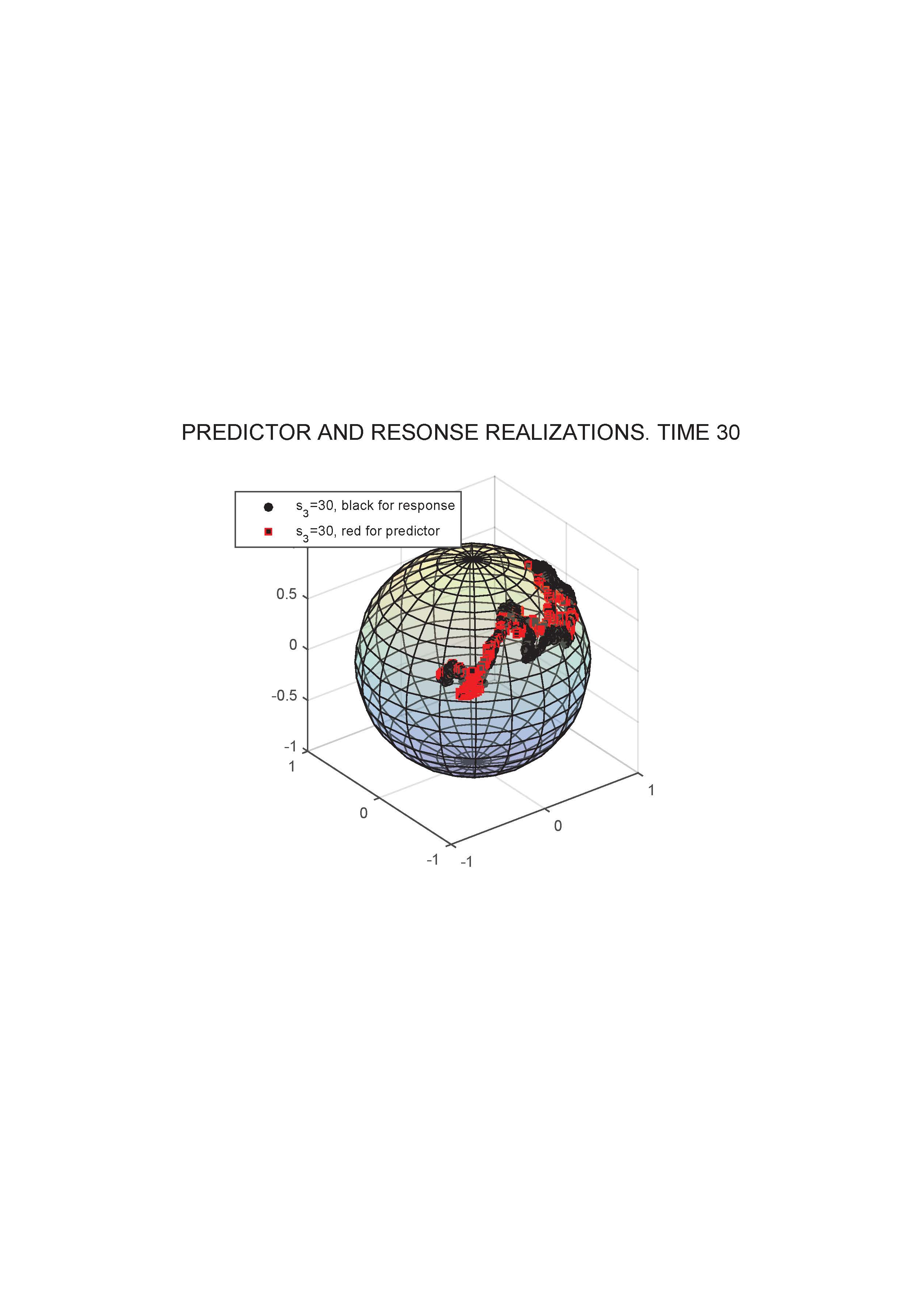}
\includegraphics[trim= 100 280 100 220, clip, width=0.32\textwidth]{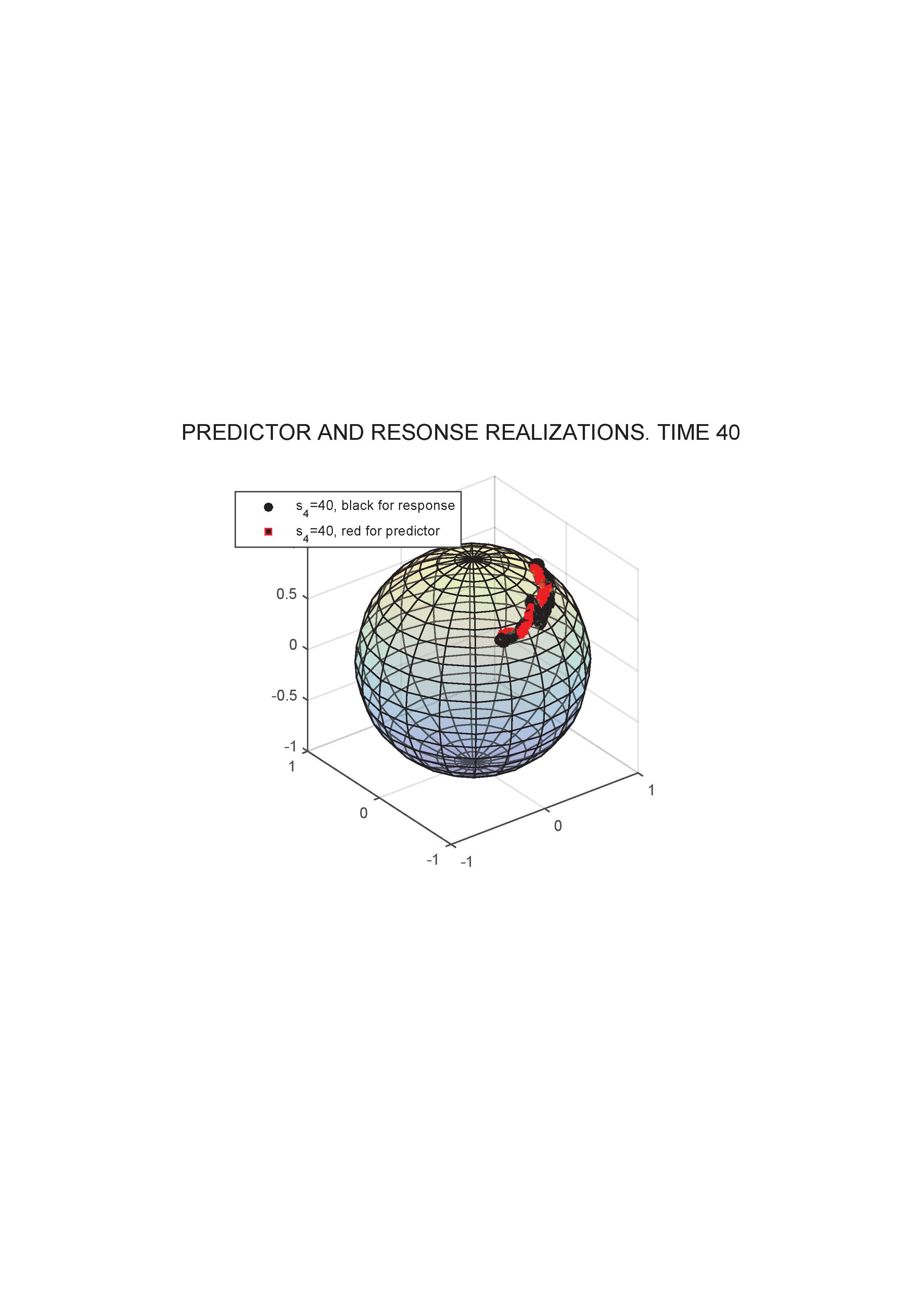}
\includegraphics[trim= 100 280 100 220, clip, width=0.32\textwidth]{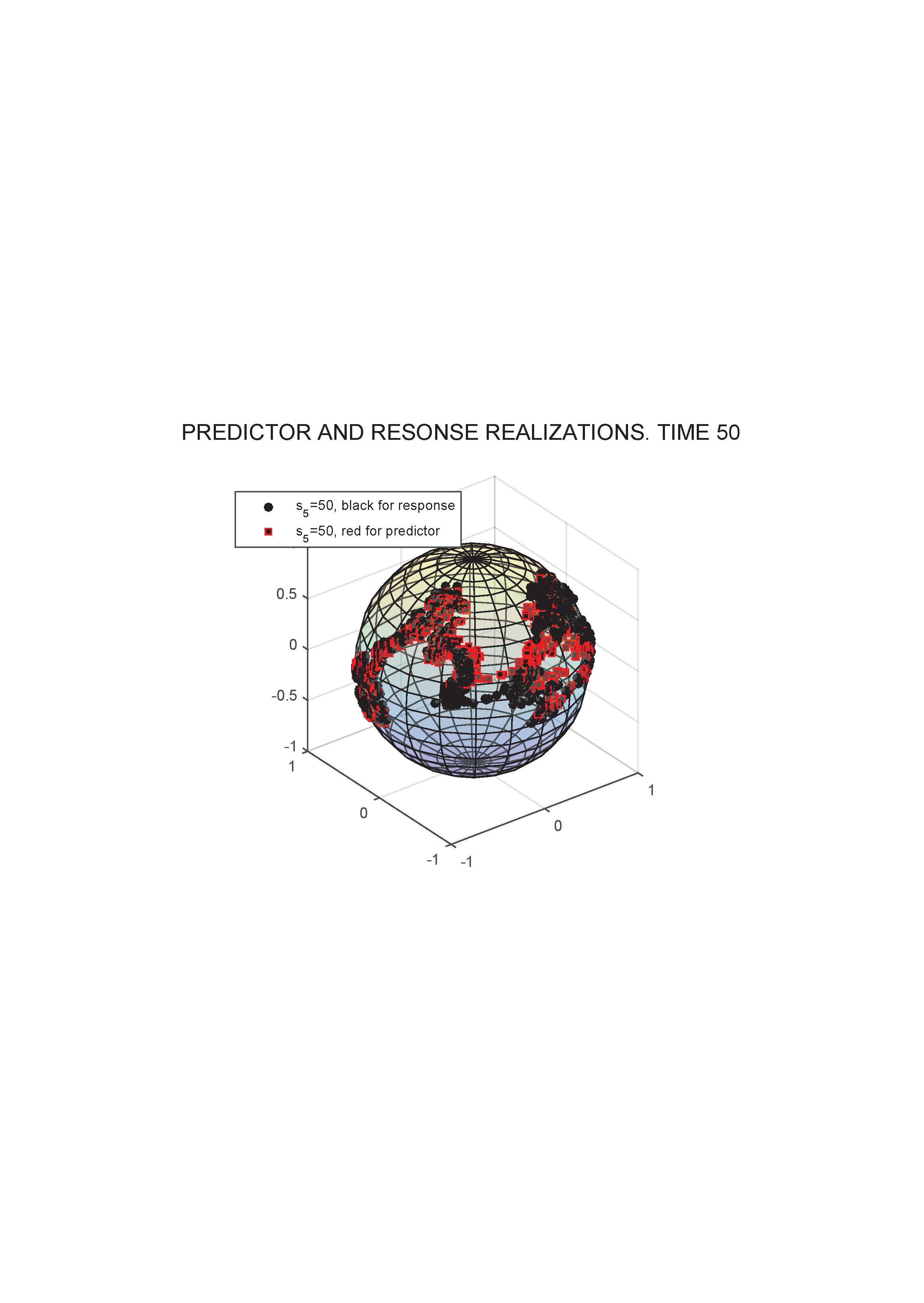}
\includegraphics[trim= 100 280 100 220, clip, width=0.32\textwidth]{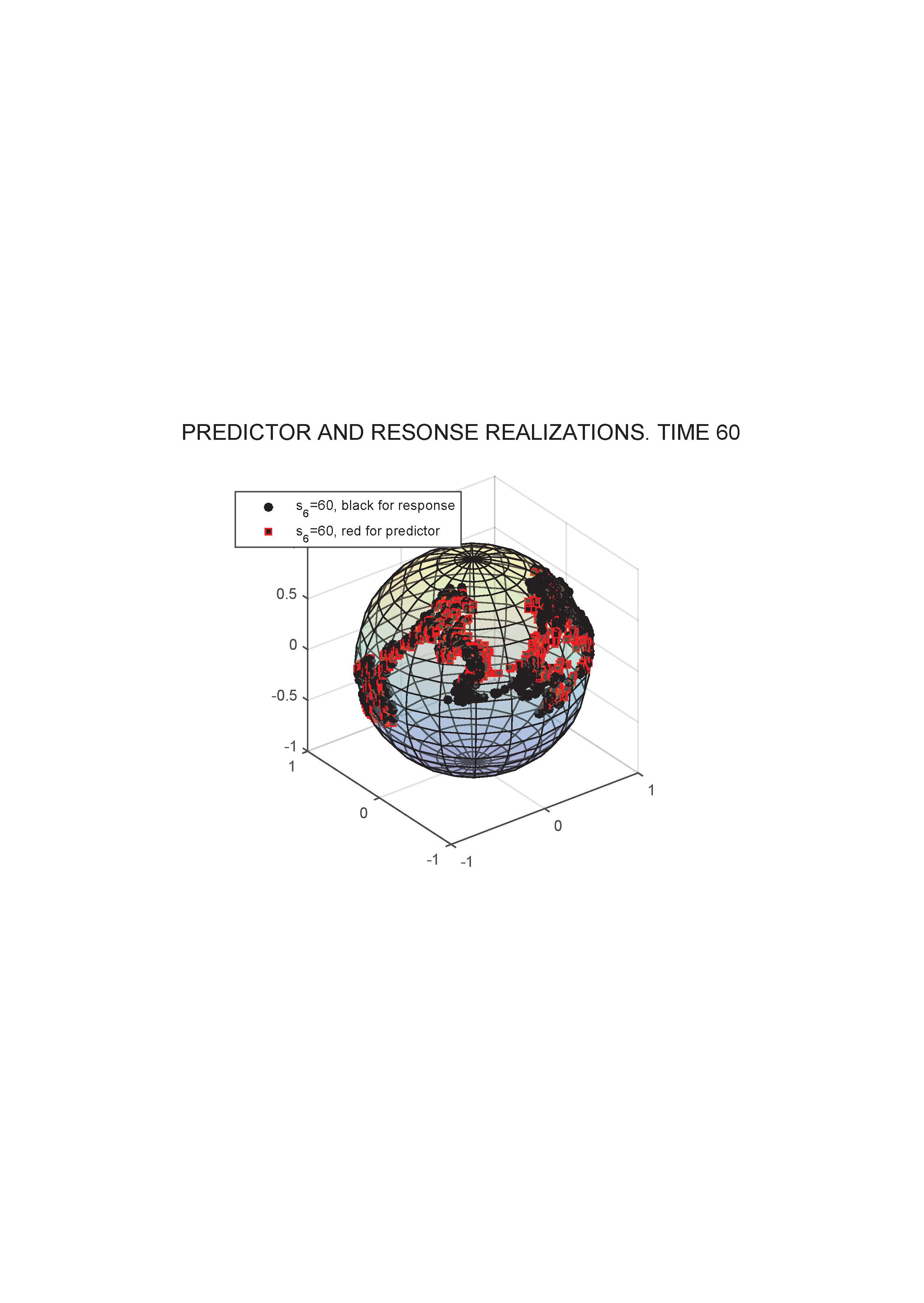}
\includegraphics[trim= 100 280 100 220, clip, width=0.32\textwidth]{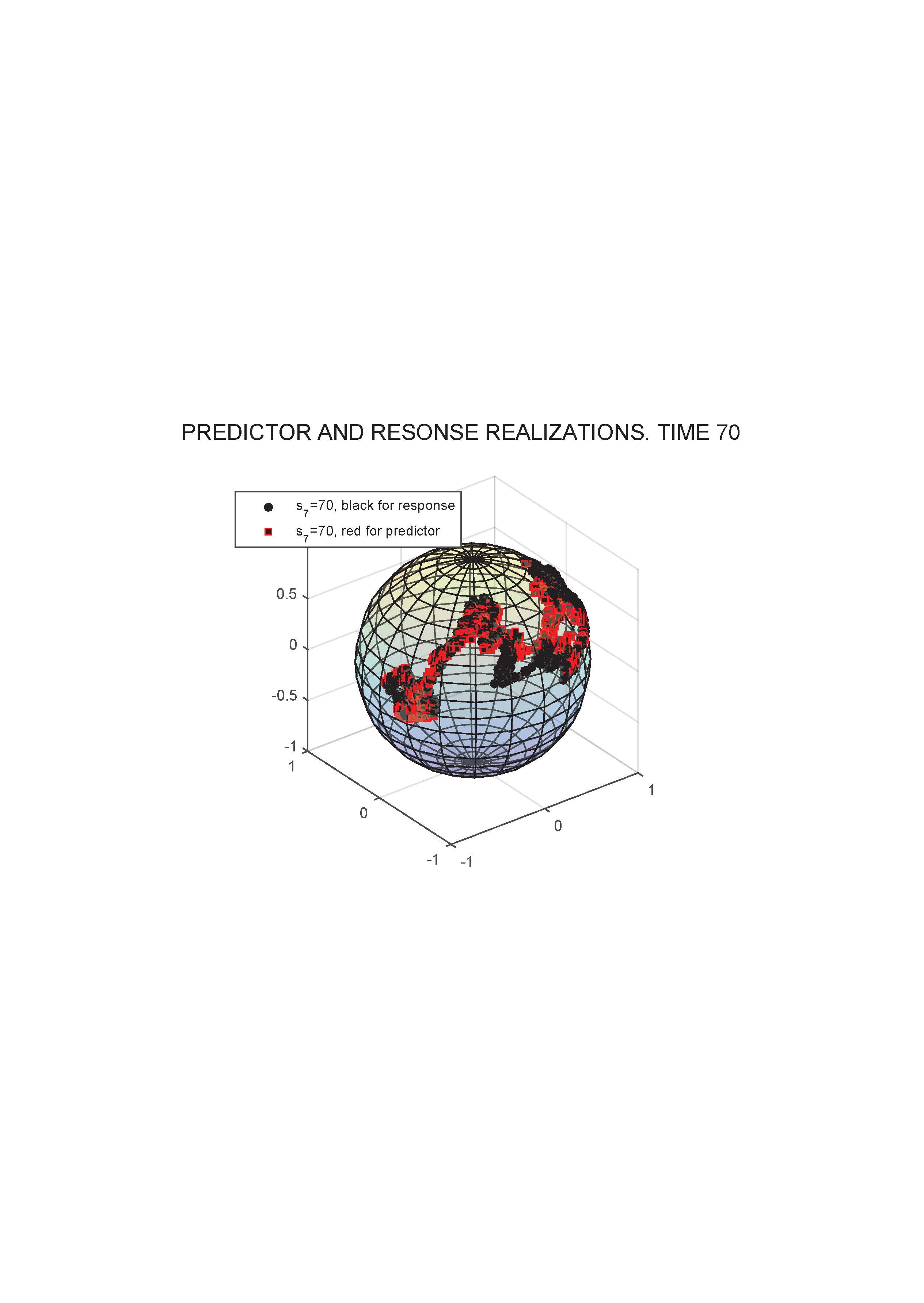}
\includegraphics[trim= 100 280 100 220, clip, width=0.32\textwidth]{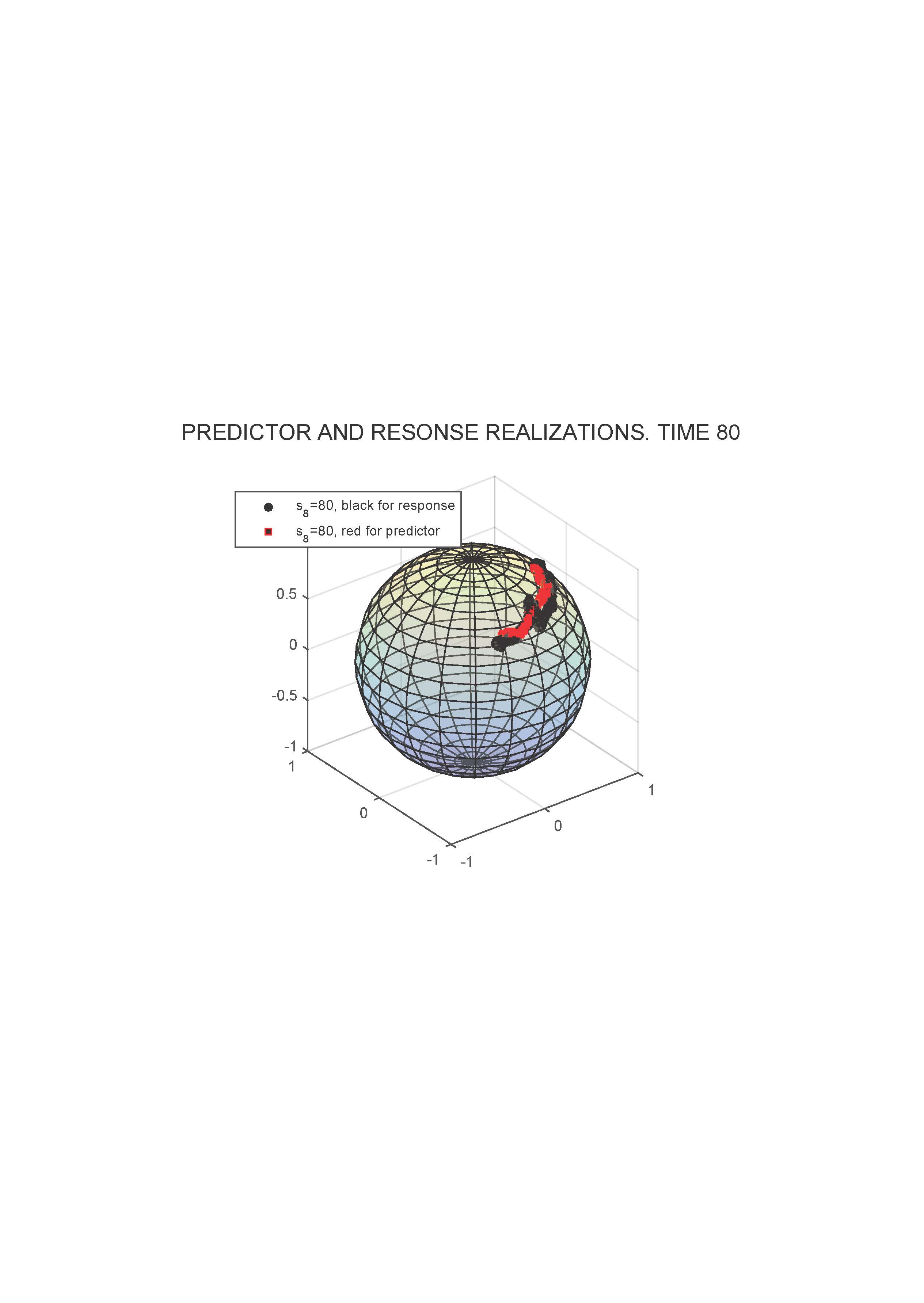}
\includegraphics[trim= 100 280 100 220, clip, width=0.32\textwidth]{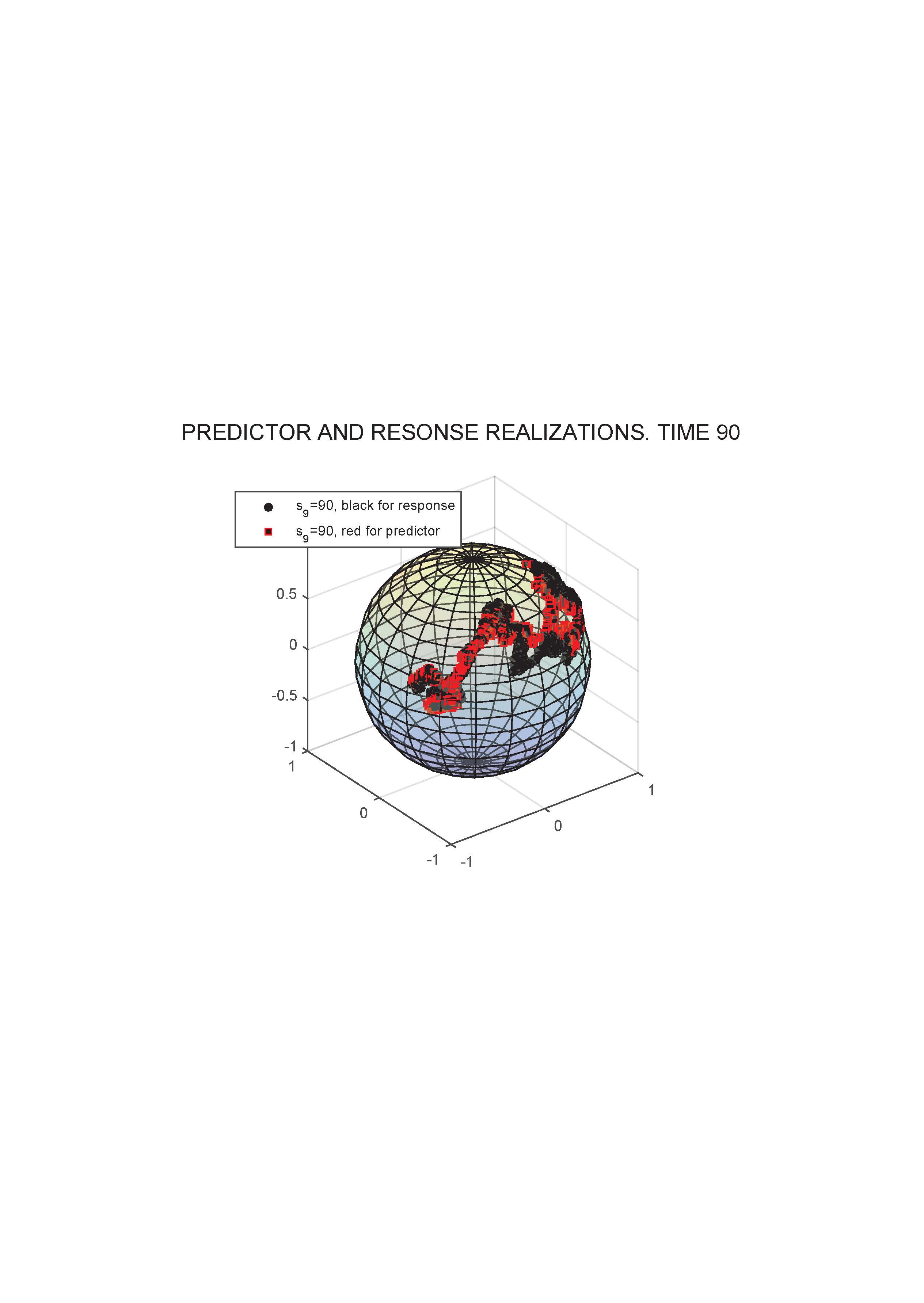}
\caption{Response curve values in black color, and its approximation in terms of the NW-type Fr\'echet curve predictor  curve values in red color,  at times $s=10, 20, 30, 40, 50, 60, 70, 80, 90$}\label{Fig:1.4}
\end{figure}

From    (\ref{ev}), the empirical extrinsic local linear Fr\'echet curve predictor  is  computed via the exponential map.  Its weights involve the empirical weighted moments  of the time-varying RFPC scores of the regressor process  (see also equations (\ref{s2302f})--(\ref{s2302h}) in Appendix \ref{app0}). Figure \ref{Fig:1.9bb}-top displays  the original response curve values, and their extrinsic local linear  Fr\'echet curve regression estimation at the bottom.

\begin{figure}
\centering
\includegraphics[ width=0.32\textwidth=0.4]{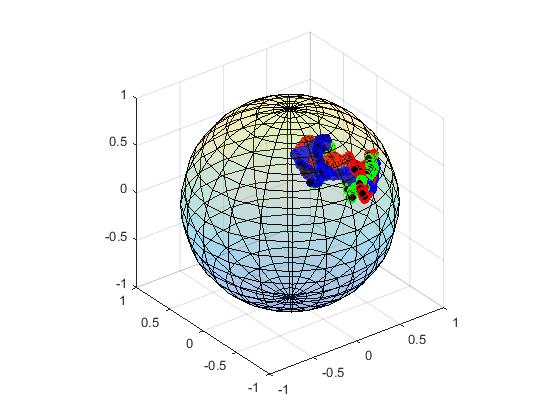}
\includegraphics[ width=0.32\textwidth=0.4]{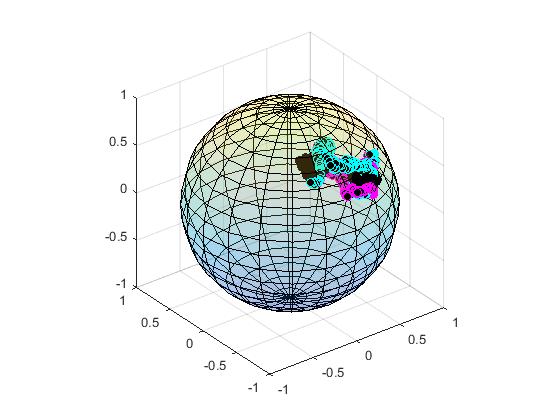}
\includegraphics[ width=0.32\textwidth=0.4]{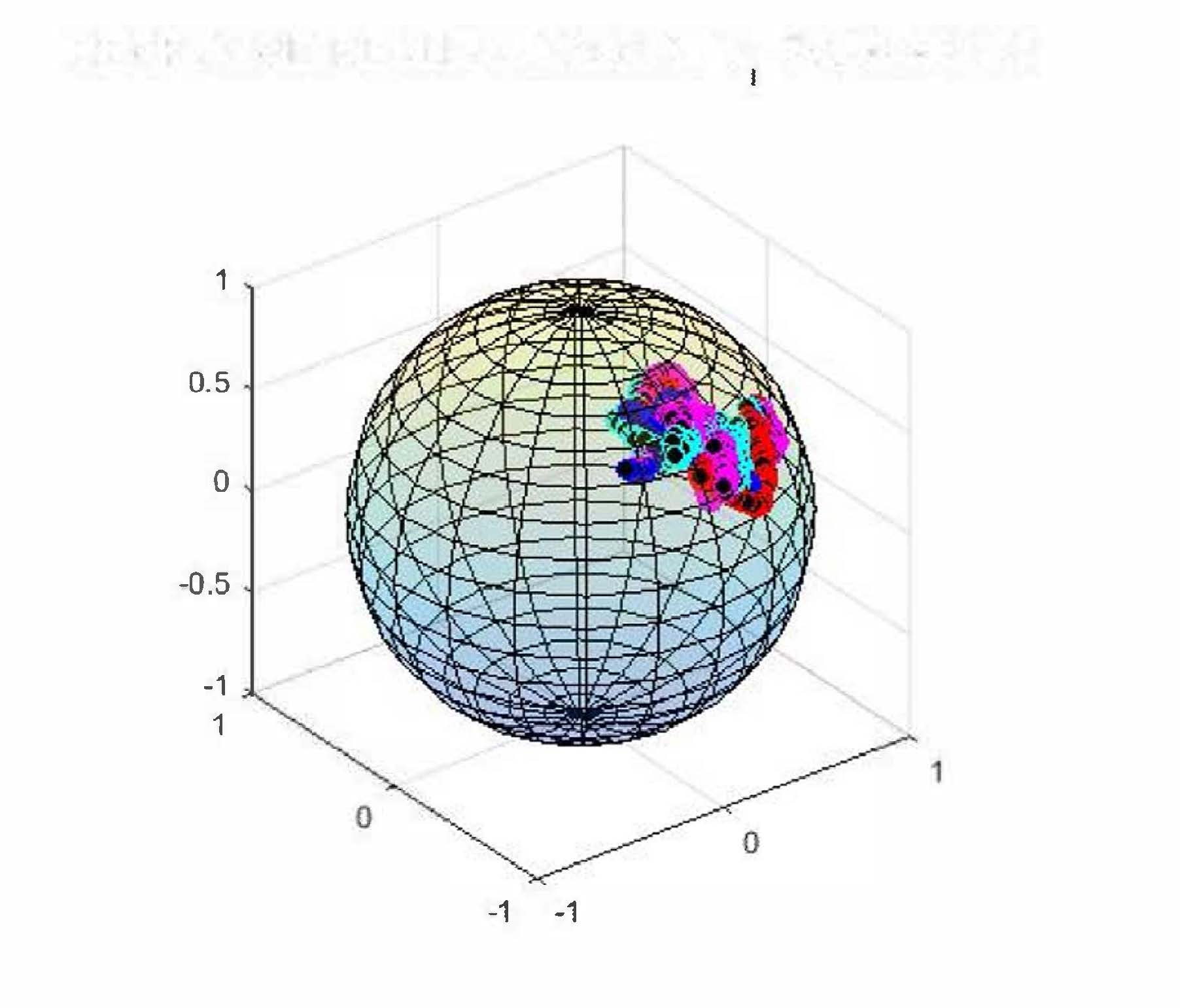}
\includegraphics[ width=0.32\textwidth=0.4]{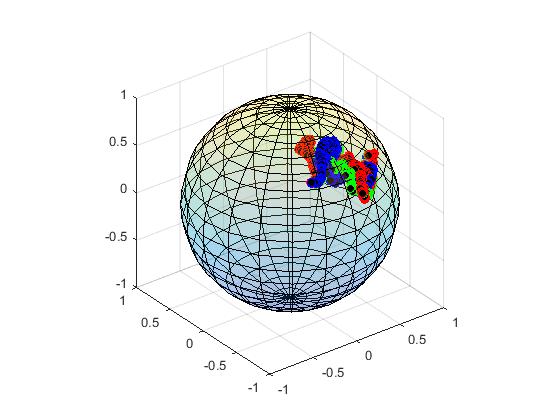}
\includegraphics[ width=0.32\textwidth=0.4]{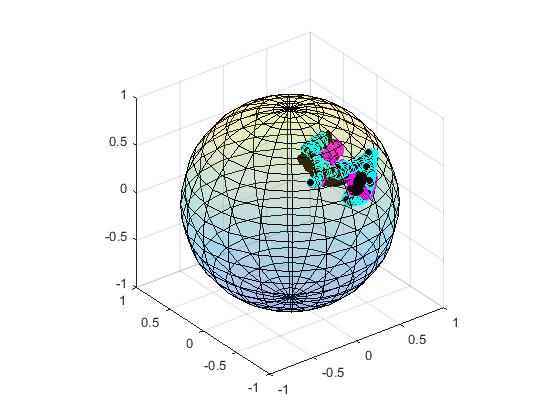}
\includegraphics[ width=0.32\textwidth=0.4]{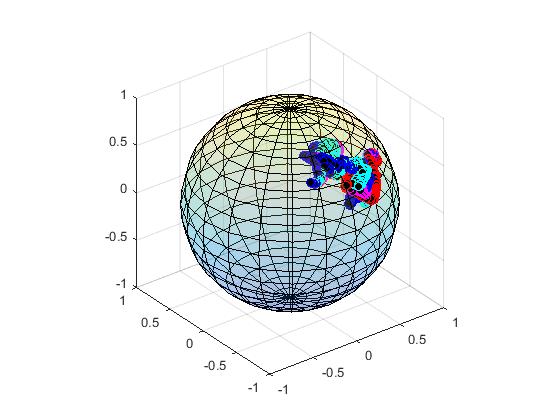}
\caption{ \emph{Response curve values from RFPCA}. At the top--left--hand--side,  response curve values  at $t=10$ (red), $t=20$ (green), $t=30$ (blue); top--center for $t=40$ (magenta), $t=50$ (cyan), $t=60$ (black), and top--right--hand--side $t=70$ (red), $t=80$ (magenta), $t=90$ (cyan) and $t=100$ (blue). \emph{Extrinsic local linear Fr\'echet curve predictions}. At the bottom--left--hand--side, extrinsic Fr\'echet curve response estimation at $t=10$ (red), $t=20$ (green), $t=30$ (blue); bottom--center for $t=40$ (magenta), $t=50$ (cyan), $t=60$ (black), and bottom--right--hand--side $t=70$ (red), $t=80$ (magenta), $t=90$ (cyan) and $t=100$ (blue)}\label{Fig:1.9bb}
\end{figure}

In our bandwidth parameter analysis of the extrinsic local linear Fr\'echet curve predictor, we have considered for $n=100,$   $B_{n}=(log(n))^{-1/\beta }$ with \linebreak $\beta =10.10,  10.11,    10.12,  10.14,    10.15,   10.16,   10.18,   10.19,   10.20,    10.22,    10.23,    10.25$,
according to the RFPCA truncation parameter $RT=[log(n)]_{+}=5.$  See   Figure \ref{Fig:1.10bb}, where the empirical mean absolute
geodesic functional errors have been computed, for these
   bandwidth parameter values in the  interval $(0.8585, 0.862).$

\begin{figure}
\centering
\includegraphics[height=0.25 \textheight, width=0.42\textwidth]{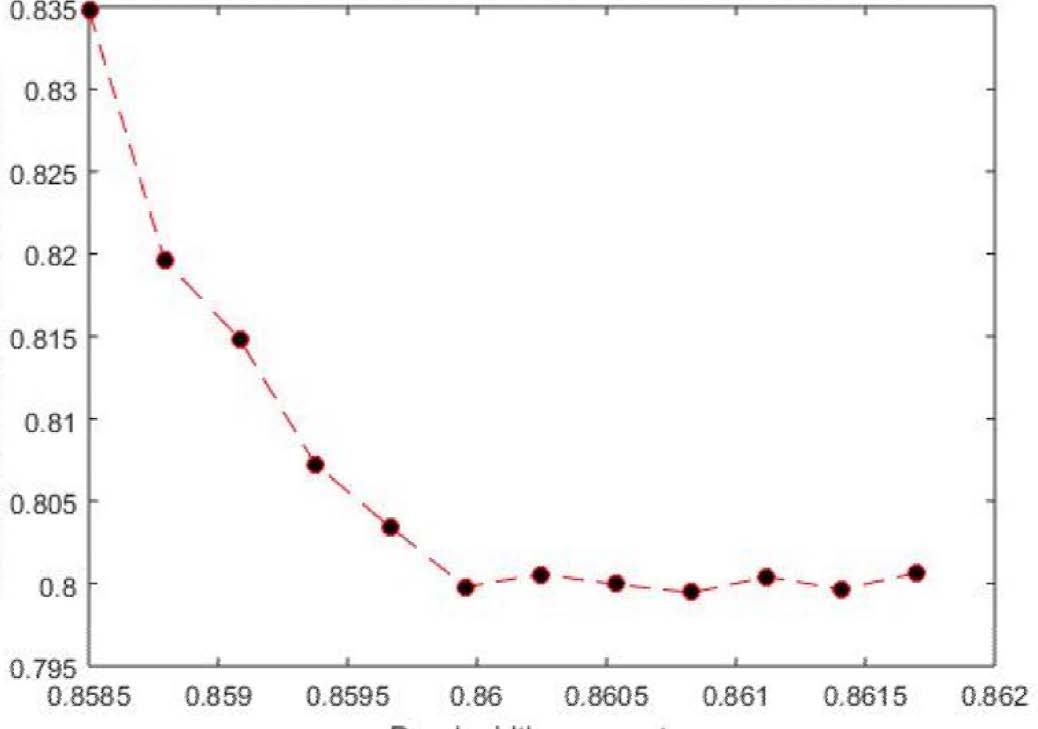}
\includegraphics[height=0.25 \textheight, width=0.42\textwidth]{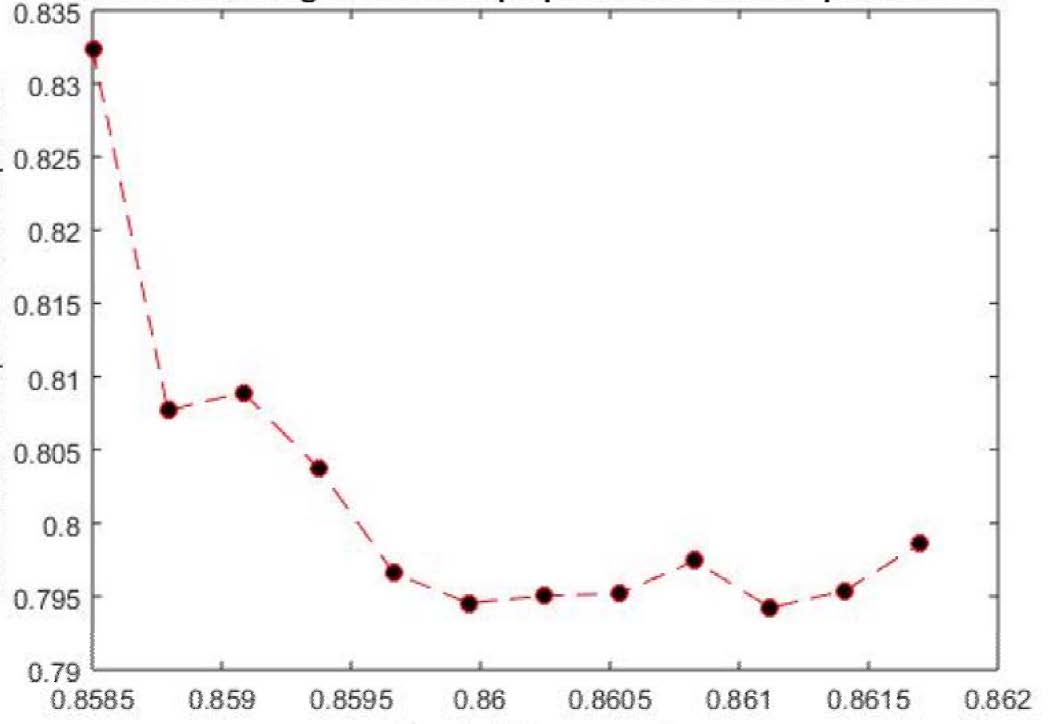}
\includegraphics[height=0.25 \textheight, width=0.42\textwidth]{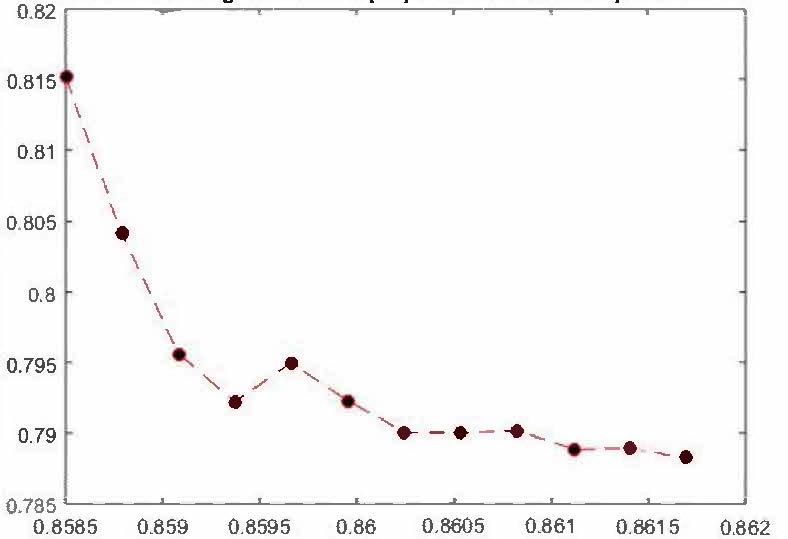}
\includegraphics[height=0.25 \textheight, width=0.42\textwidth]{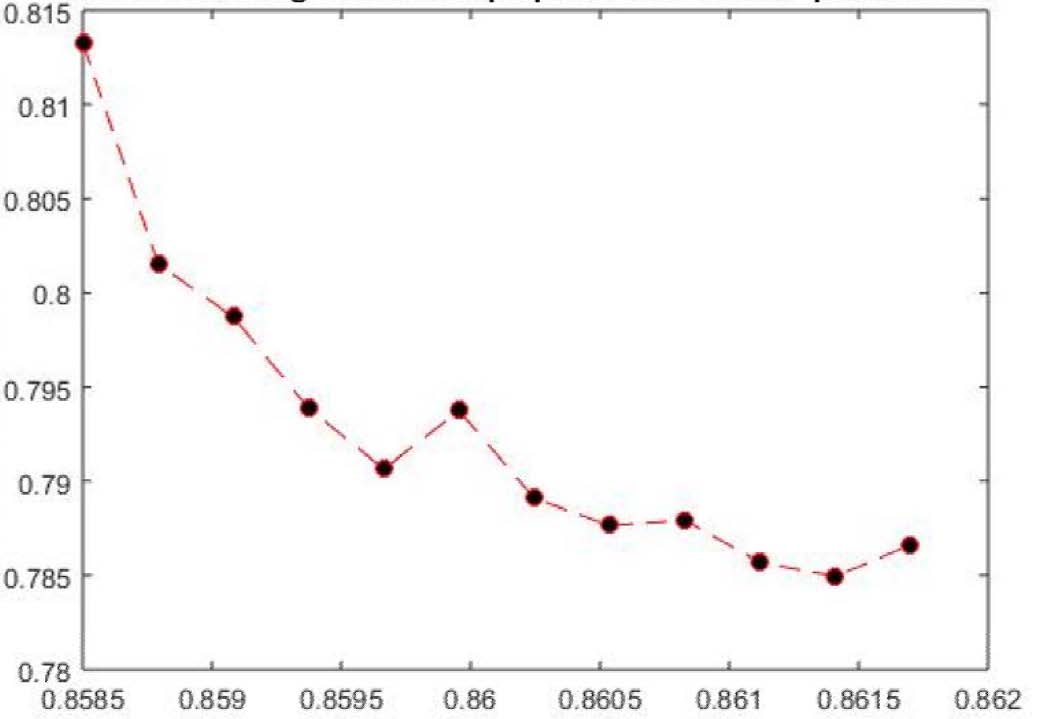}
\caption{\emph{Extrinsic local linear Fr\'echet curve prediction}. Maximum over the observed times in the functional sample of the supremum norms of the time--varying empirical mean of the  geodesic absolute functional errors, based on 50 repetitions, under the  four candidate set scenarios, characterized by shape   parameter $\theta .$
At the top--left--hand side $\theta =1.65,$  $\theta =1.70$ (top--right--hand side), $\theta = 1.80$ (bottom--left--hand--side) and $\theta =1.85$ (bottom--right--hand--side)}\label{Fig:1.10bb}
\end{figure}

We have again considered  bandwidth parameter values in the interval \linebreak $( 0.3162,
 0.6310),$ corresponding to
   $[h(n)]^{-1}=n^{-\beta},$ $\beta=0.1000,    0.1375,    0.1750,$     $0.2125,    0.2500,$ in the implementation of the  intrinsic local linear Fr\'echet curve  predictor (\ref{elfb}). For $\beta= 0.2500,$ i.e., $[h(n)]^{-1}=[0.3162]_{-},$ and
   $[h(n)]_{+}=4,$ Figure \ref{Fig:1.9} displays  its spherical functional predictions (red color), approximating the observed original curve values of the response (black color) at different times.

\begin{figure}
\centering
\includegraphics[ width=0.32\textwidth]{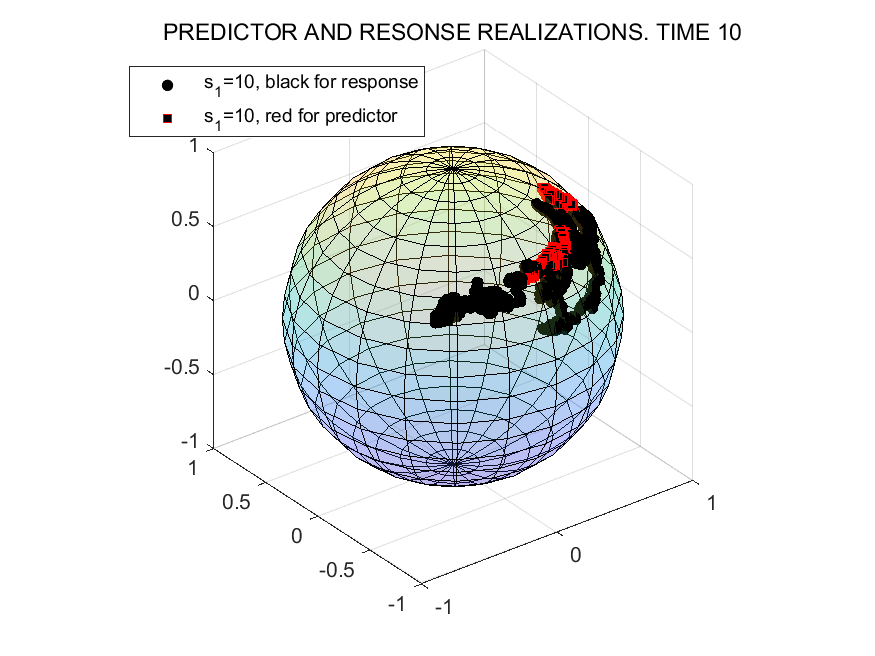}
\includegraphics[ width=0.32\textwidth]{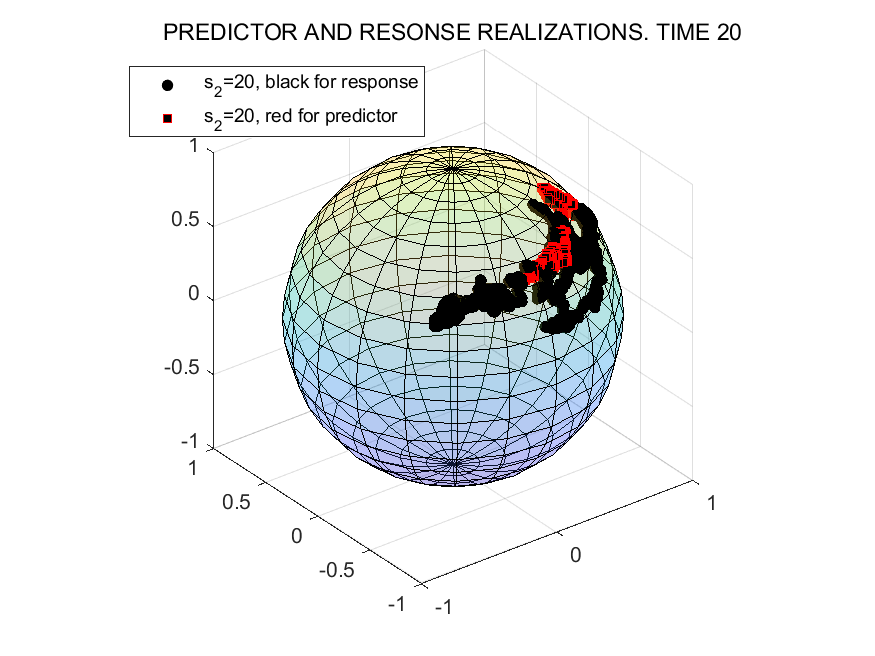}
\includegraphics[ width=0.32\textwidth]{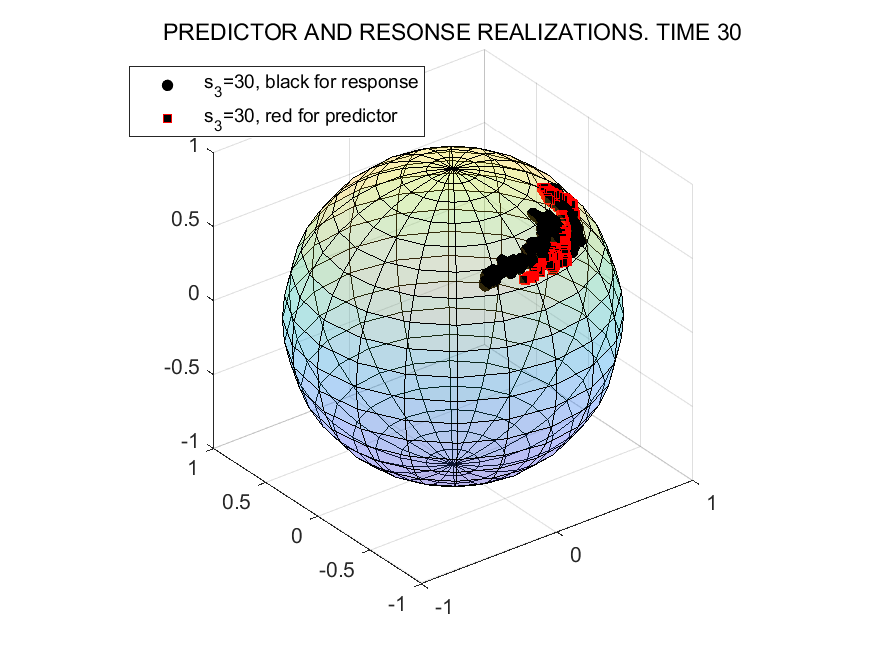}
\includegraphics[ width=0.32\textwidth]{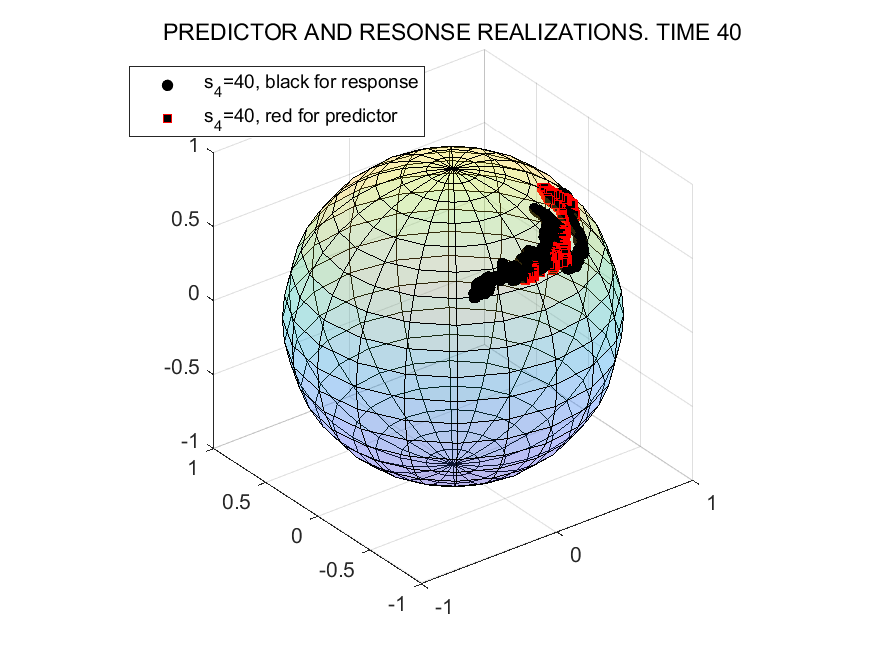}
\includegraphics[ width=0.32\textwidth]{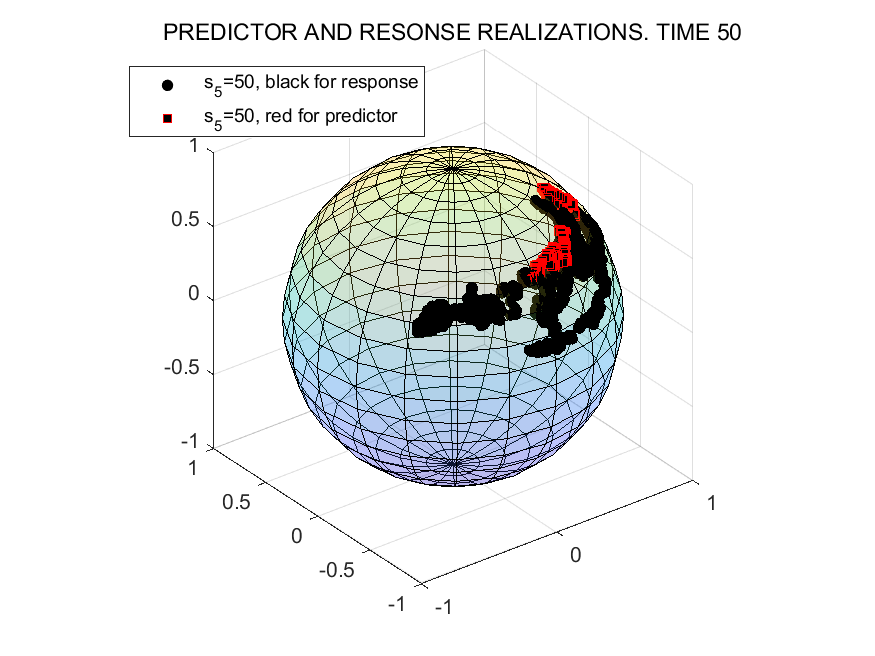}
\includegraphics[ width=0.32\textwidth]{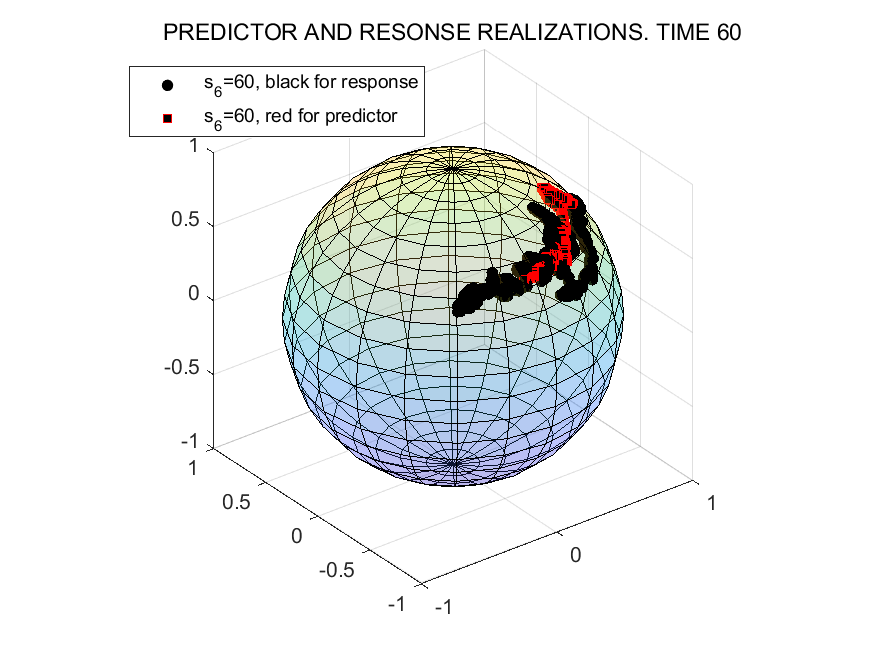}
\caption{Intrinsic local linear  Fr\'echet curve predictor in red, and original curve  response values  in black for times 10, 20, 30, 40, 50, 60}\label{Fig:1.9}
\end{figure}

\subsection{Quantitative analysis}
\label{QA}
Our quantitative  assessment of the predictive performance of the  three  local Fr\'echet curve predictors implemented  is based on
 the  intra- and inter- curve  variability displayed by  their  quadratic geodesic  errors.
Specifically,  Figures \ref{Fig:1.5},  \ref{Figex}  and \ref{Fig:1.10}    display at the left-hand-side the empirical distribution of the mean  intra-curve variability  of the geodesic functional quadratic errors for the NW-type Fr\'echet curve predictor,  and the  extrinsic and intrinsic  local linear Fr\'echet curve predictors,  respectively. The mean  inter-curve variability  of these errors is plotted at the right-hand-side of these figures for such  curve predictors.  Our  quantitative comparative analysis is based on the empirical mode and range of the    mean  intra- and inter- curve variabilities, which are  displayed in   Table \ref{tabS62} for the mean bandwidth value $[\overline{h}(n)]^{-1}=n^{-\overline{\beta}}= 0.4601,$ with $\overline{\beta}=\frac{1}{5}\left[0.1000 +  0.1375+ 0.1750+0.2125+0.2500\right].$ One can observe that the  worst finite-sample  performance corresponds to the extrinsic local linear Fr\'echet curve predictor displaying the largest empirical range and  mode values  of the mean intra- and inter-curve   variability of the quadratic geodesic errors.  Note that the geodesic errors of the  intrinsic local linear Fr\'echet curve predictor display less intra-curve variability than the NW-type functional predictor. Thus, this predictor  better fits model  local regularity, being more robust under temporal changes in the support of the curve data.
  While NW-type  curve predictor displays less inter-curve  variability,   better supporting the   homogeneity assumption of the functional objects through  time.

\begin{figure}
\centering
\includegraphics[width=0.45\textwidth]{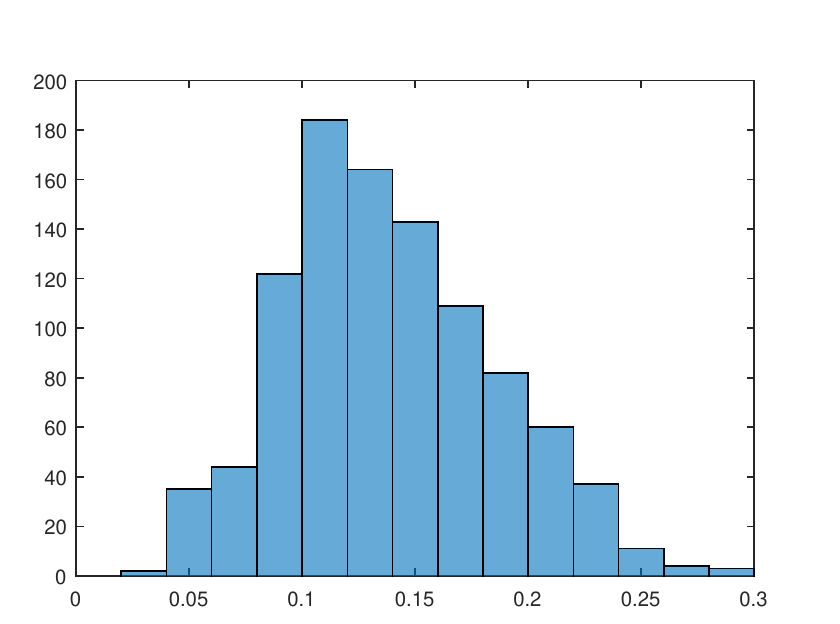}
\includegraphics[width=0.45\textwidth]{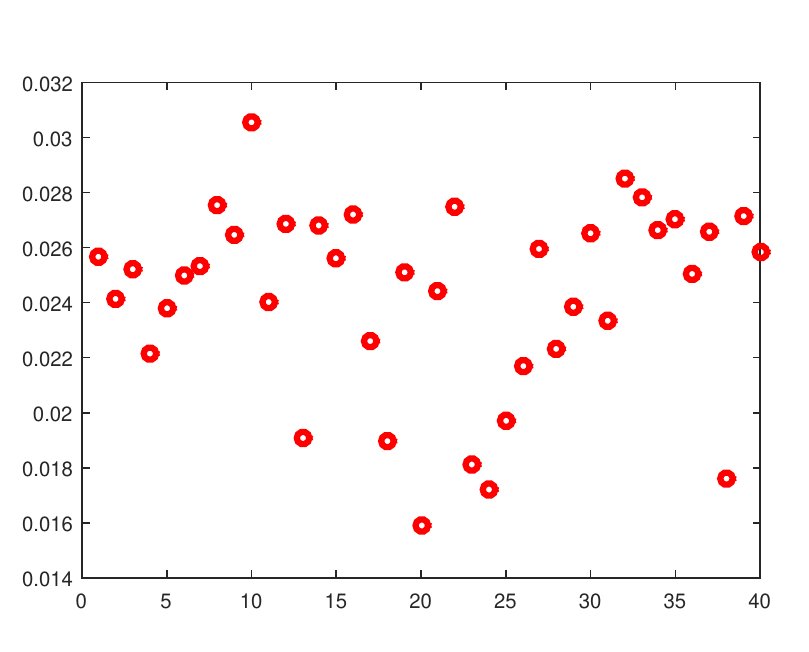}
\caption{\emph{NW-type local Fr\'echet curve predictor}.  Histogram of the empirical mean intra-curve variability  of the quadratic geodesic functional   errors  at the left-hand side, and the  empirical mean inter-curve variability values  of the quadratic geodesic functional   errors between $40$ sampled times    at the right-hand side}\label{Fig:1.5}
\end{figure}

 \begin{figure}
 \centering
\includegraphics[width=0.45\textwidth]{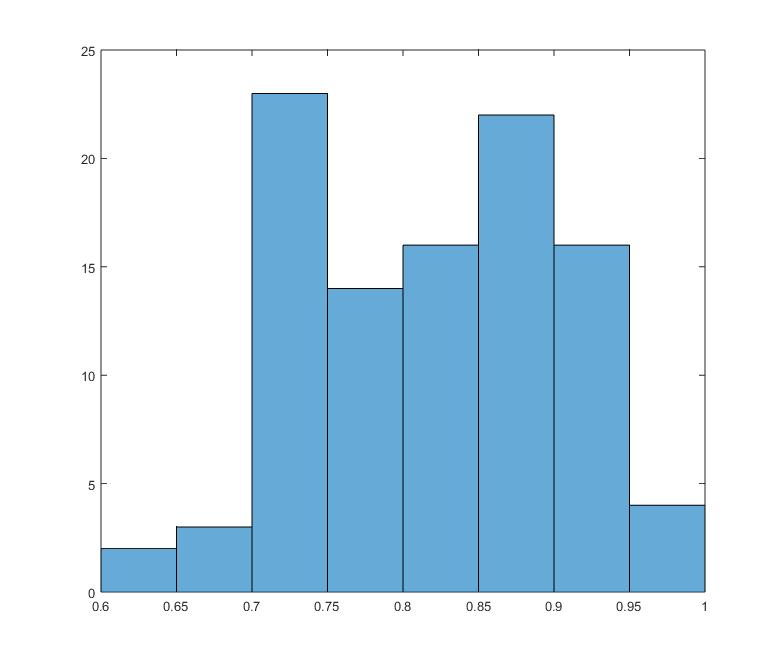}
\includegraphics[width=0.45\textwidth]{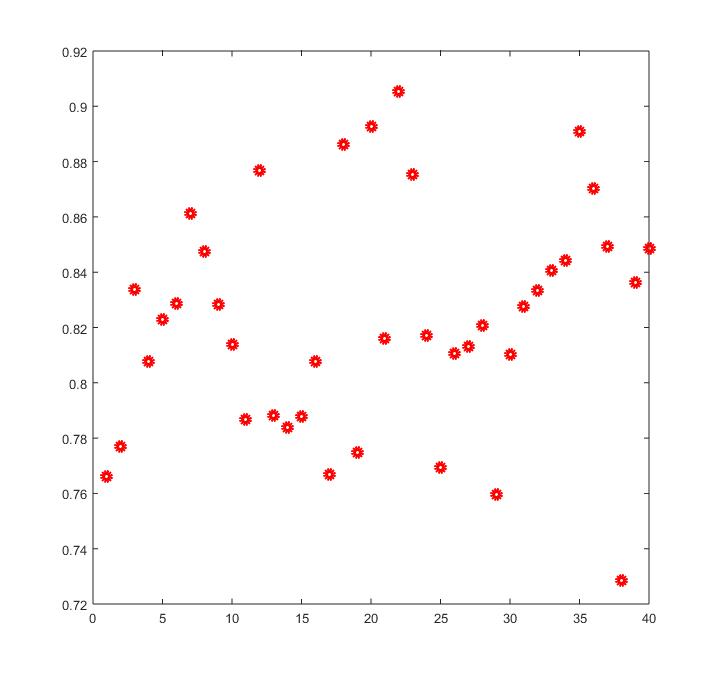}
\caption{\emph{Extrinsic  local linear Fr\'echet curve predictor}.  Histogram of the empirical mean intra-curve variability  of the quadratic geodesic functional   errors  at the left-hand side, and the  empirical mean inter-curve variability values  of the quadratic geodesic functional   errors between $40$ sampled times    at the right-hand side}\label{Figex}
\end{figure}

\begin{figure}
\centering
\includegraphics[width=0.45\textwidth]{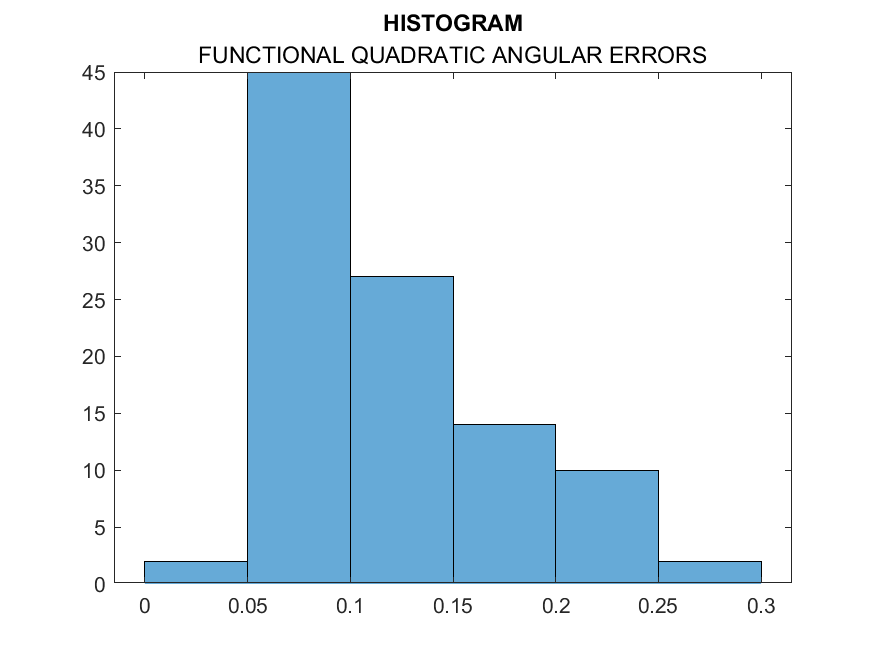}
\includegraphics[width=0.45\textwidth]{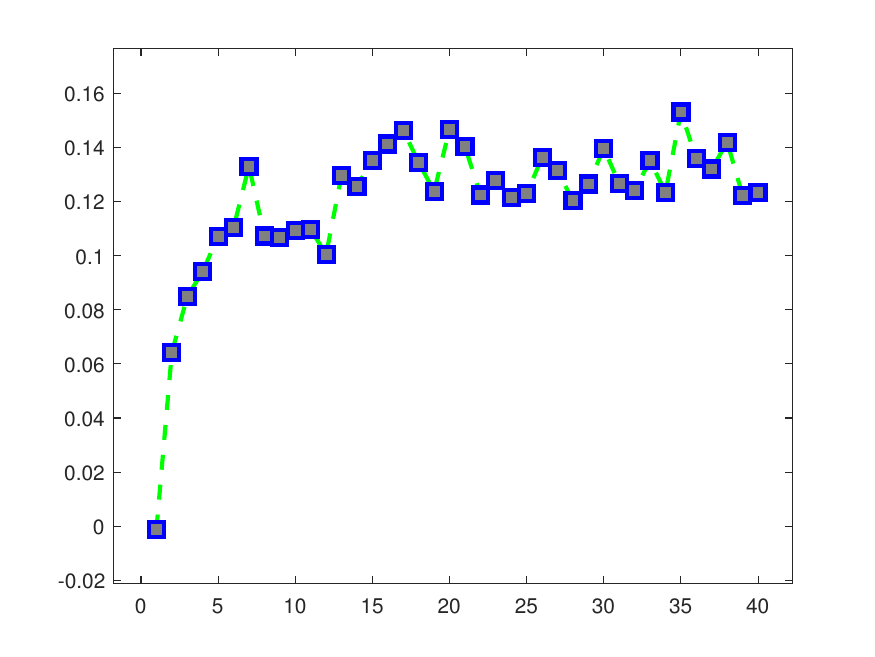}
\caption{\emph{Intrinsic local linear Fr\'echet curve predictor}. Histogram of the empirical mean intra-curve variability  of the quadratic geodesic functional   errors  at the left-hand side, and the  empirical mean inter-curve variability values  of the quadratic geodesic functional   errors between $40$ sampled times    at the right-hand side}\label{Fig:1.10}
\end{figure}

\vspace*{1cm}

\begin{table}%[width=.9\linewidth,cols=4,pos=h]
\begin{tabular}{|l|l|l|l|}
%\toprule
\hline
   &  &   \emph{Empirical Range} & \emph{Empirical Mode} \\
   \hline 
%\midrule
NW-type & EMIACV& $(0.030, 0.300)$  & $0.1125$  \\  & EMIECV & $(0.016,0.031)$ & $0.0260$\\
%\midrule
 ELLFCP & EMIACV &  $(0.600, 1.000)$ &  $0.7255$\\  & EMIECV  & $(0.720,0.920)$ & $0.8150$\\
%\midrule
ILLFCP & EMIACV & $(0.020, 0.300)$ &  $0.0550$\\  &  EMIECV  &  $(0.010,0.160)$ & $0.1300$\\
\hline
%\bottomrule
\end{tabular}
\caption{\textbf{Performance quantitative  assessment}.
Empirical mode and range of the distribution of the mean  intra- (EMIACV), and inter- (EMIECV) curve variability   of the geodesic functional quadratic errors for the NW-type, extrinsic (ELLFCP), and intrinsic  (ILLFCP)  local linear Fr\'echet curve predictors}\label{tabS62}
\end{table}

\section{Real data analysis}
\label{rda}
Data  from NASA's National Space Science Data Center are available in the period 02/11/1979--06/05/1980, recorded every half second, and correspond to the first satellite NASA's MAGSAT spacecraft (NMS), which orbited the Earth  every $88$ minutes during seven months at around $400$ km altitude.   Functional samples  of size $82,$ $84,$  $83,$  $84,$ $85,$  $72,$ and  $52$ are available  during  the  days $3, 4, 5$  of each month in the period 02/11/1979--06/05/1980,  respectively. The elements of these samples are discretely observed at  $6000$ temporal nodes, reflecting a trajectory of $50$ minutes. All the sample functional objects have been evaluated at the same number, $6000$, of consecutive temporal nodes. The starting time  is chosen randomly, ensuring certain  heterogeneity and  representativeness.  We have removed the sample information corresponding to May, 1980, because of the missing data. Thus, the functional samples analyzed correspond to the period November $1979$ - April $1980,$ providing  sample trajectories  of Earth NMS (covering approximately between $13$ and $16$ orbits
per day).  The  spacecraft and Earth's magnetic  field share the   azimuthal angle. Hence, the relative position of the Earth's magnetic  field is  given by its time-varying polar angle.  Such a temporal angle  characterizes  the  Earth's magnetic  field components defining our curve response evaluated in the sphere.  The observed  trajectories  of the  NMS  then provide our regressor curve observations.
  Figure \ref{Fig:2.1b}  displays six functional elements  of the bivariate curve sample corresponding to times  $t=1, 15, 29, 43, 57$ and $71,$ for   November, 1979 (see also Figures \ref{Fig:2.1}  and \ref{Fig:2.2}   in   Appendix \ref{app1}, for  months December, 1979, and January--April, 1980).

The 5--fold cross validation technique is implemented to assess the predictive performance of the NW--type, and intrinsic and extrinsic local linear Fr\'echet curve predictors.
The inter-curve  temporal correlation is negligible compared with the intra-curve temporal correlation, since,  for every month,
each functional object of the global sample  contains almost two complete cycle, or Earth orbits (see Figures   \ref{Fig:2.1b}, \ref{Fig:2.1} and \ref{Fig:2.2}). Under this weak-correlated curve scenario,   at each iteration of the 5-fold cross validation technique, a random splitting into training and target curve samples is then considered  ignoring inter-curve correlation, following \cite{Marzio.14}.

Given the nature of the analyzed weak-correlated  curve data set (see Section \ref{tceff}),   the exponential map of the $\mathbb{H}$-valued local predictor  projected into coarser scales, defining  the large scale approximation of the   extrinsic local linear Fr\'echet curve predictor,   does not stay on the manifold. Hence, only high resolution levels can be  reconstructed (see also Section \ref{sebwtp}). Thus, its 5--fold cross validation analysis is restricted to the  resolution levels $k=7,8, 9,$ corresponding to the projection into the subspace $\bigoplus_{k=7}^{9}\mathcal{H}_{k},$ with $\mathcal{H}_{k},$ $k=7,8,9,$ being the eigenspaces respectively generated by  the eigenfunctions $\phi_{k},$  $k=7,8,9,$ with  $$\boldsymbol{\phi}_{k}(t)=\left(\sin(\pi(k+1)t), \sin(\pi(k+1)t),\sin(\pi(k+1)t)\right),\quad t\in \mathbb{T}.$$
\noindent (see  Table \ref{tab1} in Section \ref{QA2} for more details).

 Figure \ref{Fig:2.5b} displays 5--fold cross--validation  geodesic functional  absolute   errors reflecting the predictive performance of the NW-type  curve predictor (\ref{NWF}), for   bandwidths 0.2, 0.225, 0.25, 0.275 and 0.3, based on   November 1979 sample
  (see also Figures \ref{Fig:2.5} and \ref{Fig:2.6}  in  Appendix \ref{app1},  for the  period December 1979--January 1980, and February--April 1980,  respectively).   The predictive performance of the intrinsic local linear Fr\'echet curve predictor (\ref{empv}), based on the same sample, is displayed in Figure \ref{Fig:2.3b}.
 Figures \ref{Fig:2.3} and \ref{Fig:2.4} in  Appendix \ref{app1} respectively plot the numerical results obtained for  December 1979--January 1980, and  February 1980--April 1980. We observe similar  yellow (highest pointwise values of the   geodesic curve absolute errors) and  blue (lowest pointwise values of  the geodesic curve  absolute  errors) patterns across months at the five  contourplots displayed for both intrinsic local curve predictors.
\begin{figure}
\centering
\includegraphics[width=0.28\textwidth]{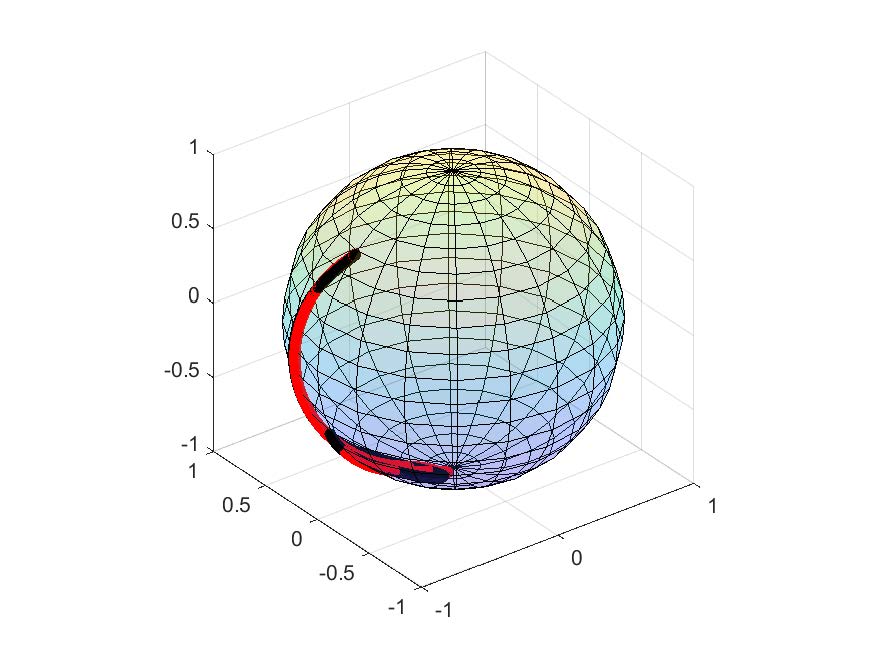}
\includegraphics[width=0.28\textwidth]{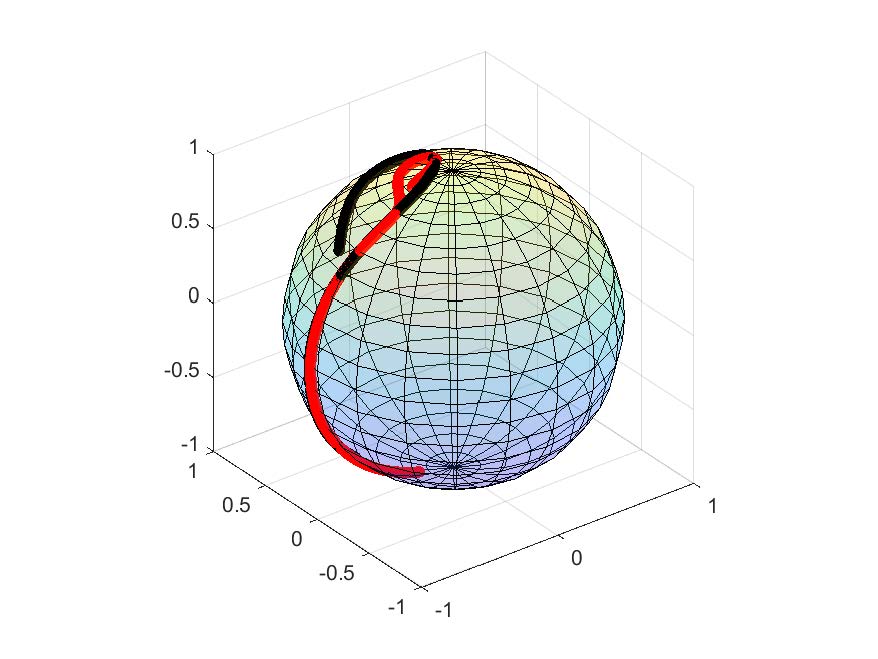}
\includegraphics[width=0.28\textwidth]{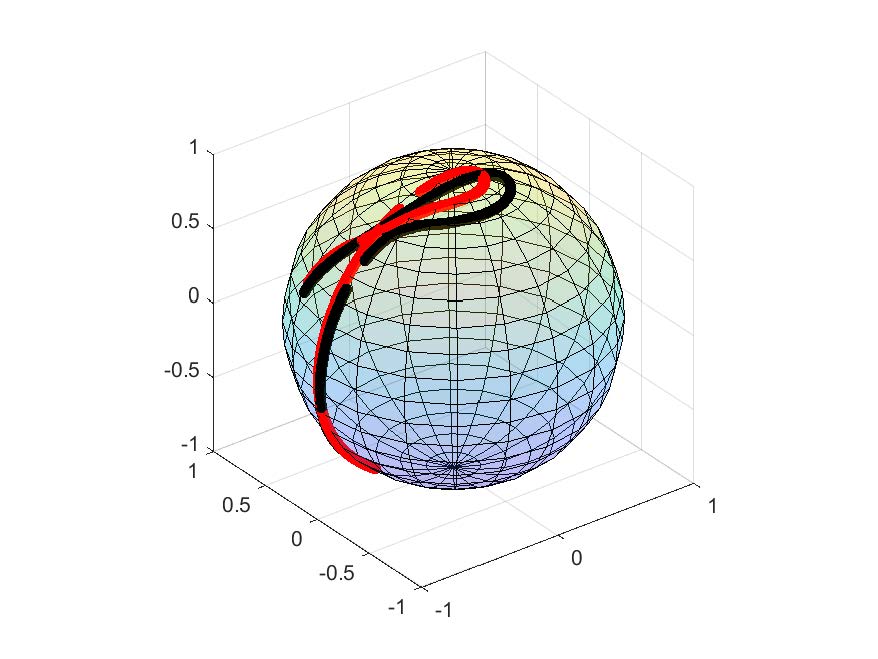}
\includegraphics[width=0.28\textwidth]{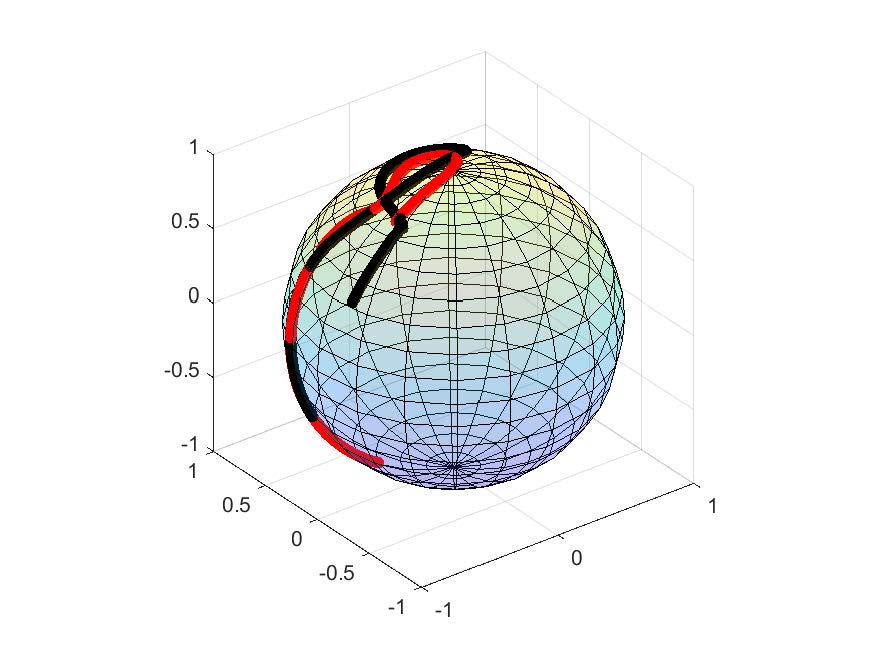}
\includegraphics[width=0.28\textwidth]{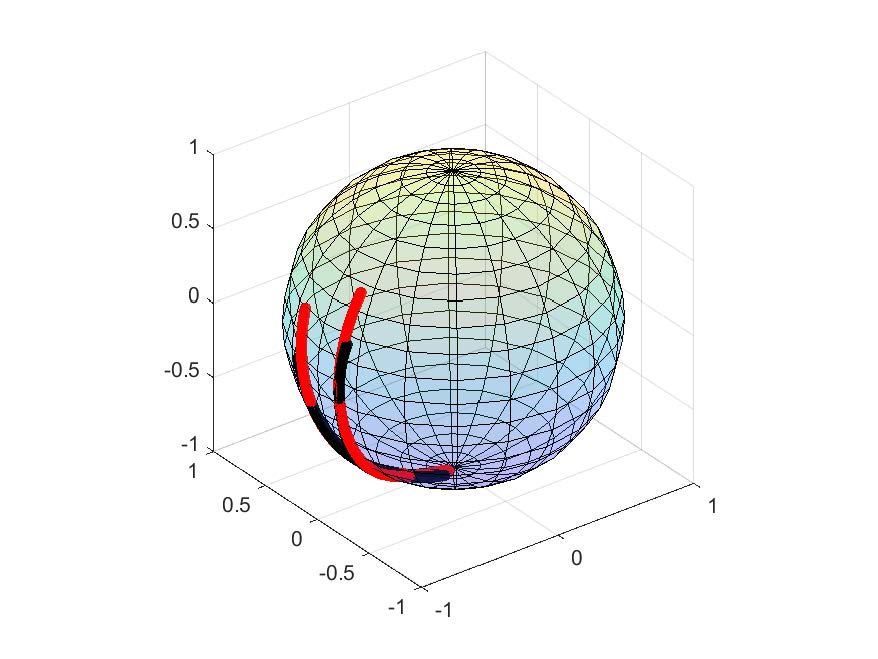}
\includegraphics[width=0.28\textwidth]{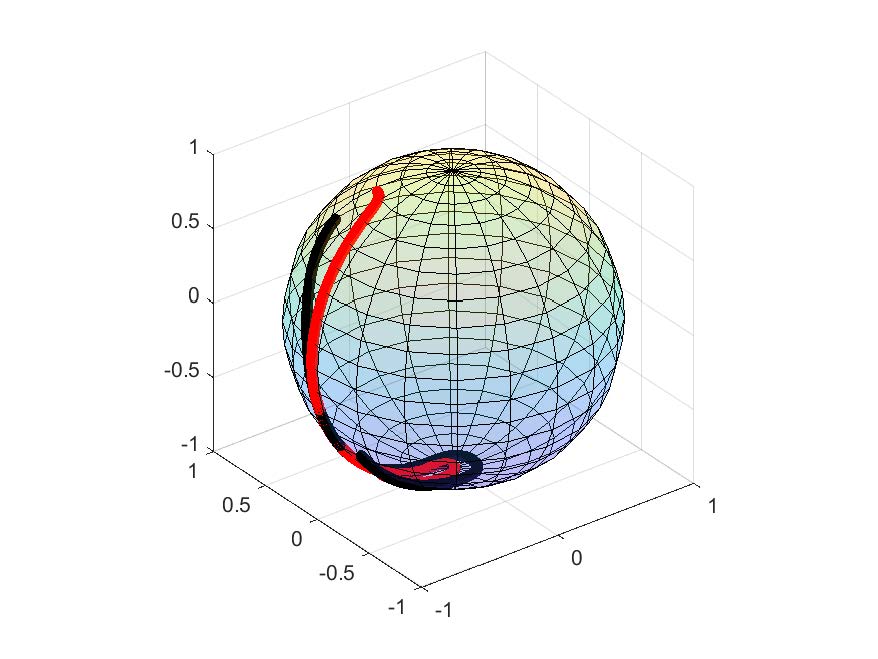}
\caption{\emph{Spherical bivariate curve data. Sample functional  elements at times $t=1, 15, 29, 43, 57, 71,$ during
 November 1979, are displayed}}
\label{Fig:2.1b}
\end{figure}

   \begin{figure}
   \centering
\includegraphics[width=0.28\textwidth]{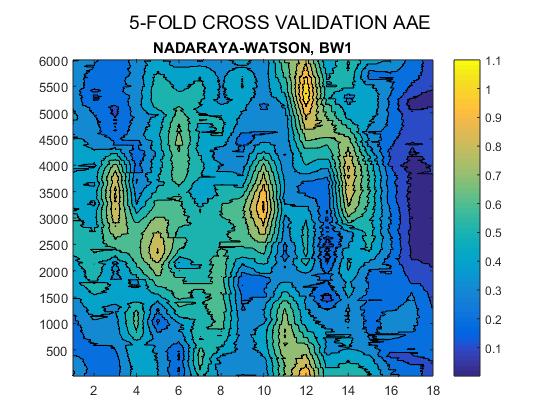}
\includegraphics[width=0.28\textwidth]{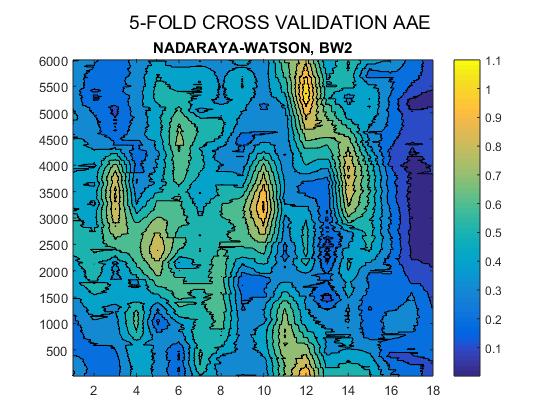}
\includegraphics[width=0.28\textwidth]{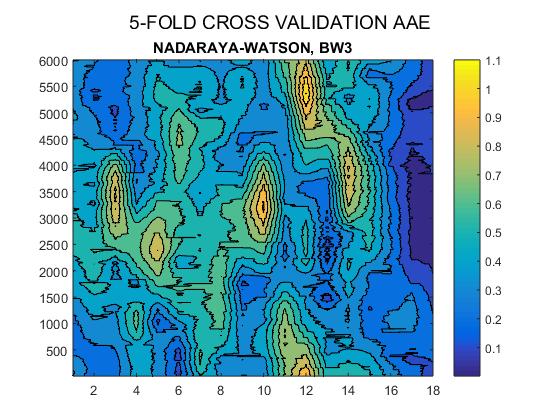}
\includegraphics[width=0.28\textwidth]{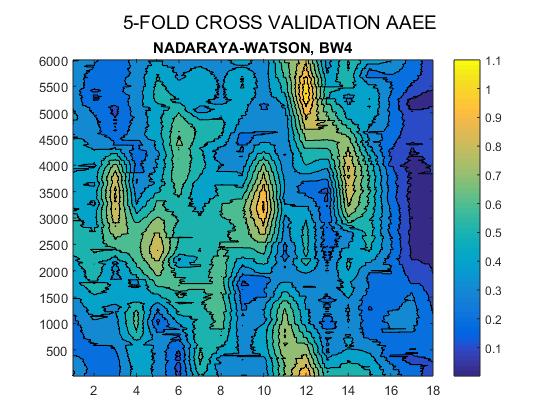}
\includegraphics[width=0.28\textwidth]{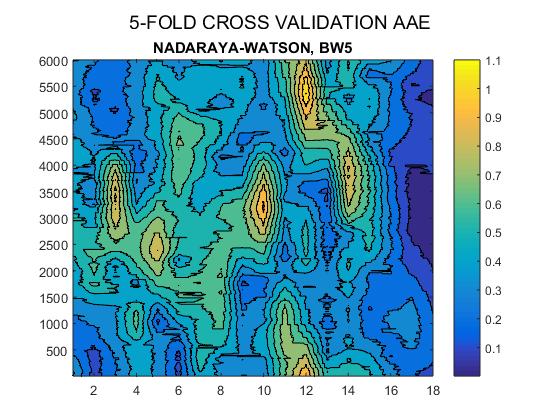}\\
\caption{  \emph{NW--type local   Fr\'echet  regression  curve predictor}.  Contour plots reflect the pointwise values of the  5-fold cross-validation geodesic  absolute  curve errors for bandwidths $BW1= 0.2000,$   $BW2= 0.2250,$    $BW3=0.2500$ (top), and     $BW4=0.2750,$   $BW5= 0.3000$ (bottom) during November 1979}\label{Fig:2.5b}
\end{figure}

\begin{figure}
\centering
\includegraphics[width=0.28\textwidth]{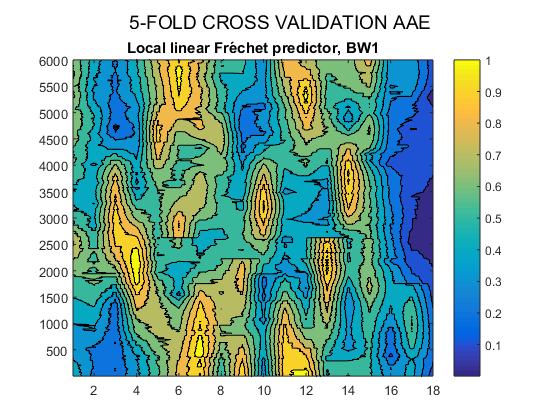}
\includegraphics[width=0.28\textwidth]{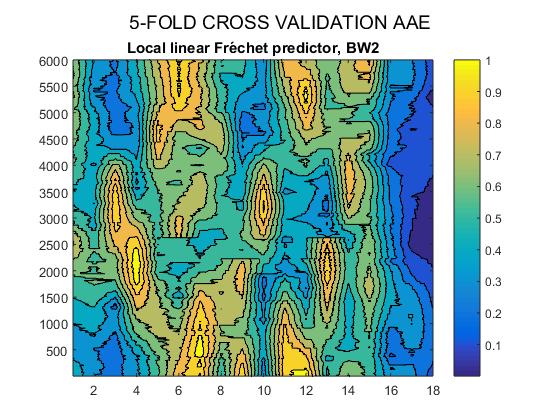}
\includegraphics[width=0.28\textwidth]{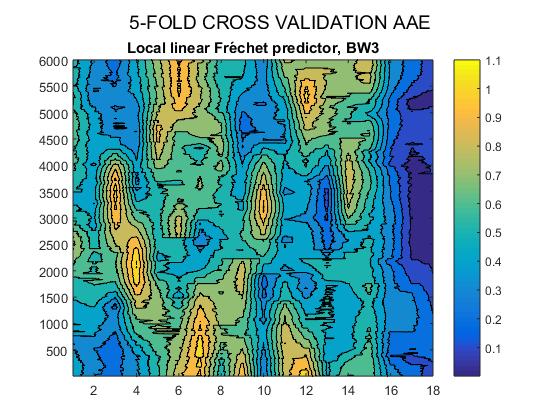}
\includegraphics[width=0.28\textwidth]{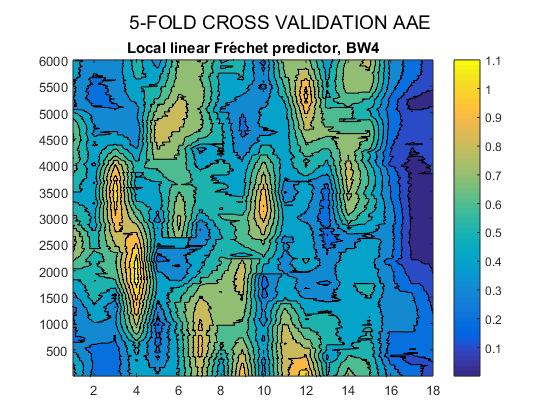}
\includegraphics[width=0.28\textwidth]{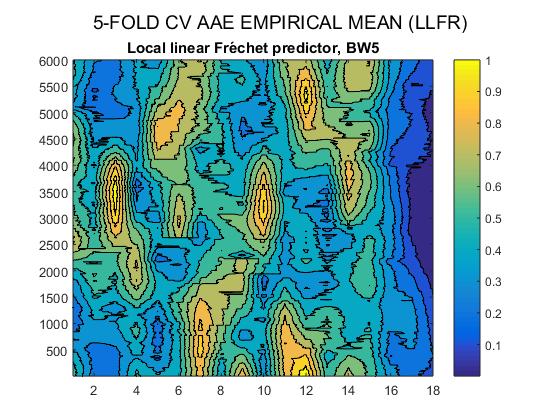}\\
\caption{\emph{Intrinsic  local linear  Fr\'echet  regression curve predictor}. Contourplots reflect the pointwise values of the  5-fold cross-validation geodesic  absolute  curve errors for bandwidths $BW1= 0.2000,$  $ BW2= 0.2250,$    $BW3=0.2500$   (top), and $BW4=0.2750,$   $BW5= 0.3000$ (bottom) during  November 1979}\label{Fig:2.3b}
\end{figure}
\subsection{Quantitative analysis}
\label{QA2}

 This section    analyzes  the numerical results plotted in the previous section for quantitative  assessment of the predictive performance of the three local Fr\'echet curve predictors  compared.  Geometrical assumptions underlying the methodology are also discussed in relation to this data set.

As commented in the previous section, the numerical results for the extrinsic local linear Fr\'echet curve predictor are restricted to  its projection into the subspace $\bigoplus_{k=7}^{9}\mathcal{H}_{k}$ in the ambient Hilbert space $\mathbb{H}.$
  Table \ref{tab1} reflects  the interaction of the RFPCA double truncation parameter, and  the bandwidth parameters tested in the intervals $(0.85, 0.86), (0.4, 0.8), (0.3, 0.7), (0.2, 0.4)$ and $(0.1, 0.2).$ Specifically, for  $j=1,2,3,$ the following quantities   $$\overline{\|CVE(Y_{j})\|}_{\infty}=\frac{1}{5}\sum_{i=1}^{5}\|\mbox{SE}(Y_{j})\|_{\infty,i},$$  \noindent   are displayed  in Table \ref{tab1}.
Here,
\begin{eqnarray}
\|\mbox{SE}(Y_{j})\|_{\infty,i}=\max_{s\in \{ s_{1},\dots, s_{n_{i}}\}}\left\|\log_{\widehat{\mu}_{Y_{0},\mathcal{M}}(t)}\left(\widehat{Y_{s,j}}^{(n_{i})}\right)-\log_{\widehat{\mu}_{Y_{0},\mathcal{M}}(t)}\left(
Y_{s,j}\right)\right\|^{2}_{\mathbb{H}},
\label{not}
\end{eqnarray}

\noindent with  $\{ s_{1},\dots, s_{n_{i}}\}$ being the time index set  defining the curve response target sample, for each iteration  $i=1,\dots, 5,$ of the $5$--fold cross validation technique.  We denote by  $\widehat{Y_{s,j}}^{(n_{i})}$ the double truncated extrinsic local linear Fr\'echet curve  prediction  of the  $j$--th functional component $Y_{s,j}$  ($j=1,2,3$) of the response $Y_{s}$ at target time $s\in \{ s_{1},\dots, s_{n_{i}}\},$    based on the  training  curve sample  of the iteration $i$  of the $5$--fold cross validation technique, for  $i=1,\dots, 5.$

\vspace*{0.5cm}

\begin{table}%[width=.9\linewidth,cols=5,pos=h]
\begin{tabular}{|l|l|l|l|l|}
%\toprule
\hline
\emph{BW interval} &$\overline{\|CVE(Y_{1})\|}_{\infty}$ & $\overline{\|CVE(Y_{2})\|}_{\infty}$ & $\overline{\|CVE(Y_{3})\|}_{\infty}$ & \emph{Mean}\\
\hline
(0.85, 0.86) & 0.3226 & 0.2524 & 0.3576 & 0.3109 \\
(0.4, 0.8) & 0.4710 & 0.3972 & 0.3576 & 0.4086 \\
(0.3, 0.7) & 0.3226 & 0.3427 & 0.3576 & 0.3410 \\
(0.2, 0.4) & 0.4710 & 0.3959 & 0.3973 & 0.4214 \\
(0.1, 0.2) & 0.4306 & 0.3959 & 0.3973 & 0.4079 \\
\hline 
\end{tabular}
\caption{\emph{Bandwidth analysis for eigenspaces 7--9 in the time--varying tangent space  for  the projection  extrinsic local linear  Fr\'echet curve response predictor}.
 For the five bandwidth parameter intervals   tested, the  mean  over  the $5$ iterations of the maximum over the target times  of each $5$--fold cross validation  iteration of the quadratic $\mathbb{H}$-norm of the functional errors is displayed  (see equation (\ref{not}))}
 \label{tab1}
\end{table}

Table \ref{tab1} then
summarizes  the 5-fold cross validation error variability across  bandwidth intervals centered at $B_{n}=n^{-\beta },$   $\beta  >0, $ with $n$ being, as before,  the functional sample size. These $\beta $ values have been selected according to the  double  RFPCA truncation parameter, ensuring the exponential map of the $\mathbb{H}$-valued local curve predictor,   projected into $\bigoplus_{k=7}^{9}\mathcal{H}_{k},$ stays on the manifold.

The  5-fold cross-validation geodesic  errors computed in the implementation of the local Fr\'echet  NW-type curve predictor are quite stable across the six months analyzed, for all bandwidth parameters studied. Specifically,   their intra-curve and inter-curve variability  range  and    patterns  are very similar across the six months (see Figures \ref{Fig:2.5b} in Section \ref{rda} and Figures \ref{Fig:2.5}-\ref{Fig:2.6}  in  Appendix   \ref{app1}).
 The   intrinsic local linear Fr\'echet curve predictions display larger  inter-curve  variability than the NW-type curve predictor (see also Figures \ref{Fig:2.7} and \ref{Fig:2.8} in  Appendix \ref{app1}). Thus, this last predictor  better supports homogeneity   inter-curve assumption  across months. While its  intra-curve variability is  higher than the one displayed by the 5-fold cross validation geodesic errors of the  intrinsic local linear Fr\'echet curve predictions.  We then conclude that  the intrinsic local linear Fr\'echet curve predictor  seems to better fit the local regularity assumption across months.

The  empirical range of the 5-fold cross-validation geodesic  errors is  quite stable across months for both predictors, and all bandwidth parameter values tested,  with slightly less stability observed  for bandwidths  $BW1= 0.2000,$  $ BW2= 0.2250,$    $BW3=0.2500,$ in the case of the intrinsic local linear Fr\'echet curve predictor   (see Figures \ref{Fig:2.3b} in Section \ref{rda} and \ref{Fig:2.3}-\ref{Fig:2.4}, and \ref{Fig:2.7}--\ref{Fig:2.8} in Appendix \ref{app1}).  Note that, in the case of NW-type curve predictions the interval $[0,1.1]$ includes the 5-fold cross-validation geodesic  error range across all months.

As commented, conditions (i)-(ii) hold for the sphere.  We have applied the time-varying  exponential and logarithmic maps with origin at  the regressor empirical Fr\'echet curve mean values. The numerical results displayed support this small deviation from the assumption of a common Fr\'echet curve mean in condition (v) (see Remark \ref{rem1aa}).    The concentration of curve data around Fr\'echet curve means provides empirical evidence about  the remaining assertions in condition (v) are satisfied. Sample path local regularity assumed in condition (iii) is also satisfied, since  Lipschitz-continuous-type  curve patterns are observed  in the trajectories displayed  in Figure \ref{Fig:2.1b}, and Figures \ref{Fig:2.1}-\ref{Fig:2.2} in  Appendix  \ref{app1}.   Finally, the observed smooth  inter-curve variability across months supports the stationarity assumption in (iv). The   curve frequency smooth local variation displayed in Figures \ref{Fig:2.1b}, and  \ref{Fig:2.1} and \ref{Fig:2.2} in  Appendix \ref{app1} also supports
conditions (vi) and (vii).

\section{Final comments}

 Section \ref{sfc1} first provides a summary of the contributions, and opens some challenging research directions. A brief discussion on  the   bandwidth modelling adopted, depending on the functional sample size, is held in  Section \ref{sebwtp}. Interaction with other tuning parameters is also discussed. Section  \ref{tceff} summarizes the effect of time correlation.
\subsection{Contribution and open research lines}
\label{sfc1}
The optimality, in the mean-square sense, of the two proposed local linear Fr\'echet curve predictors on manifolds is introduced and theoretically analyzed.  In particular, up to our knowledge, this is the first time the explicit derivation of the least-squares local linear Fr\'echet functional regression predictor is  obtained for response and regressor evaluated in a separable Hilbert space.  Furthermore, the application of this result to the context of extrinsic local linear  Fr\'echet functional regression in curve spaces is non trivial (see Sections \ref{geoass}  and \ref{lefr0}).  The intrinsic local linear Fr\'echet  functional regression approach extends
\cite{Petersen.19}  to the context of infinite-dimensional regressor and response evaluated in a compact Riemannian  manifold, when curve data in the sample are temporal correlated. The simulation study and real-data application illustrate the predictive performance of the   empirical versions of both local linear Fr\'echet curve predictors,  in terms of accuracy and residual intra- and  inter-curve   variability,  comparing them with the   simpler  baseline method based on the geodesic  NW-type curve
predictor.

 Under condition (ii), the residual variability associated with the time-varying  exponential map of  the truncated RFPCA-based kernel Fr\'echet curve predictor  is upper bounded by its residual variability in the ambient Hilbert space  of vector functions, with values in the time-varying tangent space. Thus,   Chapters 4 and 8 in \cite{Bosq.00} can be applied to ensure  consistency of the truncated extrinsic local linear Fr\'echet curve predictor. The derivation of consistency in the asymptotic analysis of  the intrinsic local linear  Fr\'echet curve  predictor constitutes an interesting, but  still open research problem that could be addressed  in  a near future.

\subsection{Bandwidth and related tuning parameter selection problem}
\label{sebwtp}

In the extrinsic local linear Fr\'echet curve predictor, bandwidth parameter  interacts  with the RFPCA truncation parameter, and the sampling frequency in the discrete observation of the curve data values in the manifold.  We apply the results derived in  Chapters 4 and 8 in \cite{Bosq.00} to fit RFPCA truncation parameter,  as a function of the curve data sample size, ensuring consistency in the ambient Hilbert space. In particular, in  Section \ref{simstud}, we have applied   logarithmic RFPCA truncation satisfying the conditions assumed in these results. Condition (ii) guarantees good properties of the corresponding   exponential mapped curve predictor in the manifold. The   sampling frequency increases here as  a positive fractional power function of the curve sample size, also depending on  the resolution level  reflected in the  RFPCA truncation, established according to  the local regularity of the model. The bandwidth parameter is  obtained as a negative fractional power $B_{n}=(log(n))^{-1/\beta }$ (with $\beta =10.10,  10.11,    10.12,  10.14,    10.15,   10.16,   10.18,   10.19,   10.20,    10.22,    10.23,    10.25$) of the curve sample log-size, whose rate of convergence to zero depends on the positive fractional power function defining the sampling frequency.

Under the presence of weaker  intra-curve correlations for long geodesic distances  in Section \ref{rda}, the exponential map of the $\mathbb{H}$-valued local linear Fr\'echet curve predictor projected into coarser scales does not stay on the manifold. A double  RFPCA truncation in the ambient Hilbert space is considered, excluding coarser scales, under  suitable bandwidth parameter values given by $B_{n}=n^{-\beta },$ $\beta =  0.1500, 0.3000, 0.5000, 0.6000, 0.8550.$ Note that data  have been   recorded every half second,  allowing the choice of small bandwidths for the  reconstruction  of    local variability  at the resolution levels selected from the double RFPCA truncation.

     As commented in  Section \ref{sfc1},  consistency in the framework of intrinsic local linear Fr\'echet curve prediction constitutes an open research problem. Hence, our choice in  Section \ref{simstud}  of the bandwidth parameter, as a negative power function of the curve sample size,  is not driven by consistency, but allows conditions   (vi) and (vii) to hold, ensuring  asymptotic optimality of the intrinsic local linear Fr\'echet curve predictor.
     Ergodicity assumed in condition (iv) ensures good properties of the empirical version of this curve predictor.    In  Section   \ref{simstud}, keeping in mind condition (iii), slow  increasing of sampling  frequency  is allowed, in terms of a  positive fractional power of the functional sample size.  In Section \ref{rda}, since data  have been   recorded every half second,  smaller bandwidth parameter values are considered.
%\end{itemize}

\subsection{The effect of time correlation}
\label{tceff}
Sections \ref{simstud} and \ref{rda} correspond to a weak time correlation scenario between functional random objects,  with  stronger inter- and  intra-curve correlations displayed in Section \ref{simstud}.
 The extrinsic local linear Fr\'echet curve predictor is more affected by  changes in the  time correlation  than the other  intrinsic local Fr\'echet curve predictors tested  (see Tables \ref{tabS62} and \ref{tab1}).   In Section \ref{simstud}, a  linear time correlation scheme is generated in the time-varying tangent space, based on  a functional linear  model.  This aspect of the generations agrees with  the RFPCA nature of the log-mapped response and regressor processes, see Chapters 4 and 8 in \cite{Bosq.00}. The exponential map of the  local linear functional predictions in the ambient Hilbert space then  states on the compact Riemannian manifold. In Section \ref{rda}, given the longer temporal support of the curve data analyzed, almost null correlations are observed  at coarser scales for large geodesic distances. In particular, the exponential map of the $\mathbb{H}$-valued local linear Fr\'echet curve predictor projected   into coarser scales does not state on the compact Riemannian manifold.

The relative  intra- and inter-curve residual variability  of both, the   NW-type, and the intrinsic local linear Fr\'echet curve predictors,  is quite  stable  across   the time correlation scenarios analyzed in Sections \ref{simstud} and \ref{rda} (see, in particular,  Sections \ref{QA} and  \ref{QA2}).
The intrinsic local linear Fr\'echet curve predictor displays  higher degree of robustness under time correlation  changes regarding local regularity assumptions. While   NW-type Fr\'echet curve predictor is less affected by  time correlation changes  regarding  the assumption of homogeneity through time of the  functional random objects (see also Table \ref{tabS62}  and Figures \ref{Fig:2.7} and \ref{Fig:2.8} in Appendix  \ref{app1}).

\appendix
\section{Resolution of  equation system   (\ref{elf0})--(\ref{elf}) by projection}
\label{app0}

This section completes the derivation of the local linear Fr\'echet curve predictor for response and regressor evaluated in a separable Hilbert space $\mathcal{H}.$
Specifically, the projections defining this curve predictor are obtained in the following equations.

Denote, for $k\geq 1,$
 \begin{eqnarray}\mu_{0}&=&E\left[ K_{B_{n}}\left(\left\|X-x_{0}\right\|_{\mathcal{H}}\right)\right],\nonumber\\ \mu_{j}^{(k)}&=&E\left[[(X-x_{0})(\phi_{k})]^{j}K_{B_{n}}\left(\left\|X-x_{0}\right\|_{\mathcal{H}}\right)\right],\ j\geq 1,\nonumber\\
r_{j}^{(k)}&=&E\left[K_{B_{n}}\left(\left\|X-x_{0}\right\|_{\mathcal{H}}\right)[(X-x_{0})(\phi_{k})]^{j} Y(\phi_{k})\right],\ j\geq 0,\label{eqnirmfunct}
\end{eqnarray}
\noindent where $Y(\phi_{k})=\left\langle Y,\phi_{k}\right\rangle_{\mathcal{H}},$ and  $[(X-x_{0})(\phi_{k})]^{j}=\left[\left\langle X-x_{0}, \phi_{k}\right\rangle_{\mathcal{H}}\right]^{j},$ $k\geq 1,$ $j\geq 0.$

 From equations (\ref{exp2b}) and  (\ref{eqnirmfunct}),
 \begin{eqnarray}
 &&0=\sum_{k\geq 1}\lambda_{k}(\mathcal{A})\mu_{2}^{(k)}-r_{1}^{(k)}+\left\langle m(x_{0}),\phi_{k}\right\rangle_{\mathcal{H}}\mu_{1}^{(k)},
  \label{s2302}
 \end{eqnarray}

 \noindent and, from equation (\ref{effn2}), for every $k\geq 1,$
 \begin{eqnarray}&&
 \left\langle m(x_{0}),\phi_{k}\right\rangle_{\mathcal{H}} =m(x_{0})(\phi_{k})=\frac{\left\langle r_{0}-A(\mu_{1}),\phi_{k}\right\rangle_{\mathcal{H}} }{\mu_{0}}=\frac{r_{0}(\phi_{k})-\mathcal{A}(\mu_{1})(\phi_{k})}{\mu_{0}}\nonumber\\ &&\hspace*{2cm}=\frac{r_{0}^{(k)}-\lambda_{k}(\mathcal{A})\mu_{1}^{(k)}}{\mu_{0}}. \label{s2302b}
 \end{eqnarray}

 By replacing $\left\langle m(x_{0}),\phi_{k}\right\rangle_{\mathcal{H}}$ in equation (\ref{s2302}) by its expression in equation (\ref{s2302b}),
 we obtain
 \begin{eqnarray}
 0&=&\sum_{k\geq 1}\lambda_{k}(\mathcal{A})\mu_{2}^{(k)}-r_{1}^{(k)}+\mu_{1}^{(k)}\frac{r_{0}^{(k)}-\lambda_{k}(\mathcal{A})\mu_{1}^{(k)}}{\mu_{0}}\nonumber\\
 &=&\sum_{k\geq 1}\lambda_{k}(\mathcal{A})\left[\mu_{2}^{(k)}-\frac{[\mu_{1}^{(k)}]^{2}}{\mu_{0}}\right]-r_{1}^{(k)}+\frac{r_{0}^{(k)}\mu_{1}^{(k)}}{\mu_{0}}.
  \label{s2302c}
 \end{eqnarray}

 In particular, equation (\ref{s2302c}) holds when, for every $k\geq 1,$
 \begin{eqnarray}
 &&
 \lambda_{k}(\mathcal{A})\left[\mu_{2}^{(k)}-\frac{[\mu_{1}^{(k)}]^{2}}{\mu_{0}}\right]-r_{1}^{(k)}+\frac{r_{0}^{(k)}\mu_{1}^{(k)}}{\mu_{0}}=0,
 \label{s2302d}
 \end{eqnarray}
 \noindent leading to
 \begin{eqnarray}
 &&\widehat{\lambda_{k}}(\mathcal{A})=\frac{\mu_{0}r_{1}^{(k)}-r_{0}^{(k)}\mu_{1}^{(k)}}{\mu_{2}^{(k)}\mu_{0}-[\mu_{1}^{(k)}]^{2}},\ k\geq 1.\label{s2302e}
 \end{eqnarray}
 Equation (\ref{s2302e}) means that the Fr\'echet derivative $\mathcal{A}$ of the regression operator $m$   admits
 the following  series expansion: For every $f,g\in \mathcal{H},$
 $$\widehat{\mathcal{A}}(f)(g)=\sum_{k\geq 1}\frac{\mu_{0}r_{1}^{(k)}-r_{0}^{(k)}\mu_{1}^{(k)}}{\mu_{2}^{(k)}\mu_{0}-[\mu_{1}^{(k)}]^{2}}\left\langle f,\phi_{k}\right\rangle_{\mathcal{H}}\left\langle g,\phi_{k}\right\rangle_{\mathcal{H}}.$$

 Replacing  $\lambda_{k}(\mathcal{A})$ in equation (\ref{s2302b}) by $\widehat{\lambda_{k}}(\mathcal{A})$  in (\ref{s2302e}), we have, for $k\geq 1,$
 \begin{eqnarray}\widehat{m(x_{0})}(\phi_{k})&=&
 \left\langle \widehat{m(x_{0})},\phi_{k}\right\rangle_{\mathcal{H}}\nonumber\\&=&\frac{r_{0}^{(k)}-\left[\frac{\mu_{0}r_{1}^{(k)}-r_{0}^{(k)}\mu_{1}^{(k)}}{\sigma_{0}^{2}(k)}\right]
 \mu_{1}^{(k)}}{\mu_{0}}=\frac{\mu_{2}^{(k)}r_{0}^{(k)}-\mu_{1}^{(k)}r_{1}^{(k)}}{\sigma_{0}^{2}(k)},
\label{s2302f}\end{eqnarray}
\noindent where $\sigma_{0}^{2}(k)=\mu_{2}^{(k)}\mu_{0}-[\mu_{1}^{(k)}]^{2}.$

Thus, from equation (\ref{s2302f}), keeping in mind definition of  $r_{0}^{(k)}$ and $r_{1}^{(k)}$ in equation (\ref{eqnirmfunct}), for every
$k\geq 1,$
 \begin{eqnarray}
 \widehat{m(x_{0})}(\phi_{k})&=&\frac{1}{\sigma_{0}^{2}(k)}\int_{\mathcal{H}\times \mathcal{H}}y(\phi_{k})K_{B_{n}}\left(\left\|x-x_{0}\right\|_{\mathcal{H}}\right)\left[\mu_{2}^{(k)}-\mu_{1}^{(k)}(x-x_{0})(\phi_{k})\right]P(dx,dy)\nonumber\\
 &=&E\left[S^{(k)}(X,x_{0},B_{n})Y(\phi_{k})\right],\label{s2302g}\end{eqnarray}
 \noindent where
 \begin{eqnarray}&& S^{(k)}(X,x_{0},B_{n})=\frac{1}{\sigma_{0}^{2}(k)}\left[K_{B_{n}}\left(\left\|X-x_{0}\right\|_{\mathcal{H}}\right)\left[\mu_{2}^{(k)}
 -\mu_{1}^{(k)}(X-x_{0})(\phi_{k})\right]\right].\nonumber\\
 \label{s2302h}\end{eqnarray}

\section{Auxiliary material for December 1979-April 1980}
\label{app1}
In this appendix, the bivariate curve data sample, and the  5-fold cross validation variability across months is visualized for the period December 1979-April 1980.  The numerical analysis of these graphical results is provided in  Sections \ref{rda}, \ref{QA2}, and    \ref{tceff}.

\begin{figure}
\centering
\includegraphics[width=0.32\textwidth]{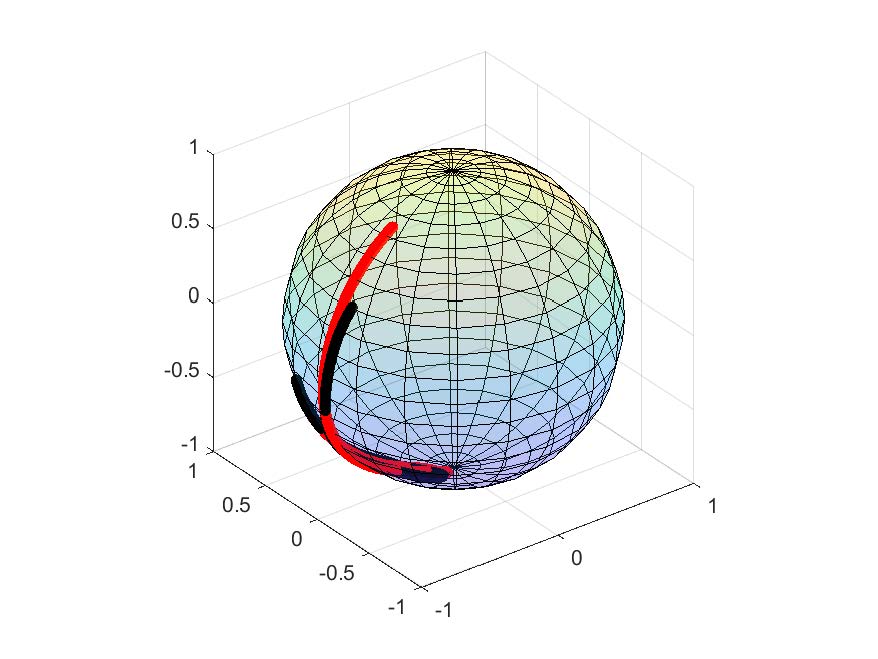}
\includegraphics[width=0.32\textwidth]{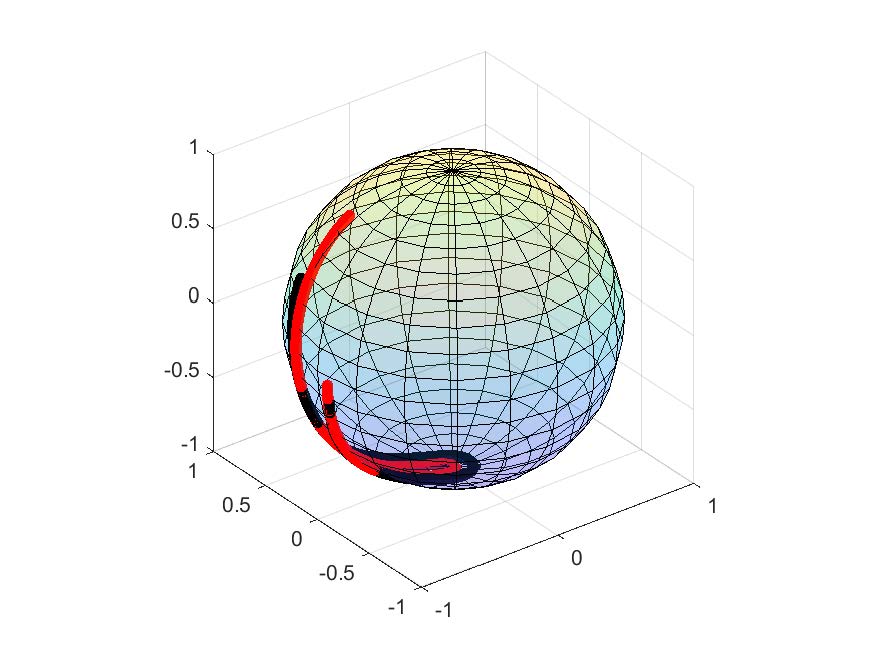}
\includegraphics[width=0.32\textwidth]{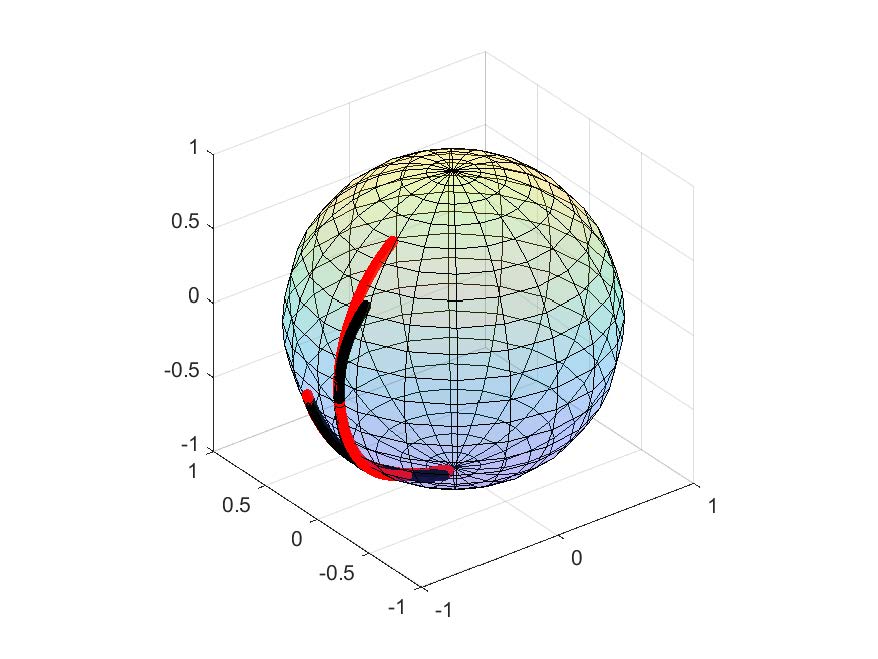}
\includegraphics[width=0.32\textwidth]{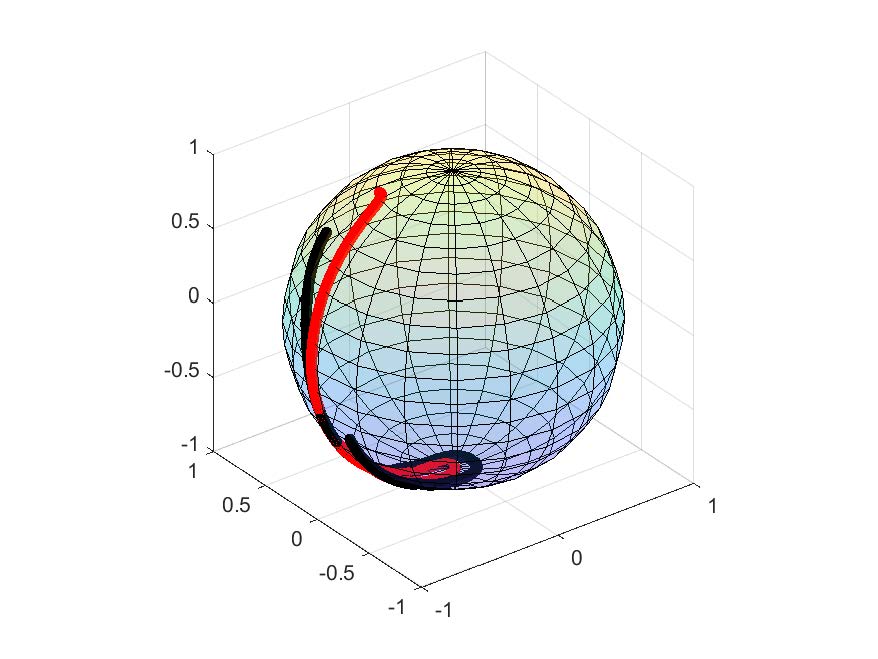}
\includegraphics[width=0.32\textwidth]{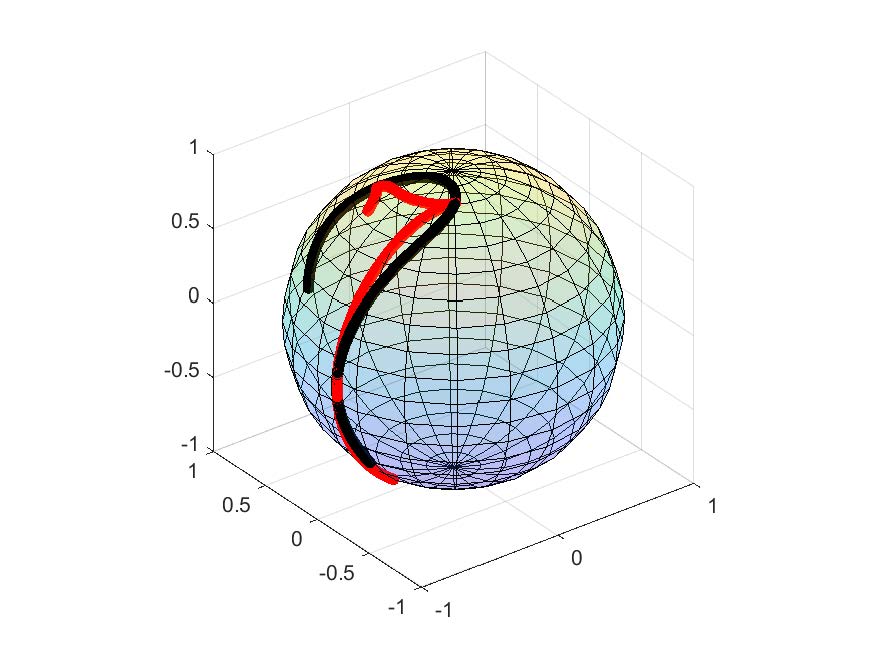}
\includegraphics[width=0.32\textwidth]{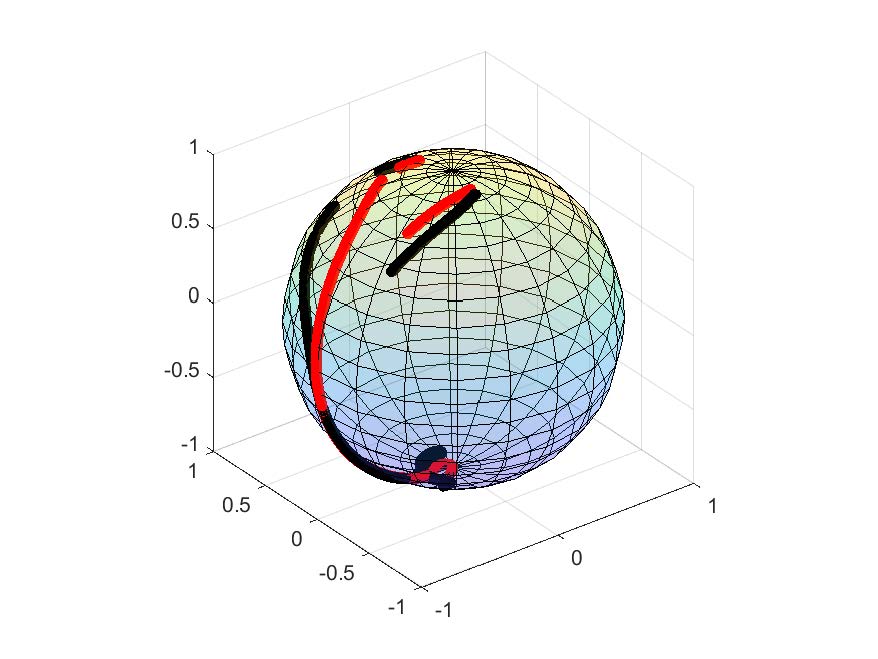}
\includegraphics[width=0.32\textwidth]{Rs_Rg_24iv15_mes_3_M1}
\includegraphics[width=0.32\textwidth]{Rs_Rg_24iv15_mes_3_M15}
\includegraphics[width=0.32\textwidth]{Rs_Rg_24iv15_mes_3_M29}
\includegraphics[width=0.32\textwidth]{Rs_Rg_24iv15_mes_3_M43}
\includegraphics[width=0.32\textwidth]{Rs_Rg_24iv15_mes_3_M57}
\includegraphics[width=0.32\textwidth]{Rs_Rg_24iv15_mes_3_M71}
\caption{ \emph{Spherical bivariate curve data}. Sample functional elements at  $t=1, 15$ $29, 43, 57, 71,$ for the months
 December 1979 (lines 1-2) and January 1980 (lines 3-4)
}\label{Fig:2.1}
\end{figure}

\begin{figure}
\centering
\includegraphics[width=0.28\textwidth]{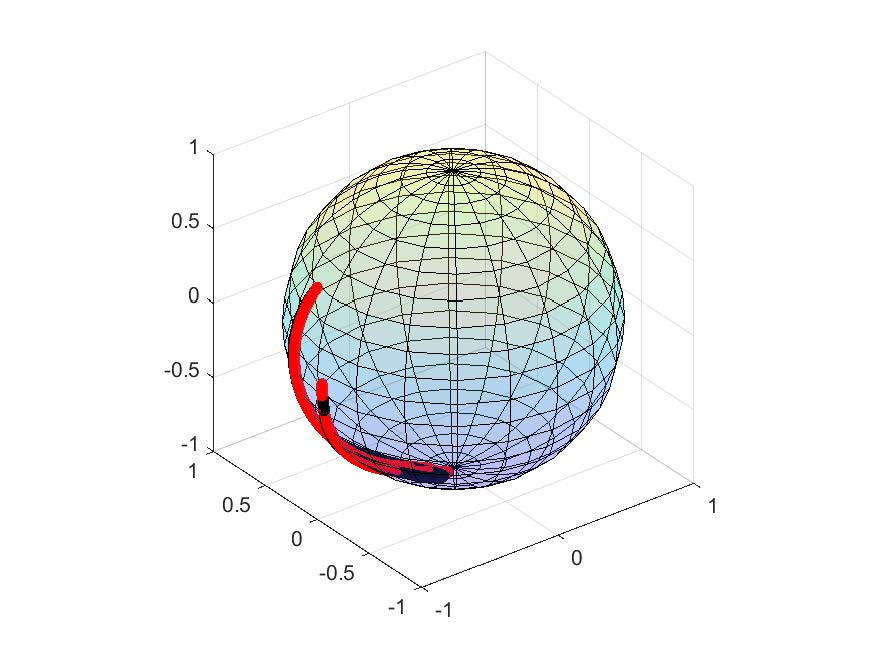}
\includegraphics[width=0.28\textwidth]{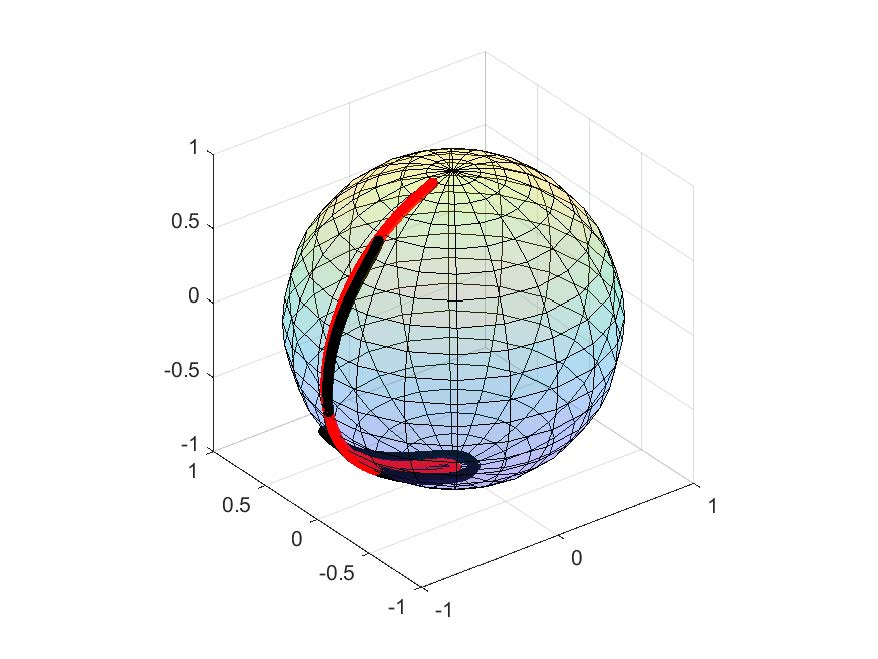}
\includegraphics[width=0.28\textwidth]{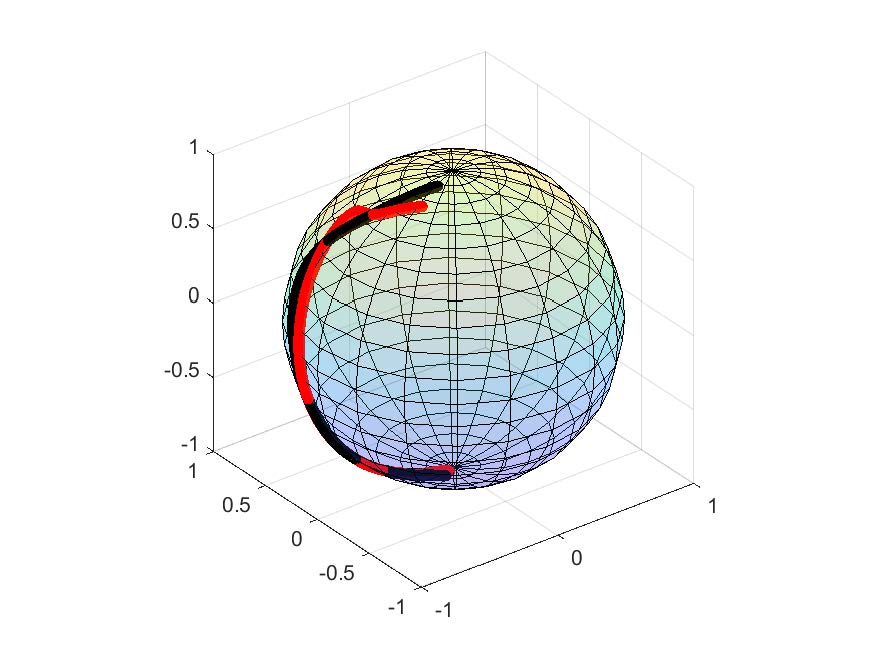}
\includegraphics[width=0.28\textwidth]{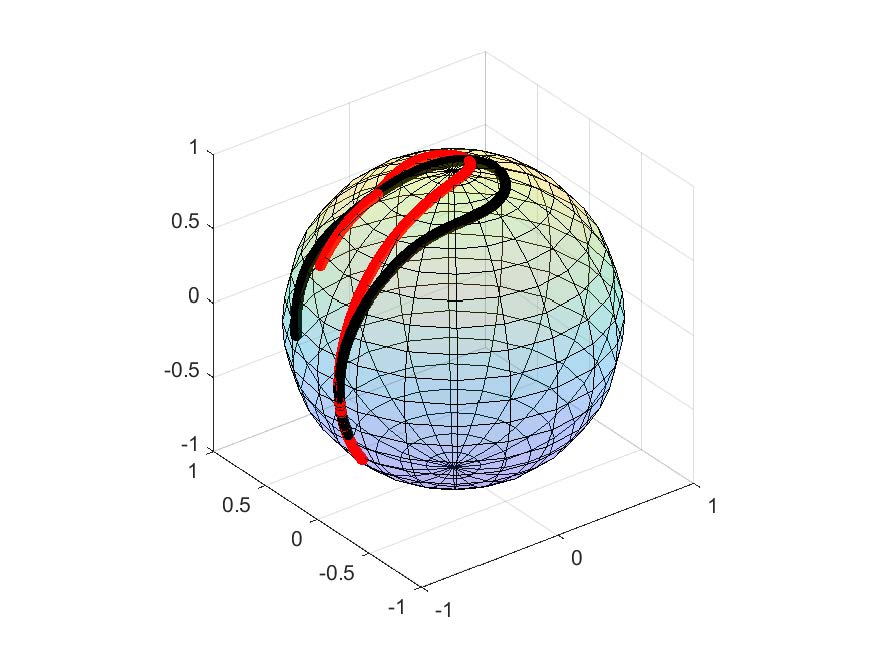}
\includegraphics[width=0.28\textwidth]{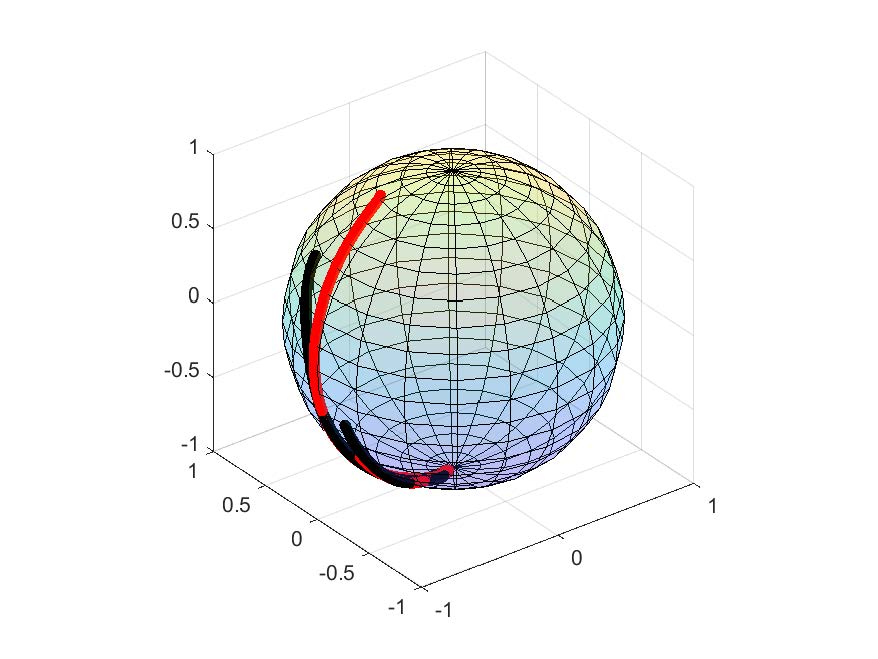}
\includegraphics[width=0.28\textwidth]{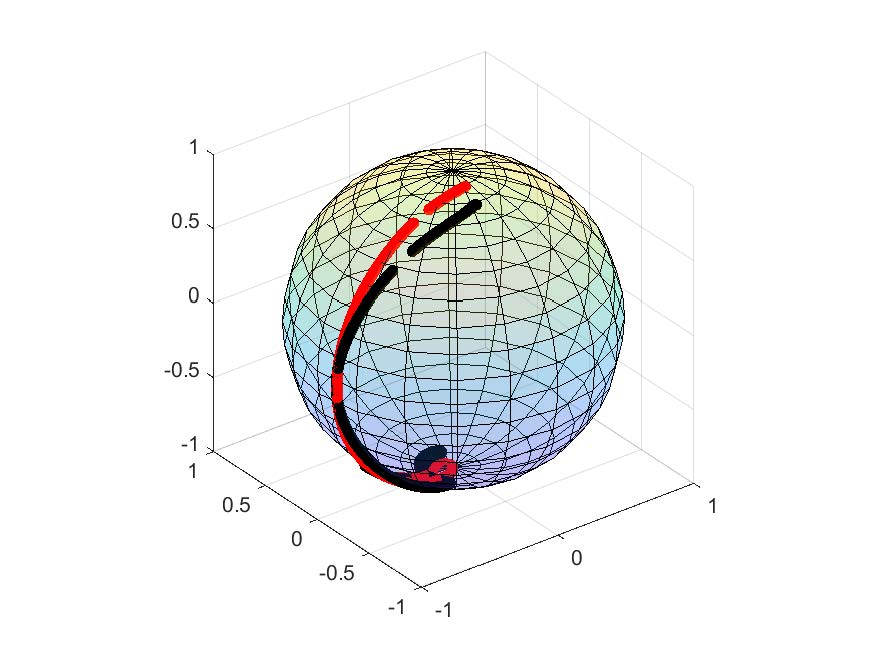}
\includegraphics[width=0.28\textwidth]{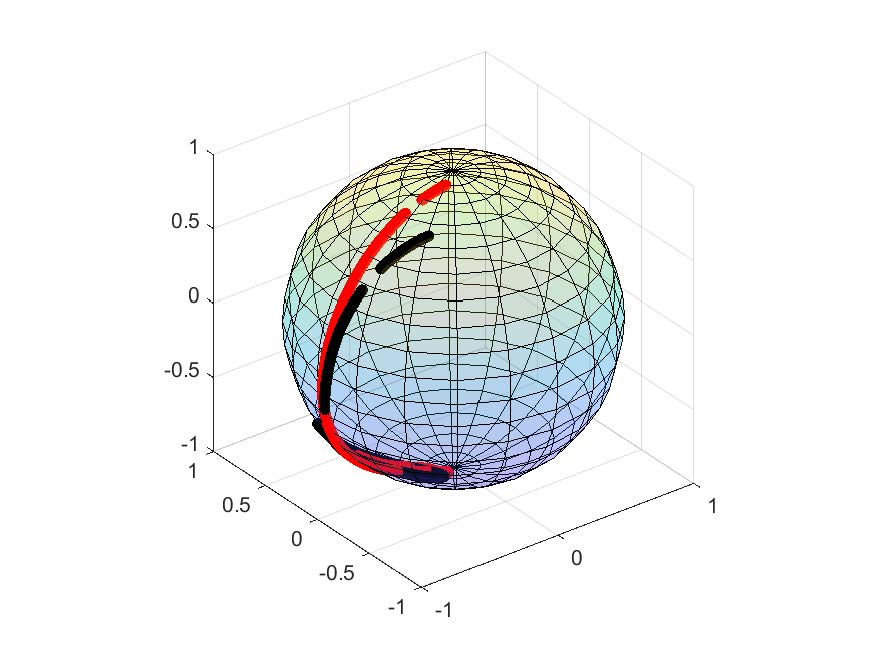}
\includegraphics[width=0.28\textwidth]{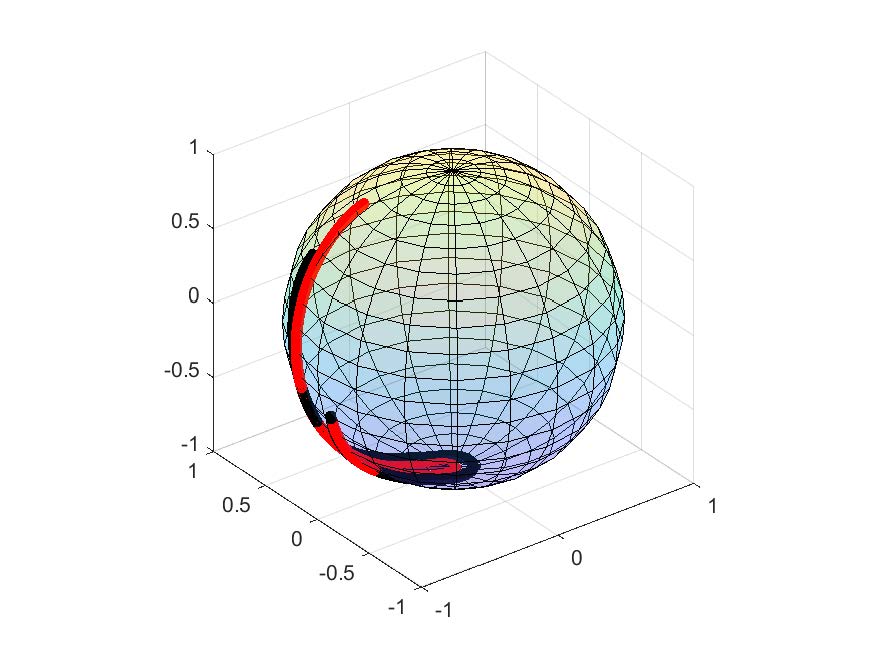}
\includegraphics[width=0.28\textwidth]{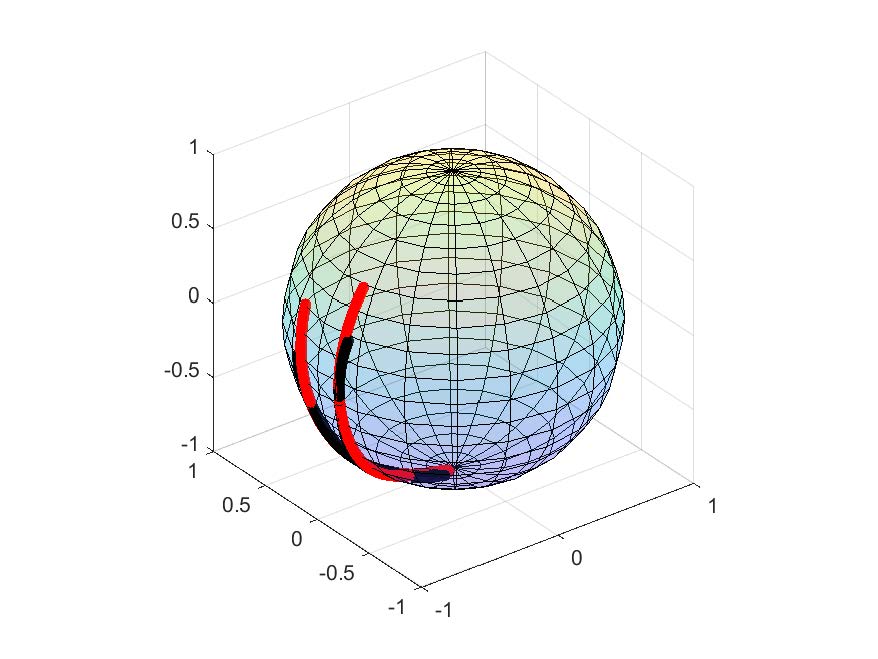}
\includegraphics[width=0.28\textwidth]{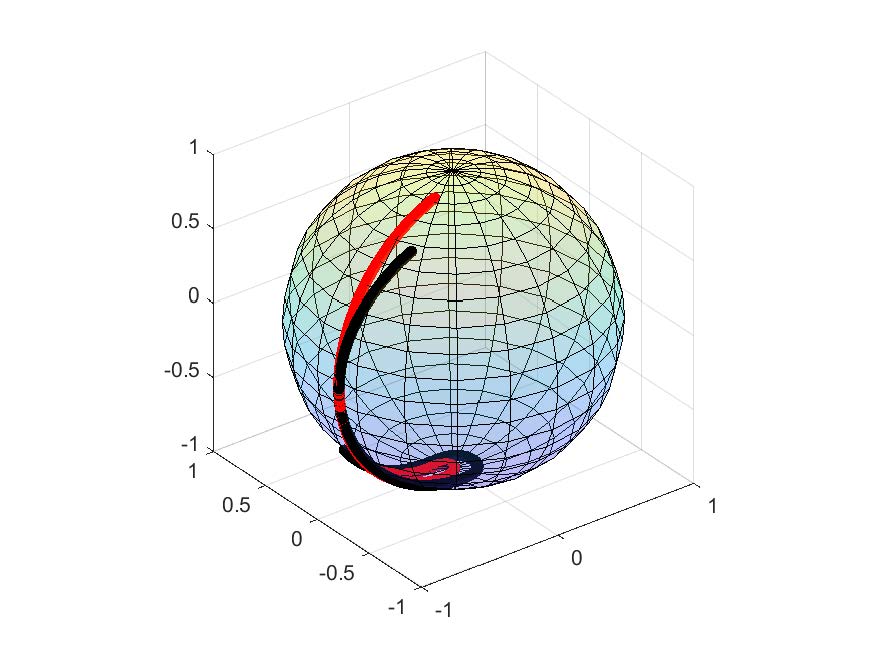}
\includegraphics[width=0.28\textwidth]{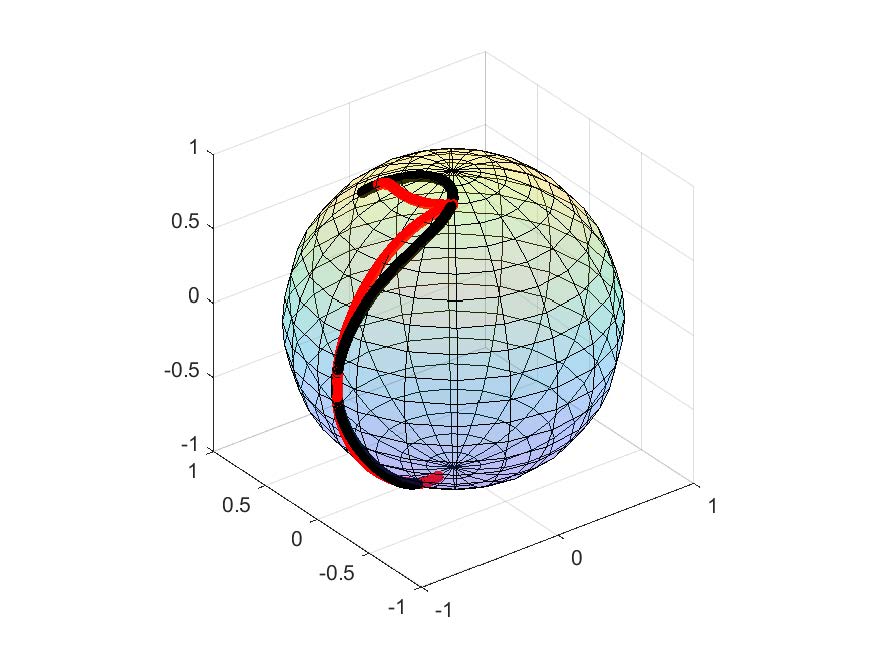}
\includegraphics[width=0.28\textwidth]{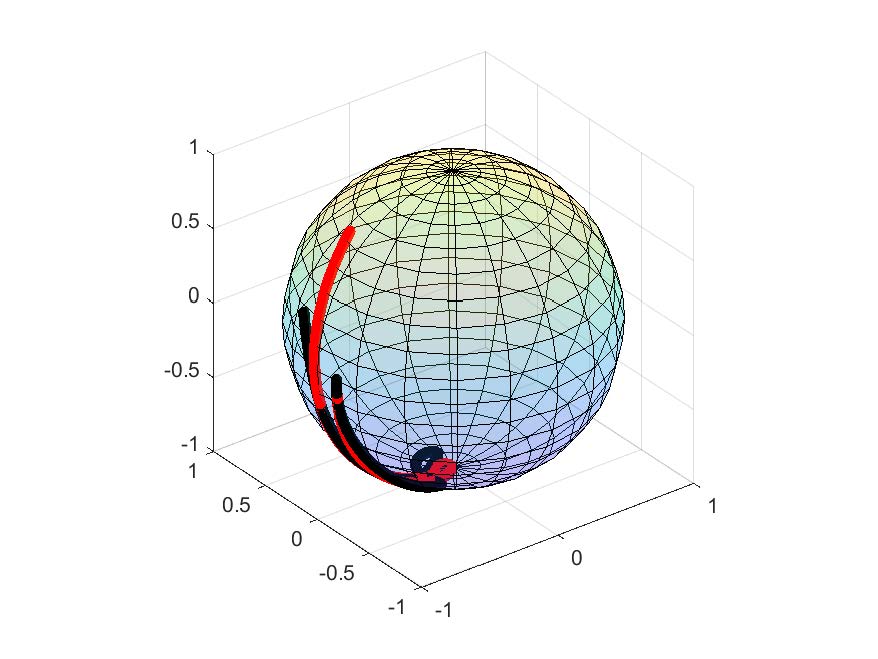}
\includegraphics[width=0.28\textwidth]{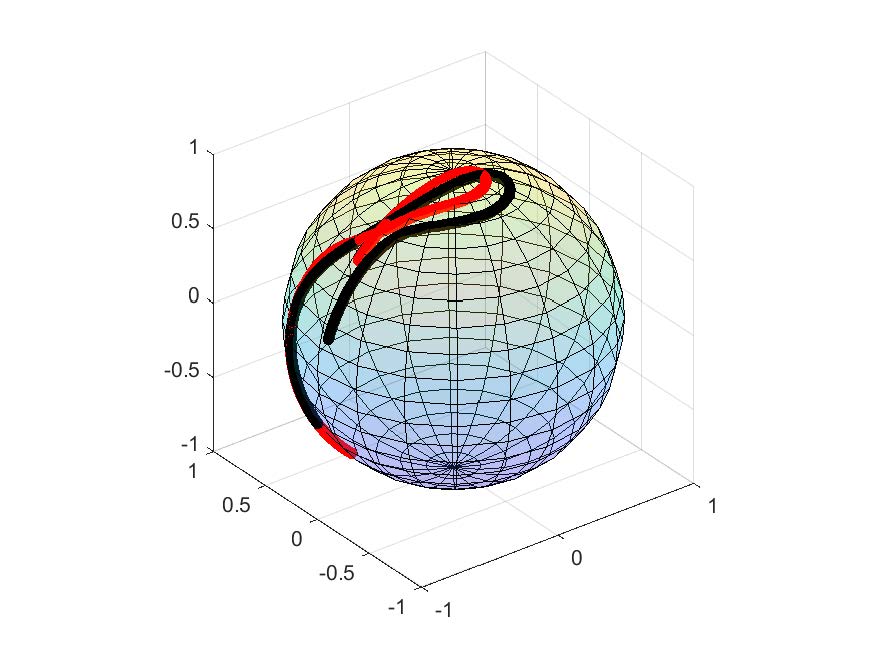}
\includegraphics[width=0.28\textwidth]{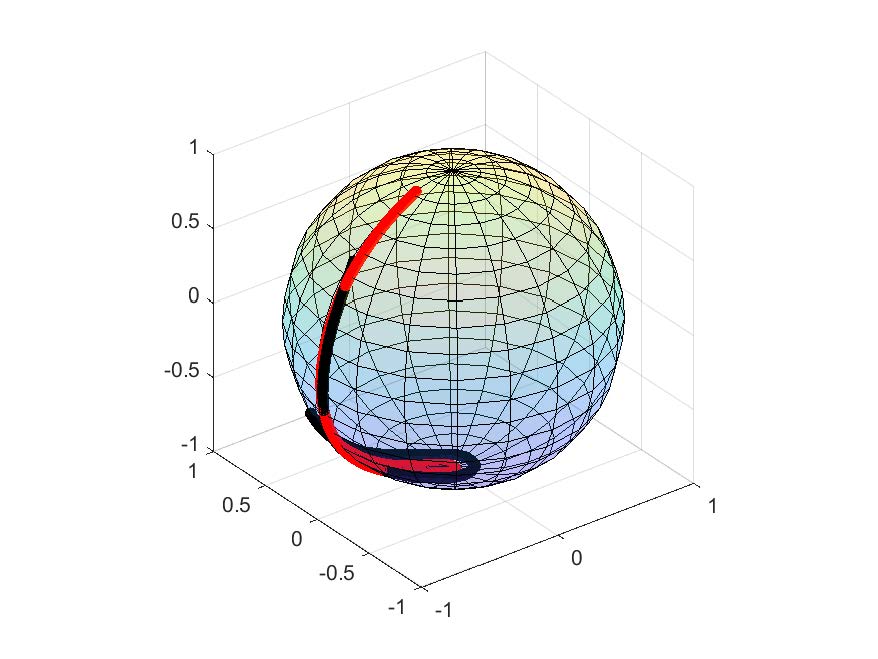}
\includegraphics[width=0.28\textwidth]{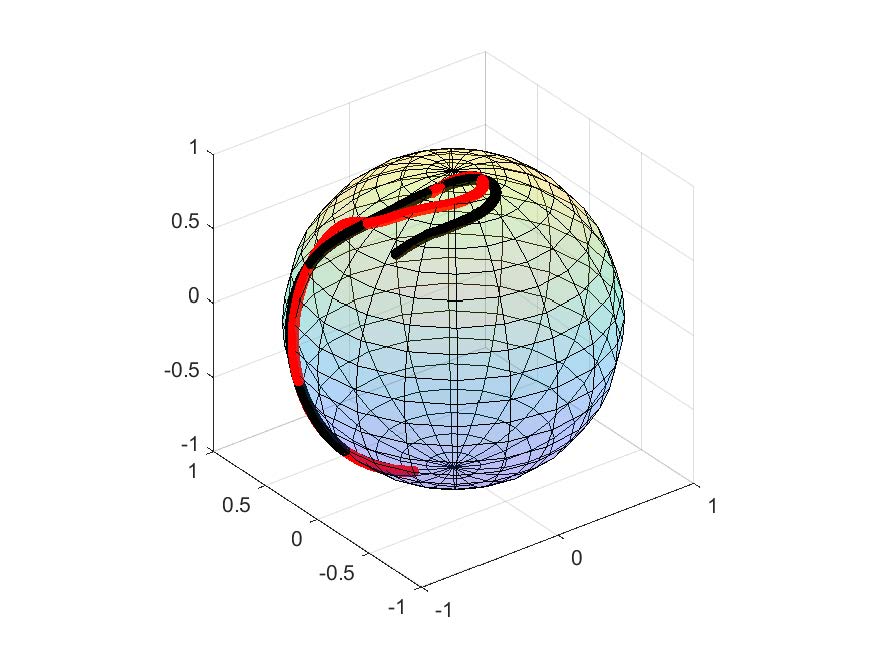}
\includegraphics[width=0.28\textwidth]{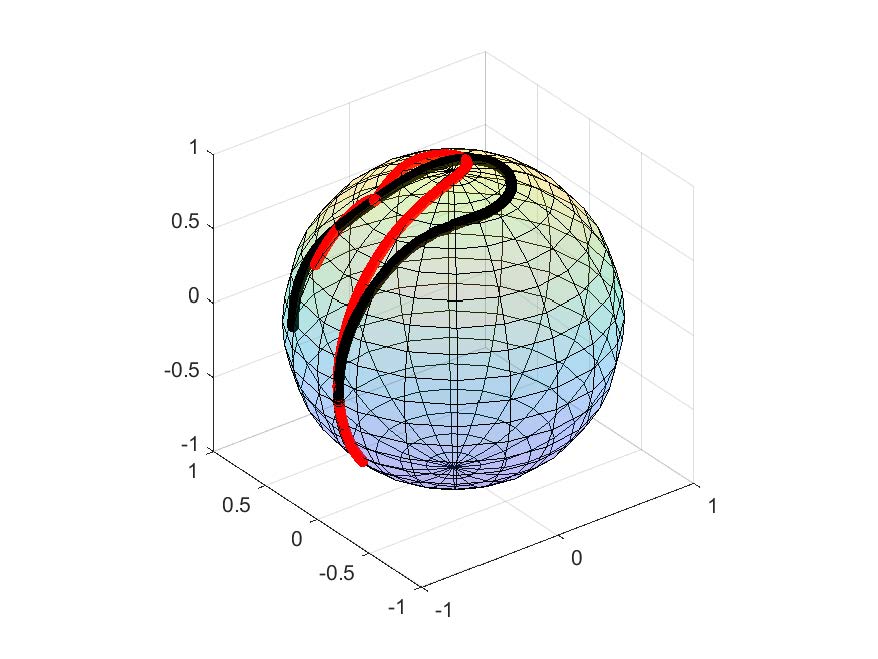}
\includegraphics[width=0.28\textwidth]{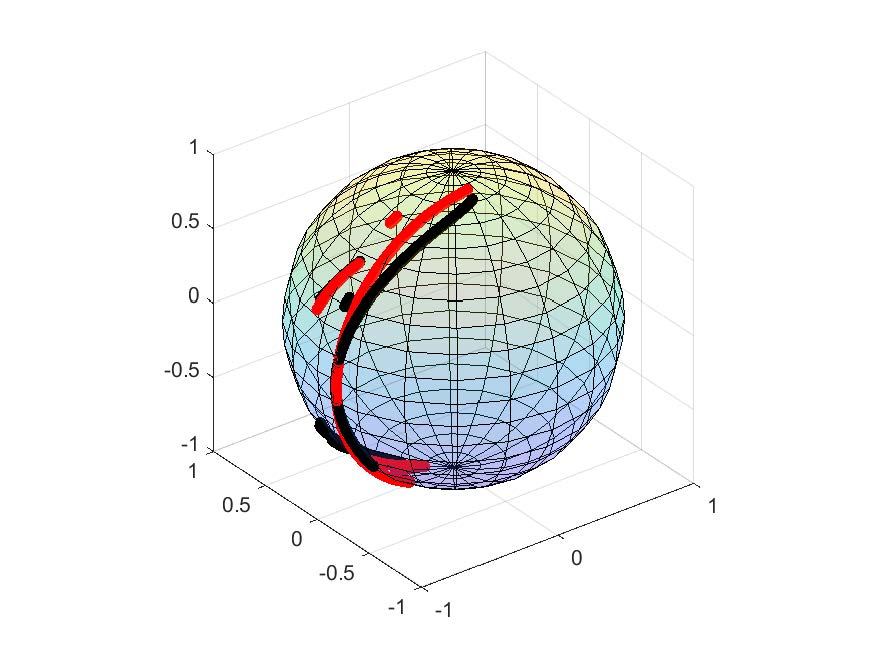}
\includegraphics[width=0.28\textwidth]{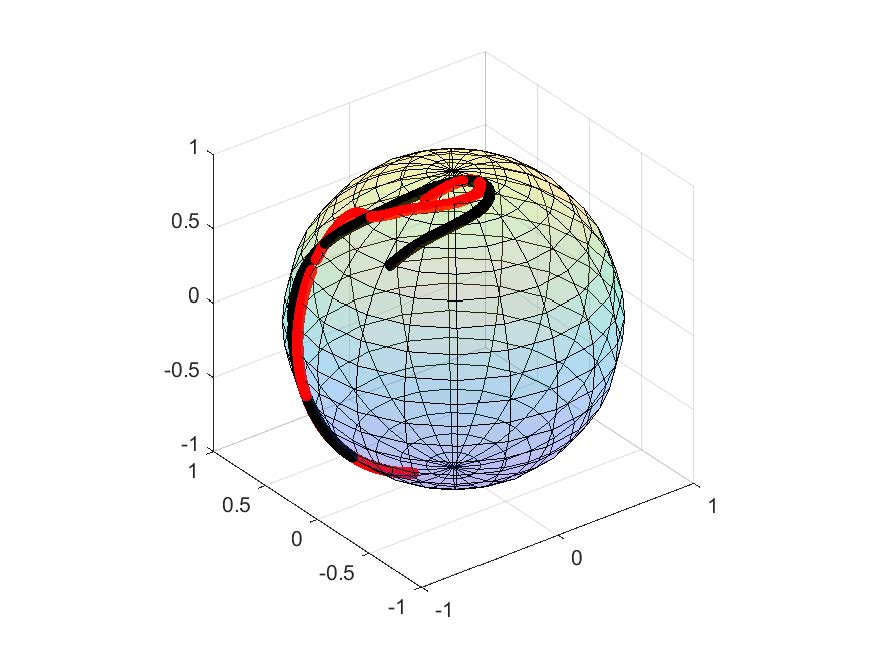}
\caption{ \emph{Spherical bivariate curve data}. Sample functional elements at  $t=1, 15,$ $29,  43, 57, 71,$ for the months
 February 1980 (lines 1-2), March 1980 (lines 3-4) and April 1980 (lines 5-6)
}\label{Fig:2.2}
\end{figure}

 \begin{figure}
\centering
\includegraphics[width=0.32\textwidth]{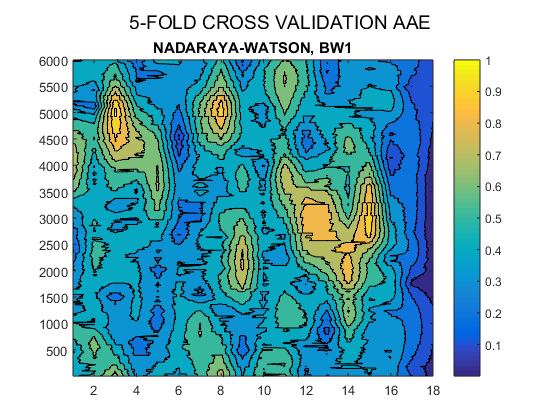}
\includegraphics[width=0.32\textwidth]{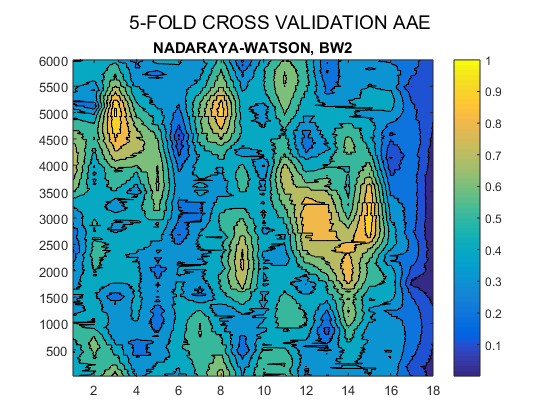}
\includegraphics[width=0.32\textwidth]{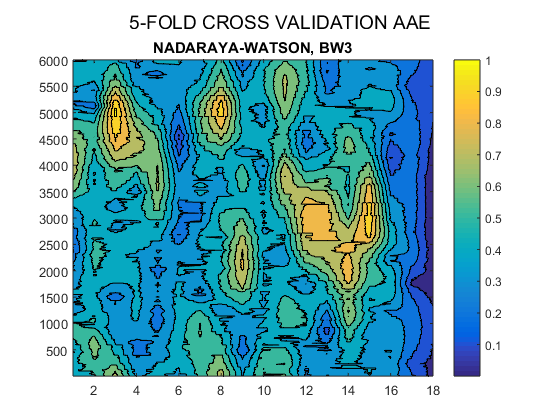}
\includegraphics[width=0.32\textwidth]{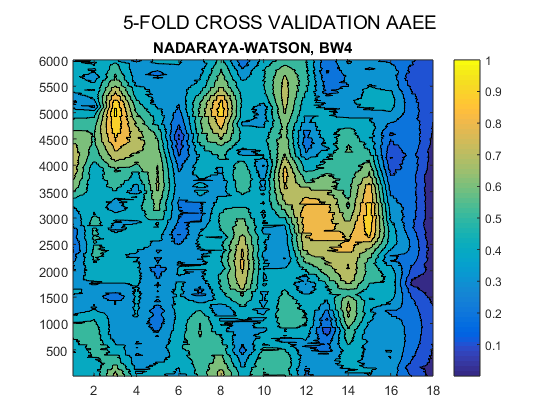}
\includegraphics[width=0.32\textwidth]{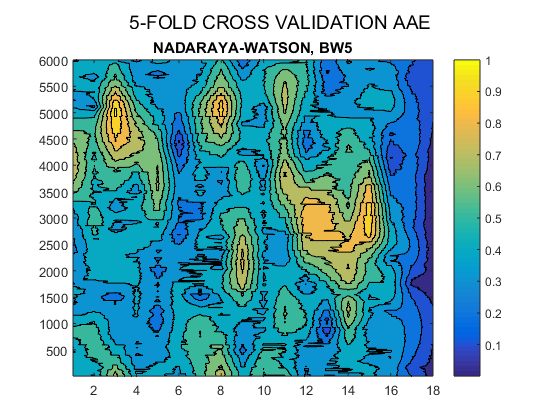}\\
\includegraphics[width=0.32\textwidth]{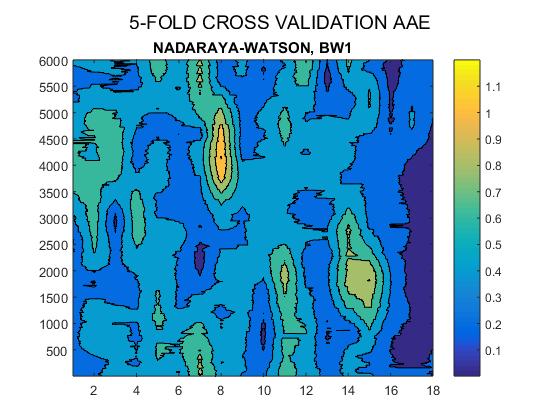}
\includegraphics[width=0.32\textwidth]{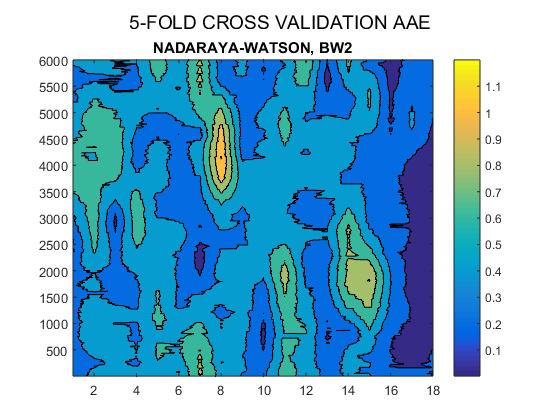}
\includegraphics[width=0.32\textwidth]{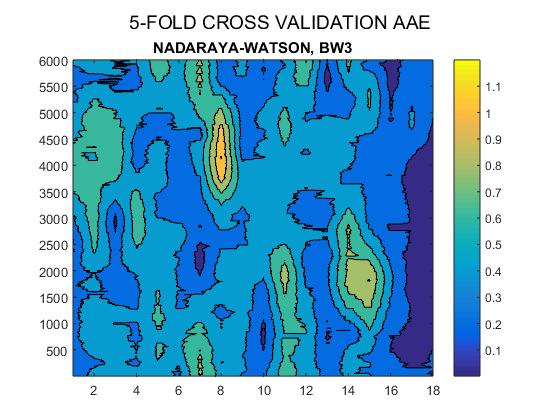}
\includegraphics[width=0.32\textwidth]{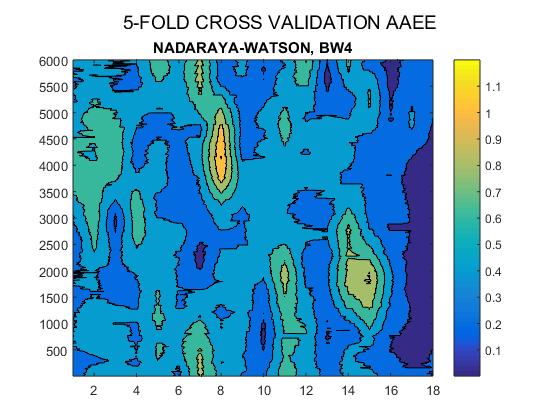}
\includegraphics[width=0.32\textwidth]{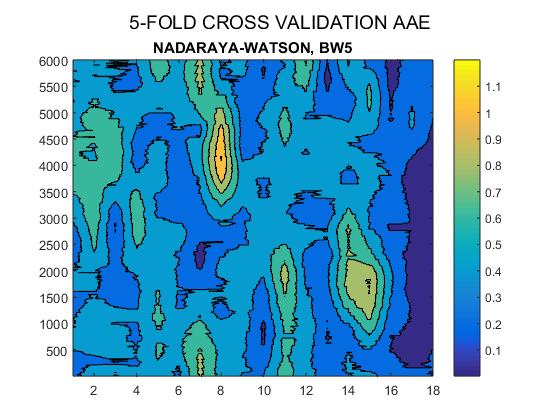}
\caption{\emph{NW-type local   Fr\'echet  regression  curve predictor}.  Contourplots of pointwise values at 6000 temporal nodes of the 5-fold cross-validation geodesic  absolute  curve errors, considering  bandwidths $BW1= 0.2000,$   $BW2= 0.2250,$    $BW3=0.2500$ (top)    $BW4=0.2750,$   $BW5= 0.3000$  (bottom)  for the months   December 1979 (lines 1-2) and January 1980 (lines 3-4)}\label{Fig:2.5}
\end{figure}

\begin{figure}
\centering
\includegraphics[width=0.28\textwidth]{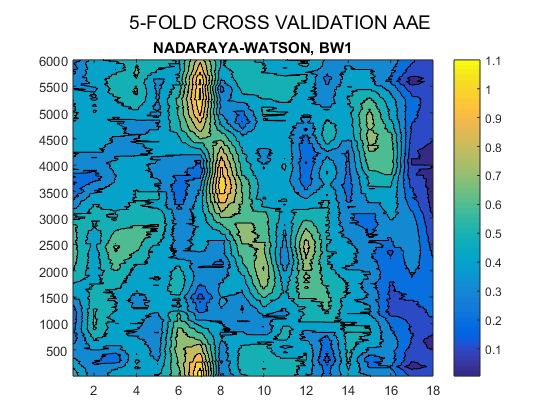}
\includegraphics[width=0.28\textwidth]{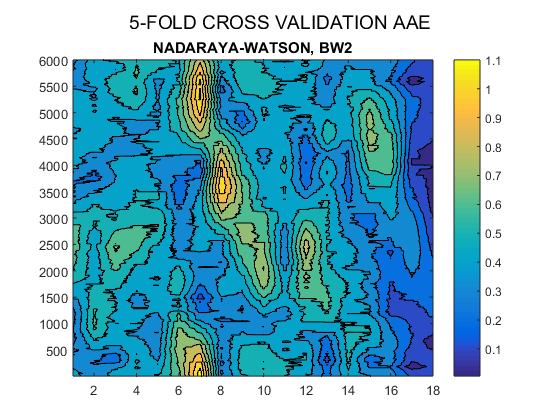}
\includegraphics[width=0.28\textwidth]{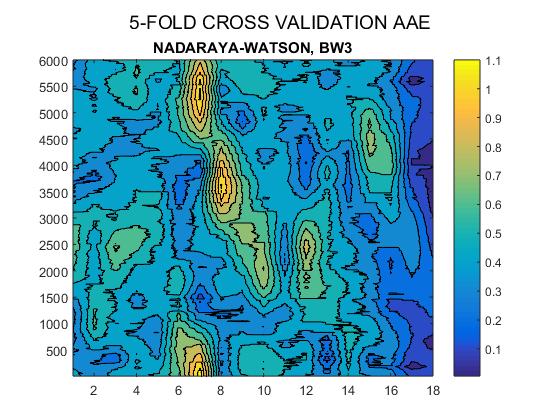}
\includegraphics[width=0.28\textwidth]{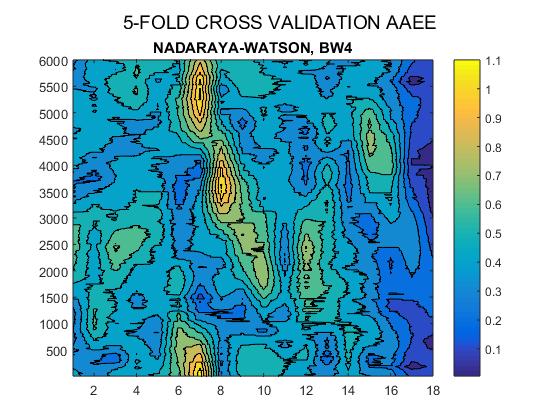}
\includegraphics[width=0.28\textwidth]{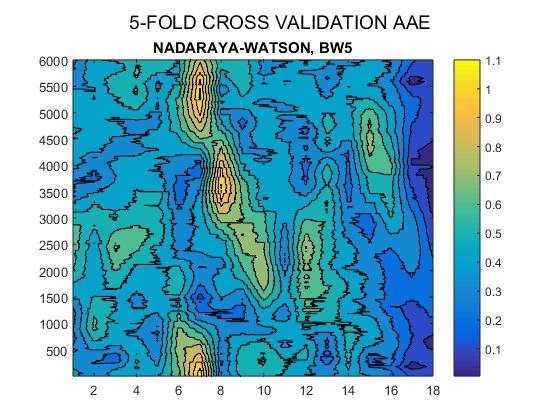}\\
\includegraphics[width=0.28\textwidth]{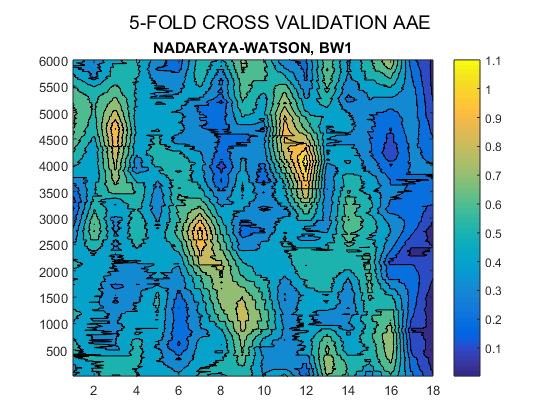}
\includegraphics[width=0.28\textwidth]{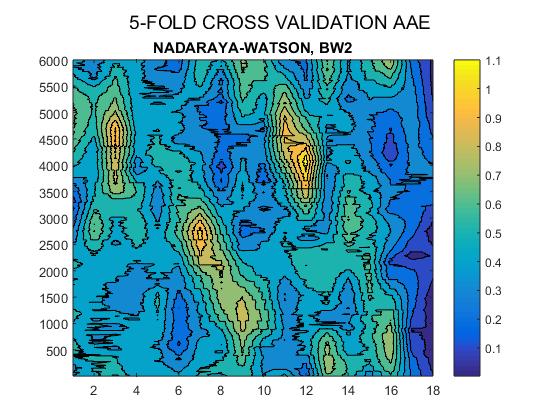}
\includegraphics[width=0.28\textwidth]{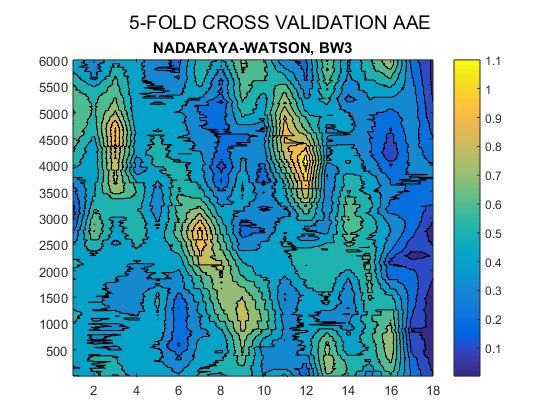}
\includegraphics[width=0.28\textwidth]{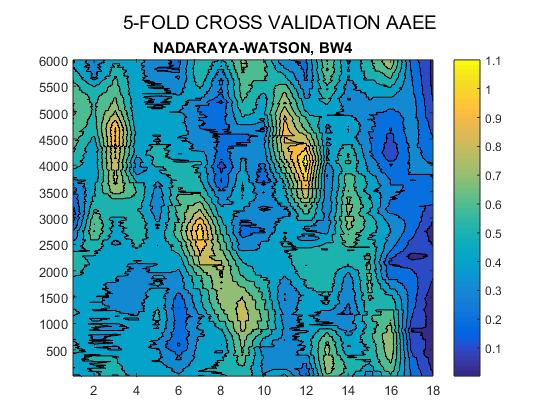}
\includegraphics[width=0.28\textwidth]{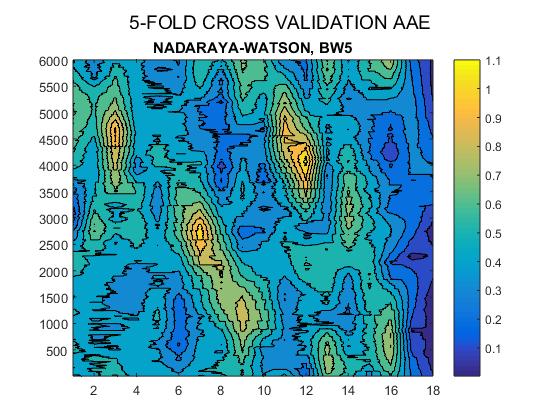}\\
\includegraphics[width=0.28\textwidth]{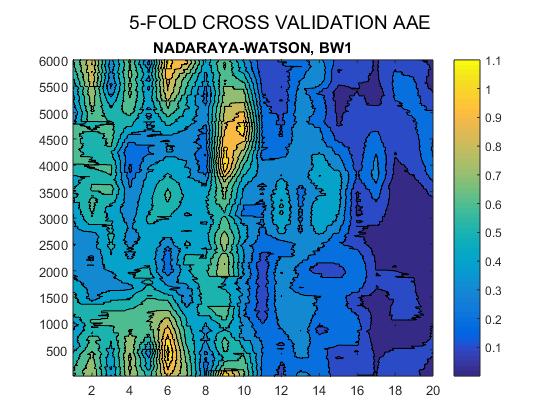}
\includegraphics[width=0.28\textwidth]{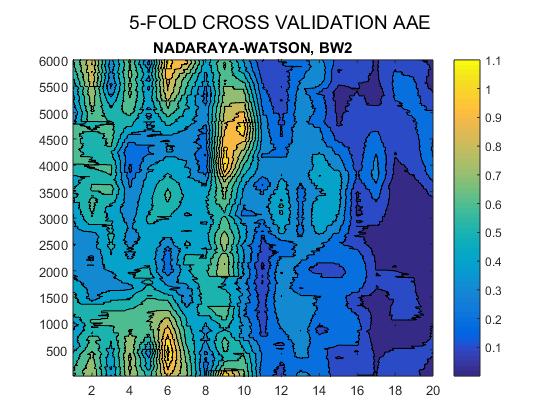}
\includegraphics[width=0.28\textwidth]{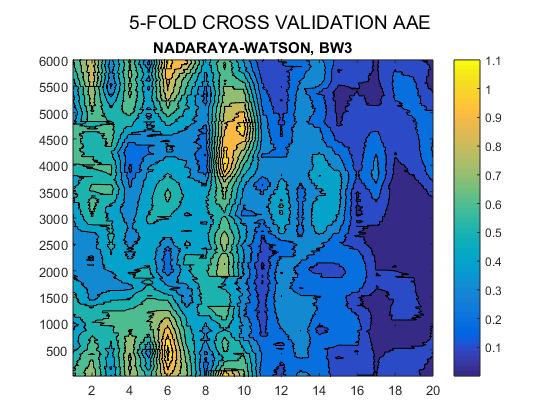}
\includegraphics[width=0.28\textwidth]{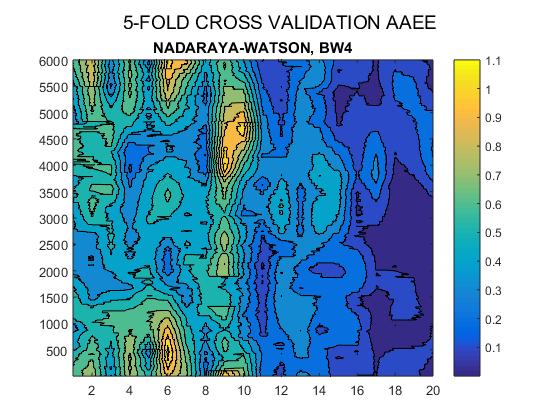}
\includegraphics[width=0.28\textwidth]{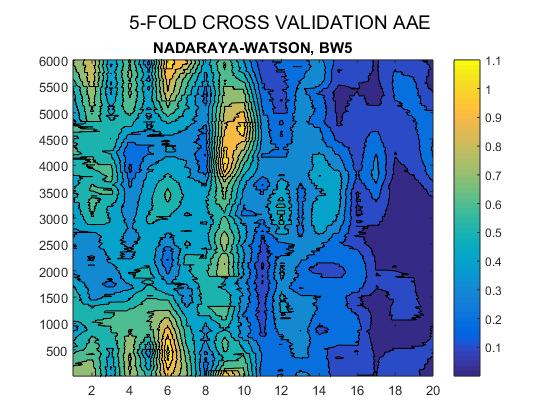}
\caption{\emph{NW--type local  Fr\'echet  regression  curve predictor}. Contourplots of pointwise values at 6000 temporal nodes of the 5-fold cross-validation geodesic  absolute  curve errors, considering  bandwidths $BW1= 0.2000,$   $BW2= 0.2250,$    $BW3=0.2500$ (top)    $BW4=0.2750,$   $BW5= 0.3000$ (bottom), for the months  February 1980 (lines 1-2), March 1980 (lines 3-4) and April 1980 (lines 5-6)}\label{Fig:2.6}
\end{figure}

\begin{figure}
\centering
\includegraphics[width=0.32\textwidth]{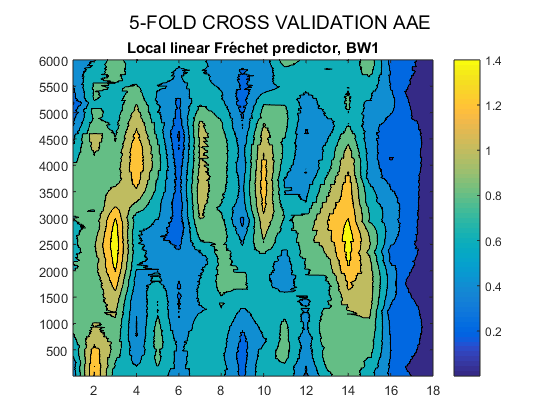}
\includegraphics[width=0.32\textwidth]{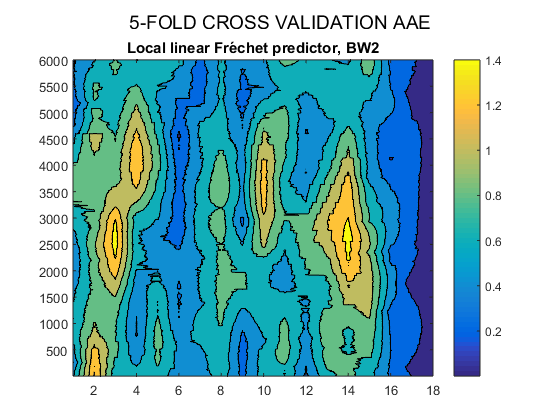}
\includegraphics[width=0.32\textwidth]{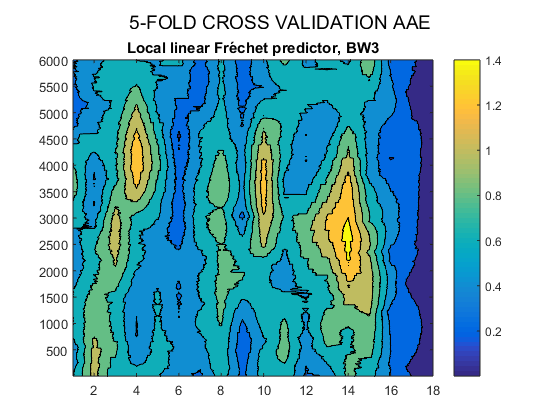}
\includegraphics[width=0.32\textwidth]{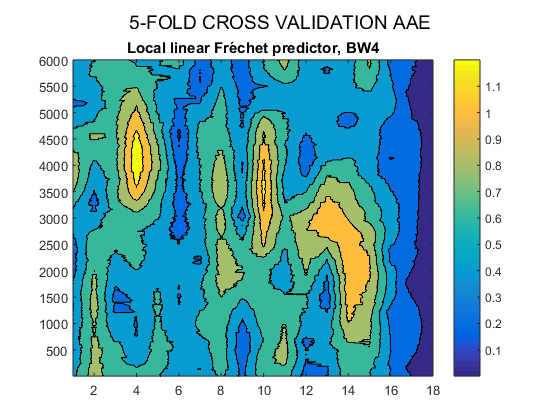}
\includegraphics[width=0.32\textwidth]{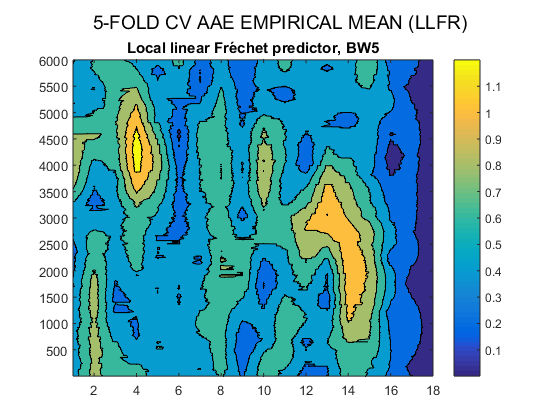}\\
\includegraphics[width=0.32\textwidth]{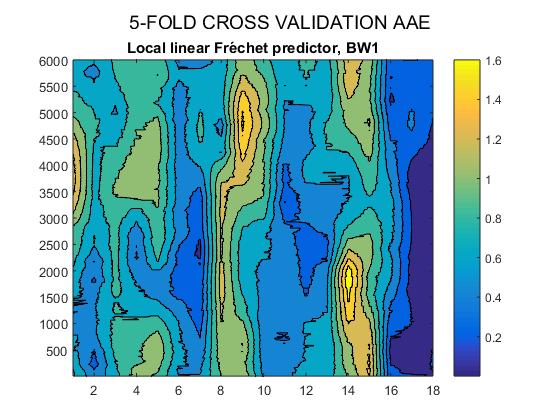}
\includegraphics[width=0.32\textwidth]{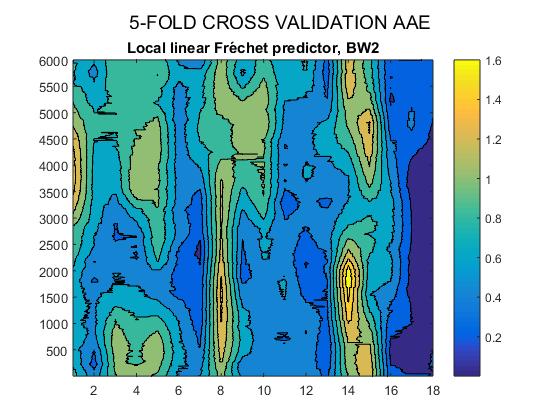}
\includegraphics[width=0.32\textwidth]{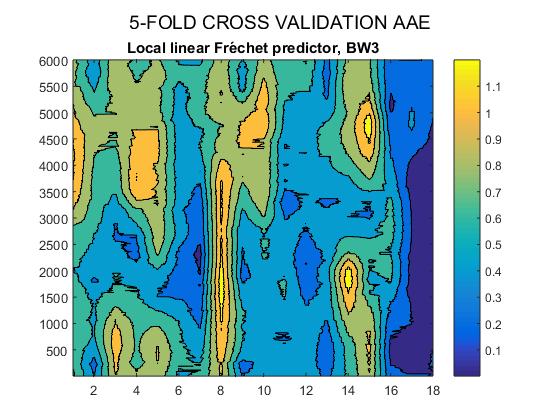}
\includegraphics[width=0.32\textwidth]{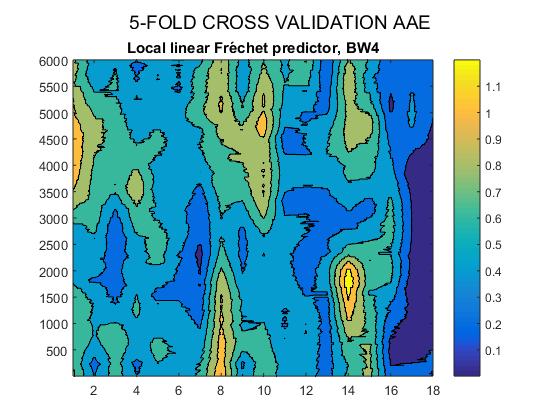}
\includegraphics[width=0.32\textwidth]{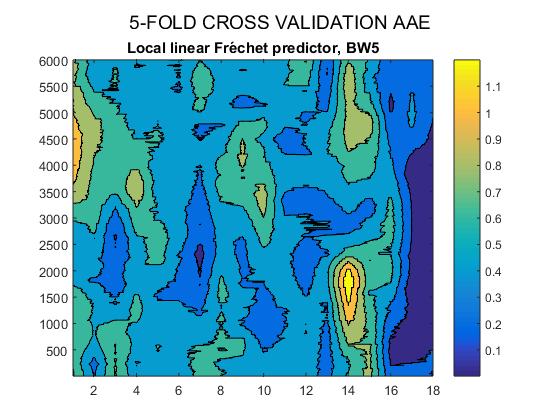}
\caption{ \emph{Intrinsic  local linear  Fr\'echet  regression curve predictor}. Contourplots of pointwise values at 6000 temporal nodes of the 5-fold cross-validation geodesic  absolute  curve errors, considering  bandwidths $BW1= 0.2000,$  $ BW2= 0.2250,$    $BW3=0.2500$ (top)   $BW4=0.2750,$   $BW5= 0.3000$ (bottom), for the months    December 1979 (lines 1-2)  and January 1980 (lines 3-4)
}\label{Fig:2.3}
\end{figure}

\begin{figure}
\centering
\includegraphics[width=0.28\textwidth]{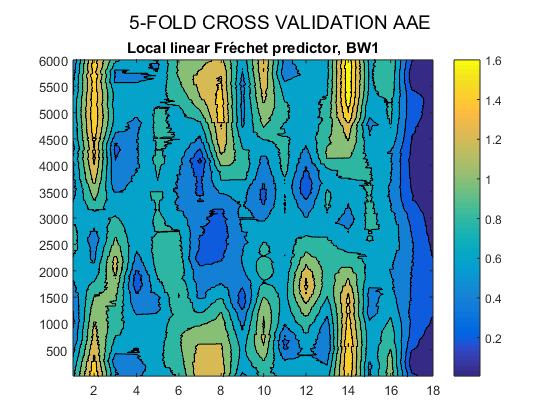}
\includegraphics[width=0.28\textwidth]{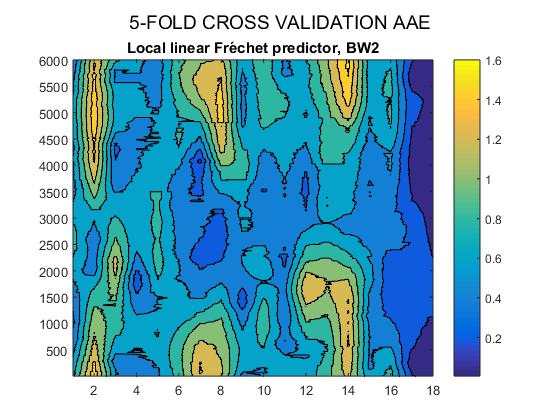}
\includegraphics[width=0.28\textwidth]{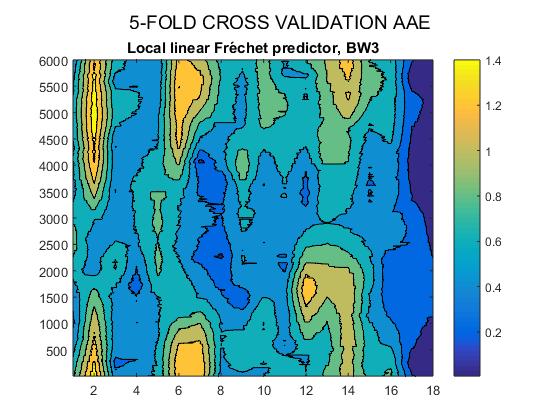}
\includegraphics[width=0.28\textwidth]{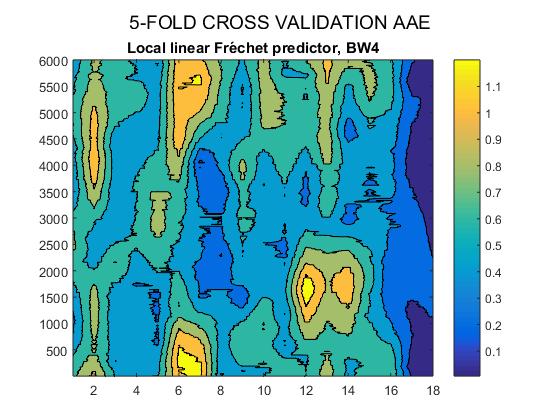}
\includegraphics[width=0.28\textwidth]{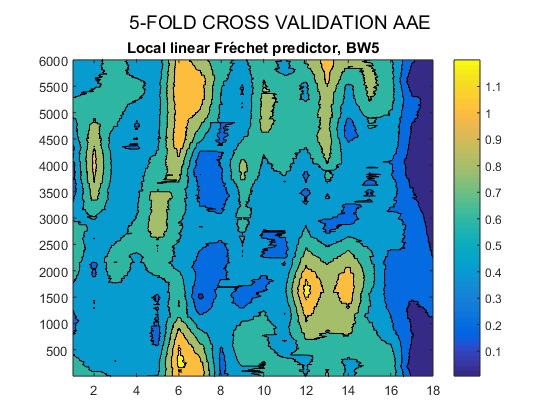}\\
\includegraphics[width=0.28\textwidth]{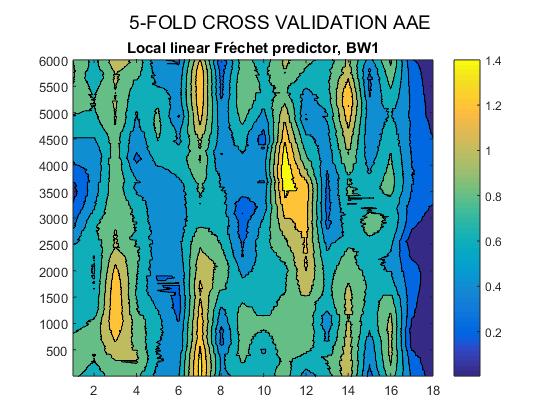}
\includegraphics[width=0.28\textwidth]{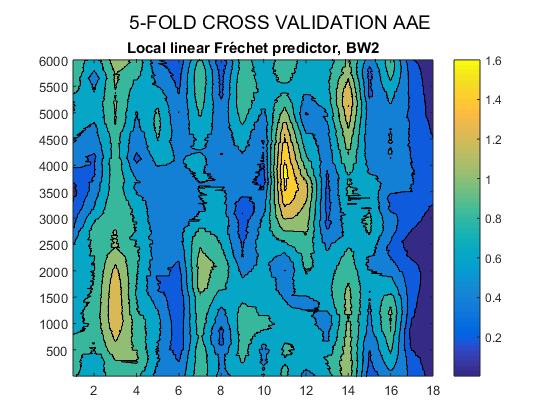}
\includegraphics[width=0.28\textwidth]{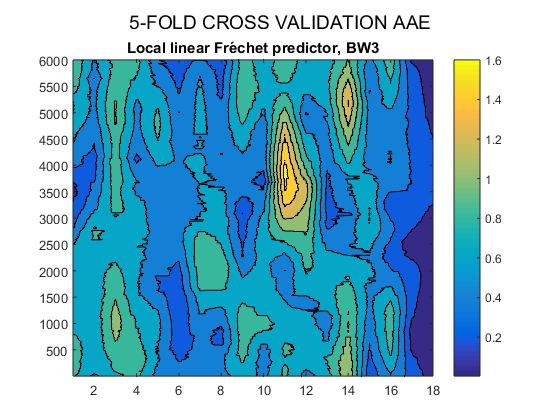}
\includegraphics[width=0.28\textwidth]{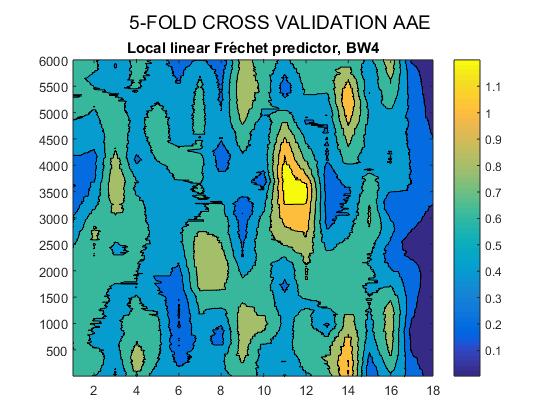}
\includegraphics[width=0.28\textwidth]{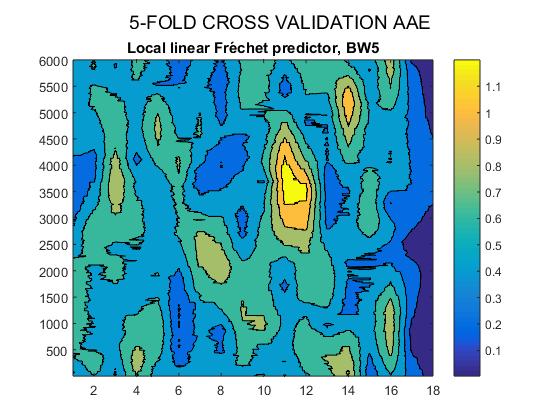}\\
\includegraphics[width=0.28\textwidth]{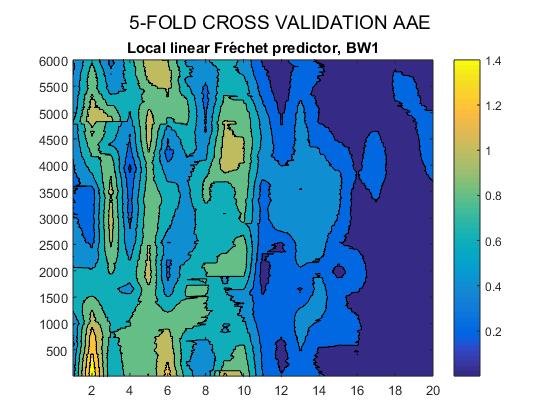}
\includegraphics[width=0.28\textwidth]{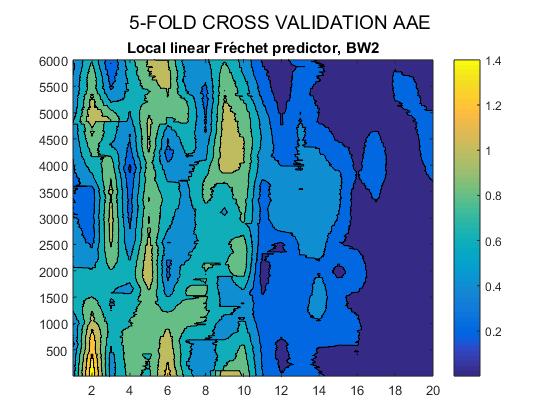}
\includegraphics[width=0.28\textwidth]{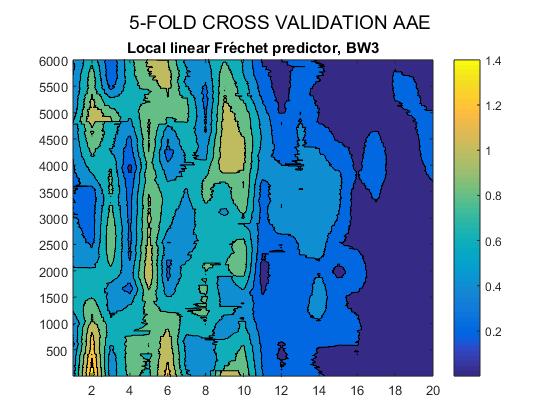}
\includegraphics[width=0.28\textwidth]{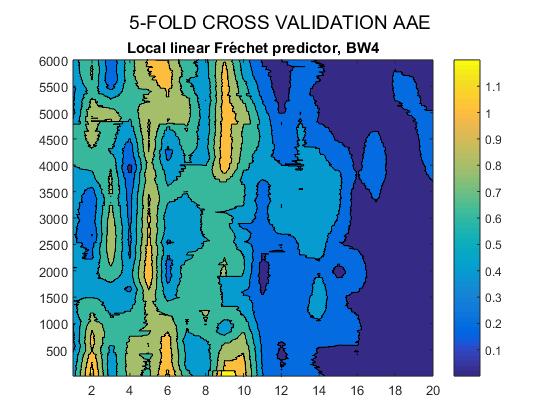}
\includegraphics[width=0.28\textwidth]{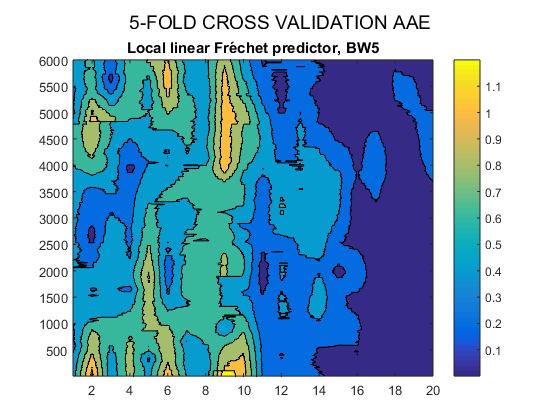}
\caption{ \emph{Intrinsic  local linear  Fr\'echet   regression curve predictor}. Contourplots of pointwise values at 6000 temporal nodes 5-fold cross-validation geodesic  absolute  curve errors, considering  bandwidths $BW1= 0.2000,$  $ BW2= 0.2250,$    $BW3=0.2500$ (top)   $BW4=0.2750,$   $BW5= 0.3000$ (bottom), for the months  February 1980 (lines 1-2), March 1980 (lines 3-4) and April 1980  (lines 5-6)
}\label{Fig:2.4}
\end{figure}

\begin{figure}
\centering
\includegraphics[width=0.48\textwidth]{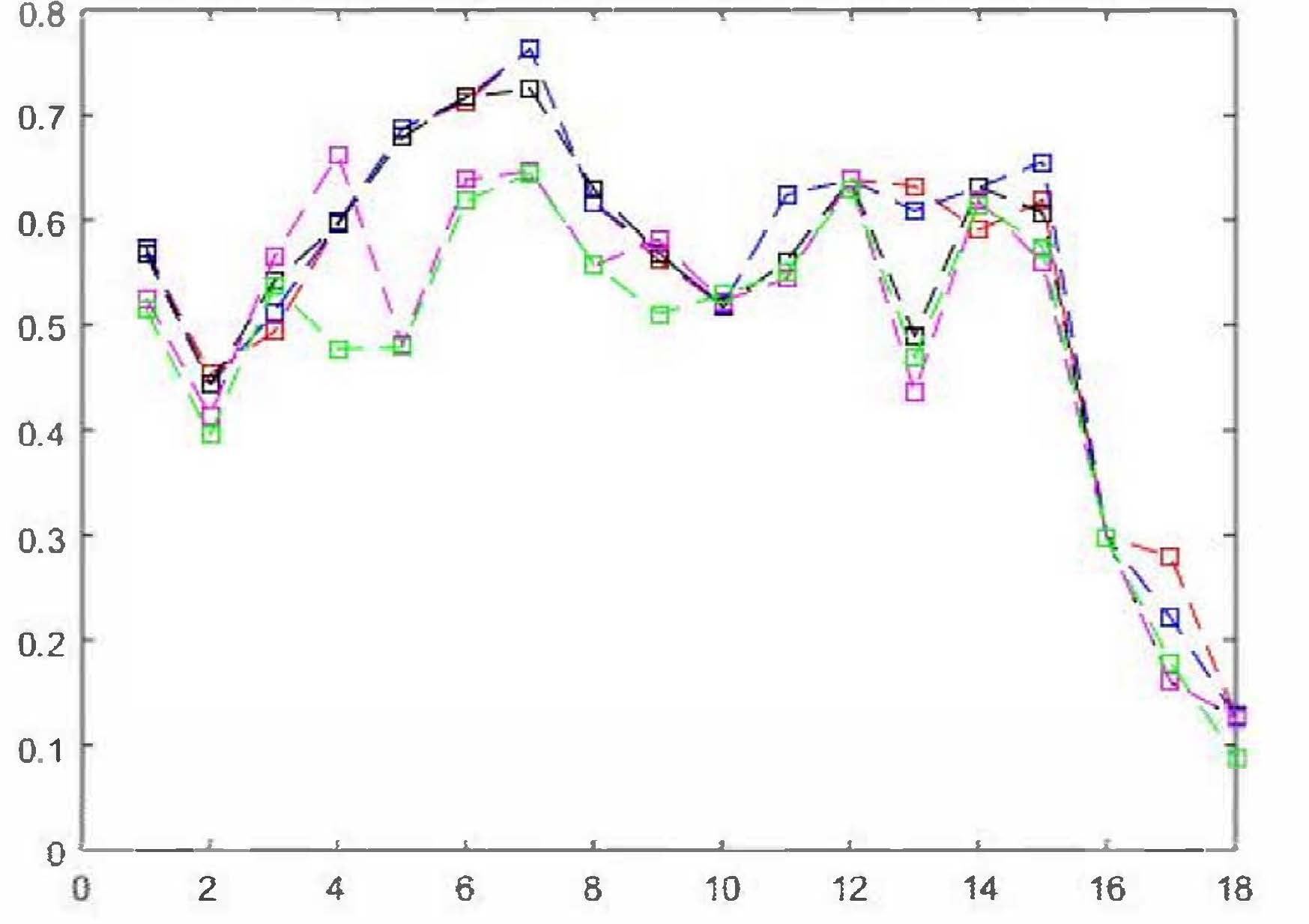}
\includegraphics[width=0.48\textwidth]{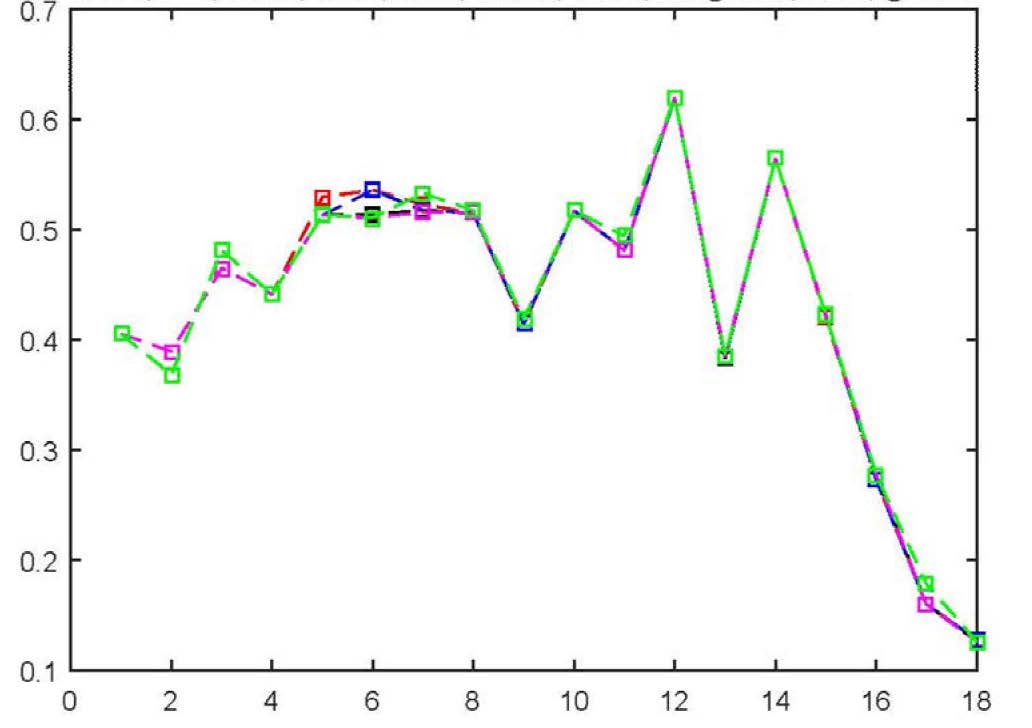}
\includegraphics[width=0.48\textwidth]{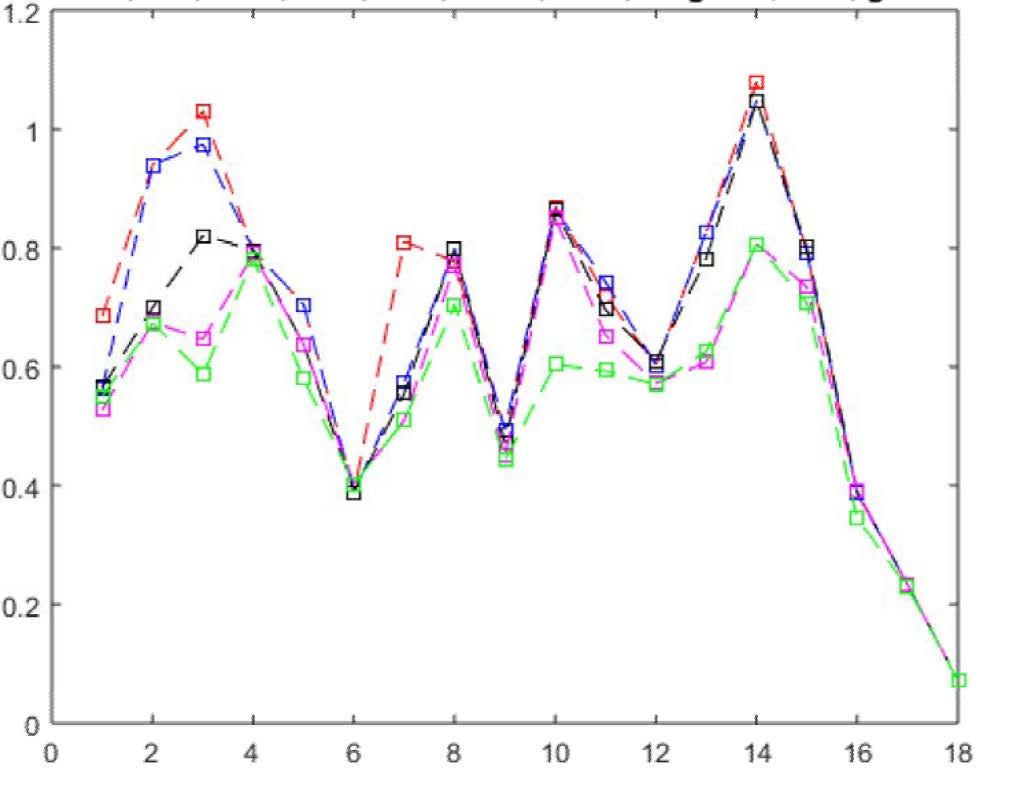}
\includegraphics[width=0.48\textwidth]{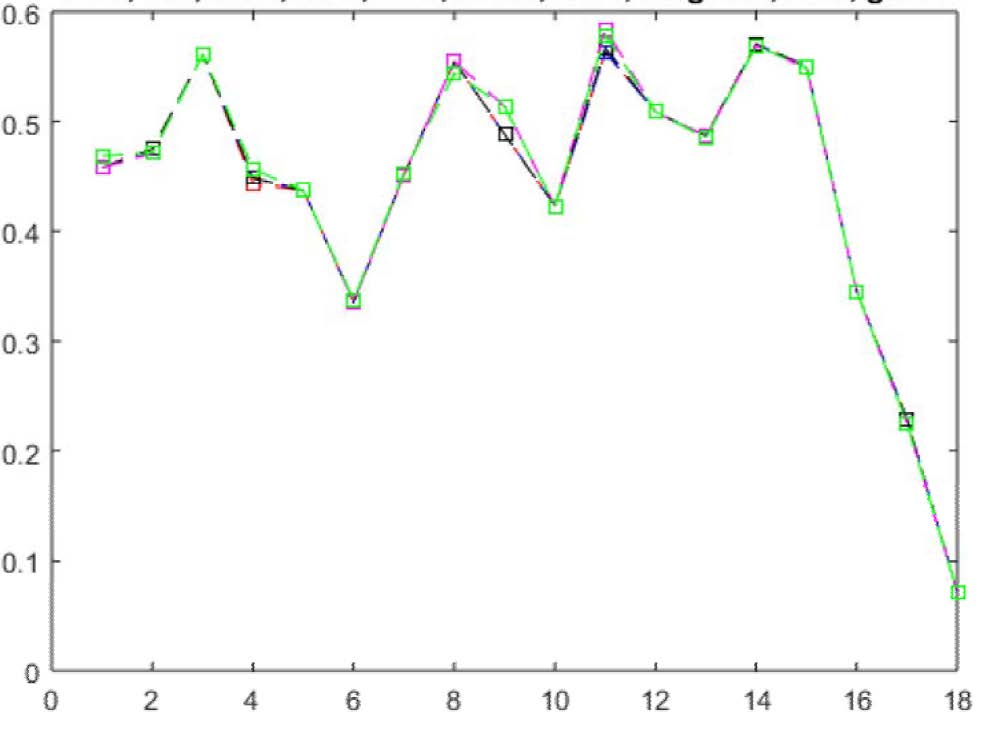}
\includegraphics[width=0.48\textwidth]{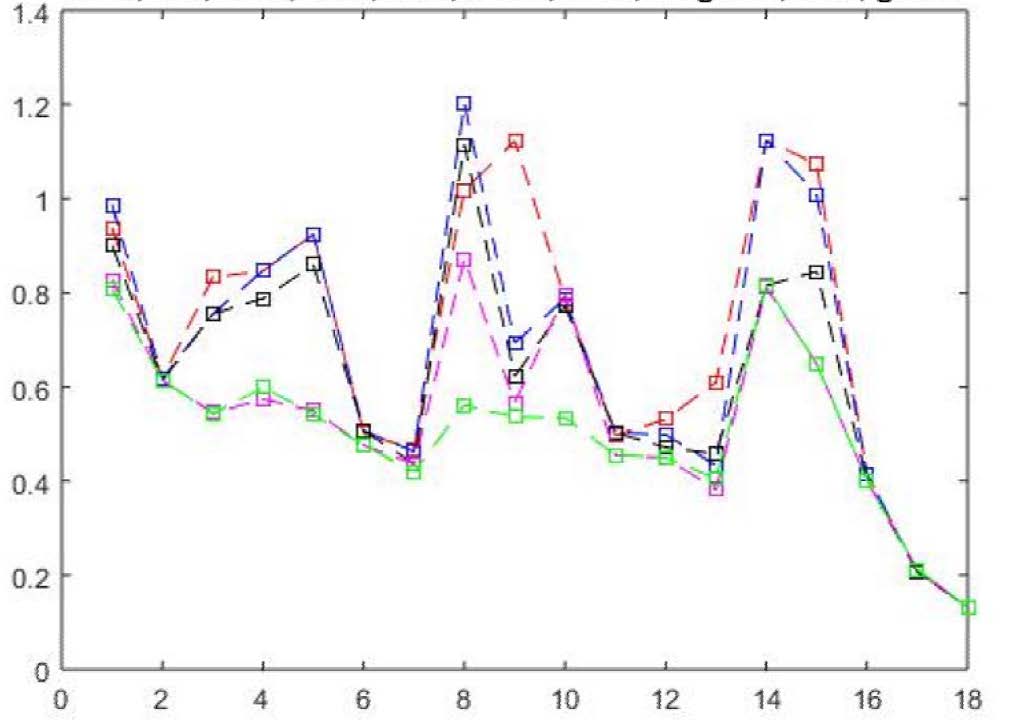}
\includegraphics[width=0.48\textwidth]{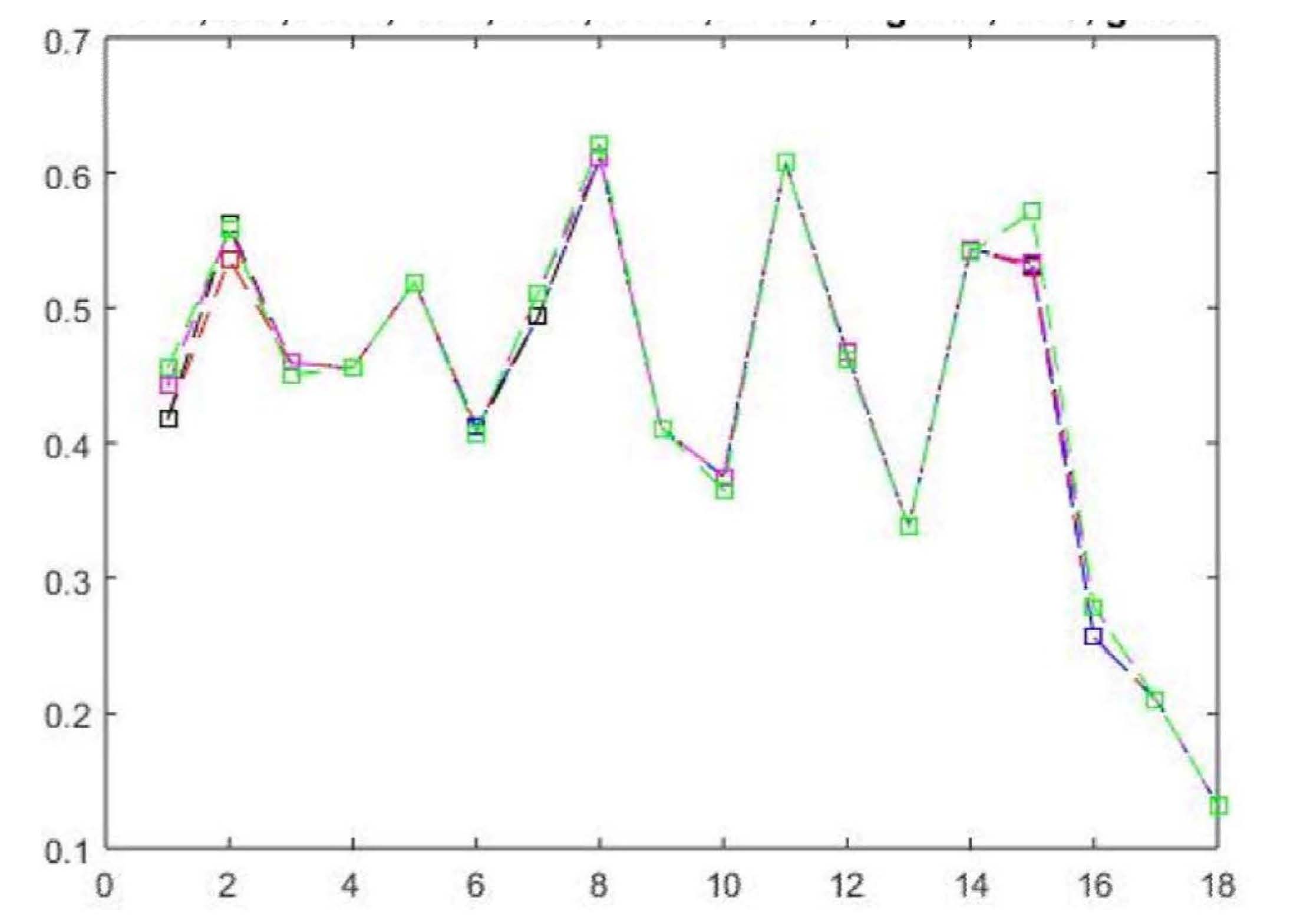}
\caption{ \emph{Intrinsic   local linear, and NW--type  local Fr\'echet curve regression predictors}. Empirical mean inter-curve variability of the 5-fold cross-validation geodesic  absolute  curve errors for local linear Fr\'echet predictor (left-hand side) and NW-type local curve predictor (right-hand side), for the months  November 1979, December 1979 and January 1980 from top to bottom. Bandwidth parameter tested $BW1= 0.2000$ (red),  $ BW2= 0.2250$ (blue)  $BW3=0.2500$ (black)   $BW4=0.2750$ (magenta)   $BW5= 0.3000$ (green)}
\label{Fig:2.7}
\end{figure}

\begin{figure}
\centering
\includegraphics[width=0.48\textwidth]{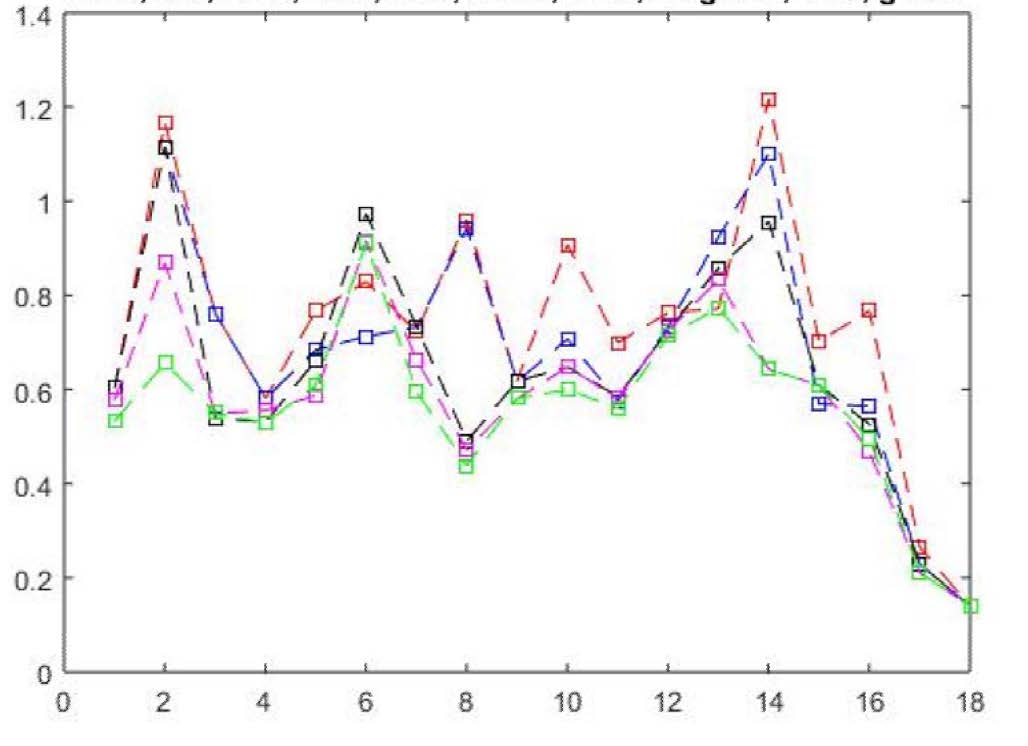}
\includegraphics[width=0.48\textwidth]{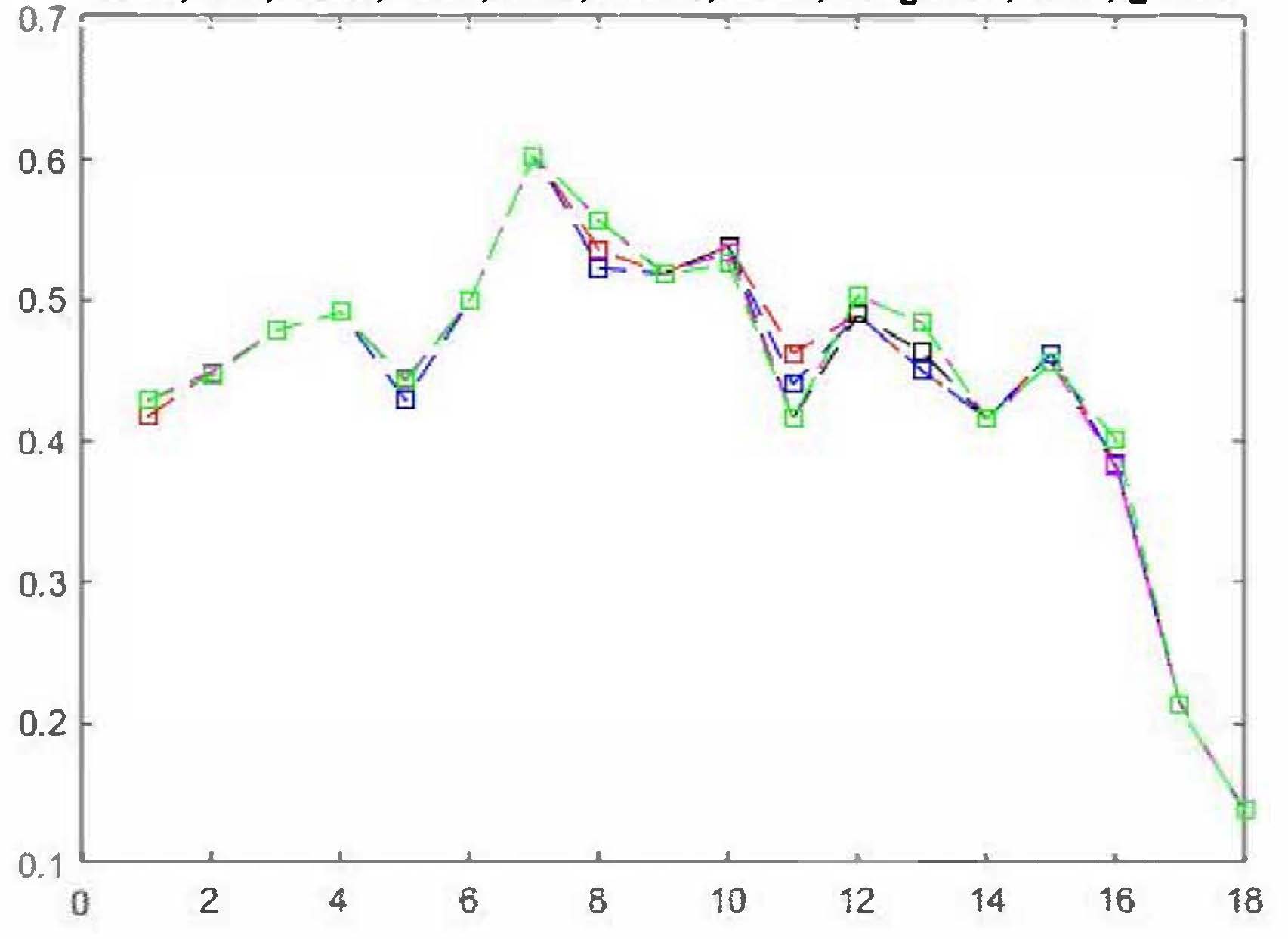}
\includegraphics[width=0.48\textwidth]{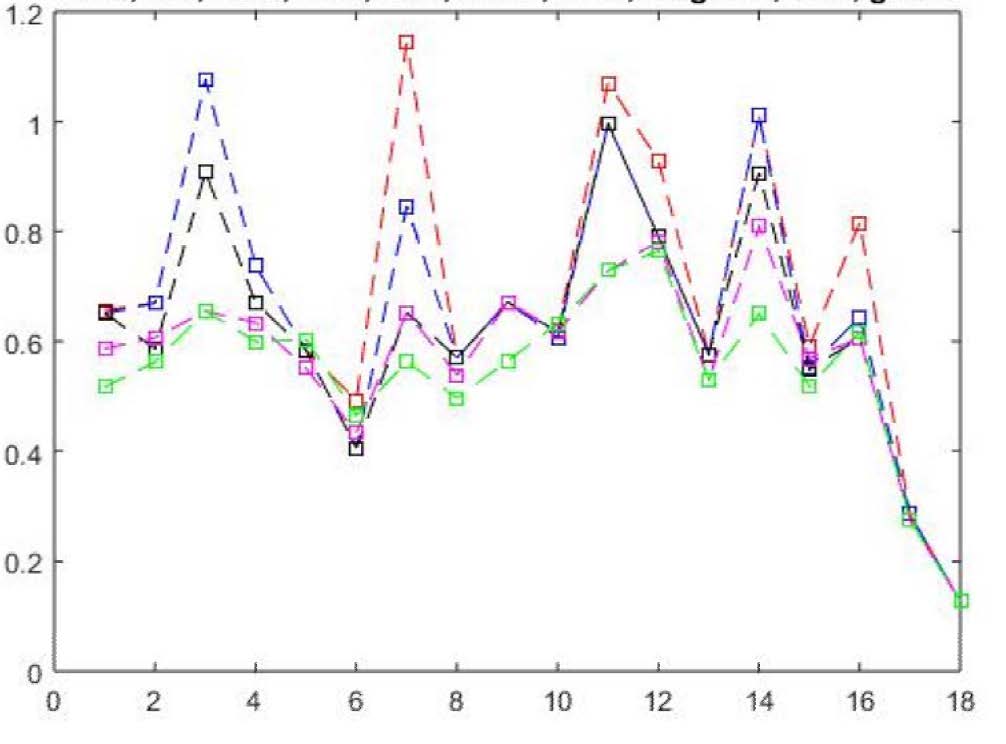}
\includegraphics[width=0.48\textwidth]{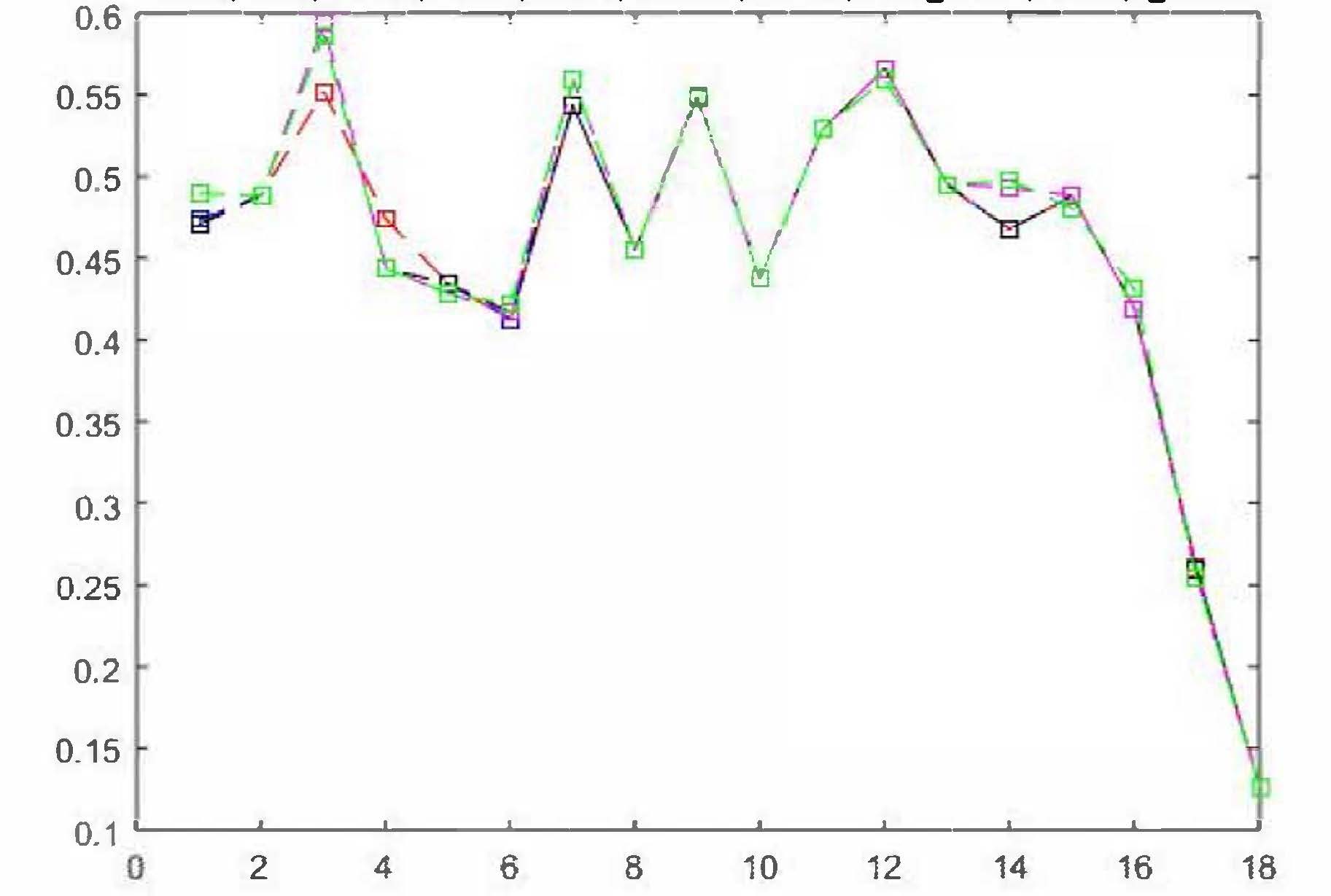}
\includegraphics[width=0.48\textwidth]{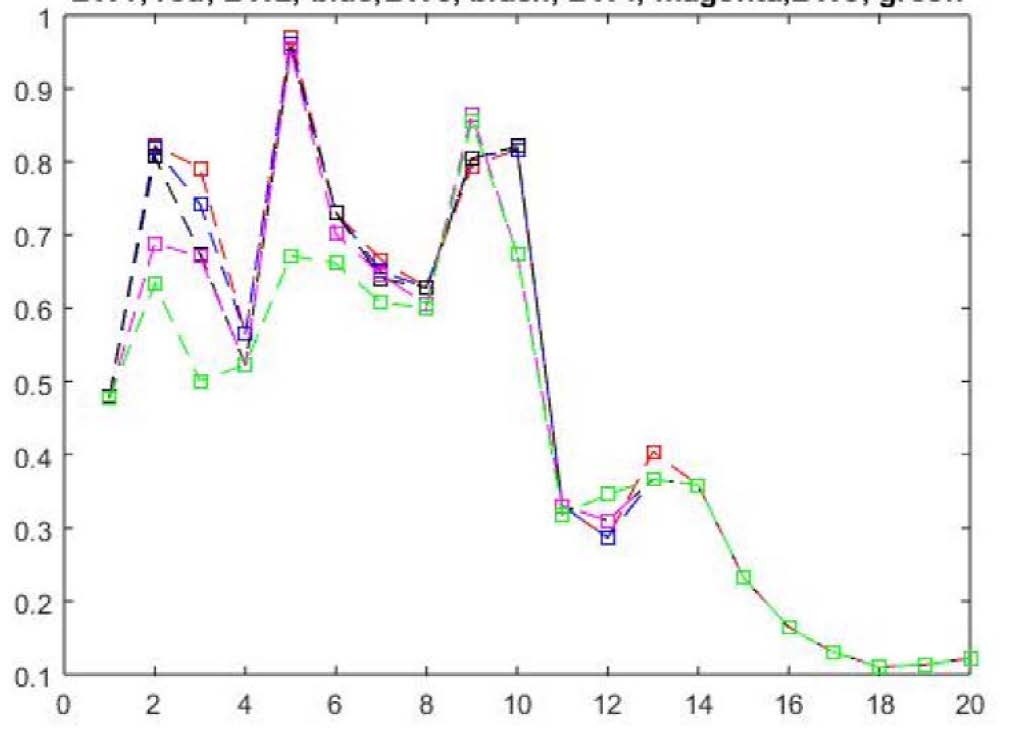}
\includegraphics[width=0.48\textwidth]{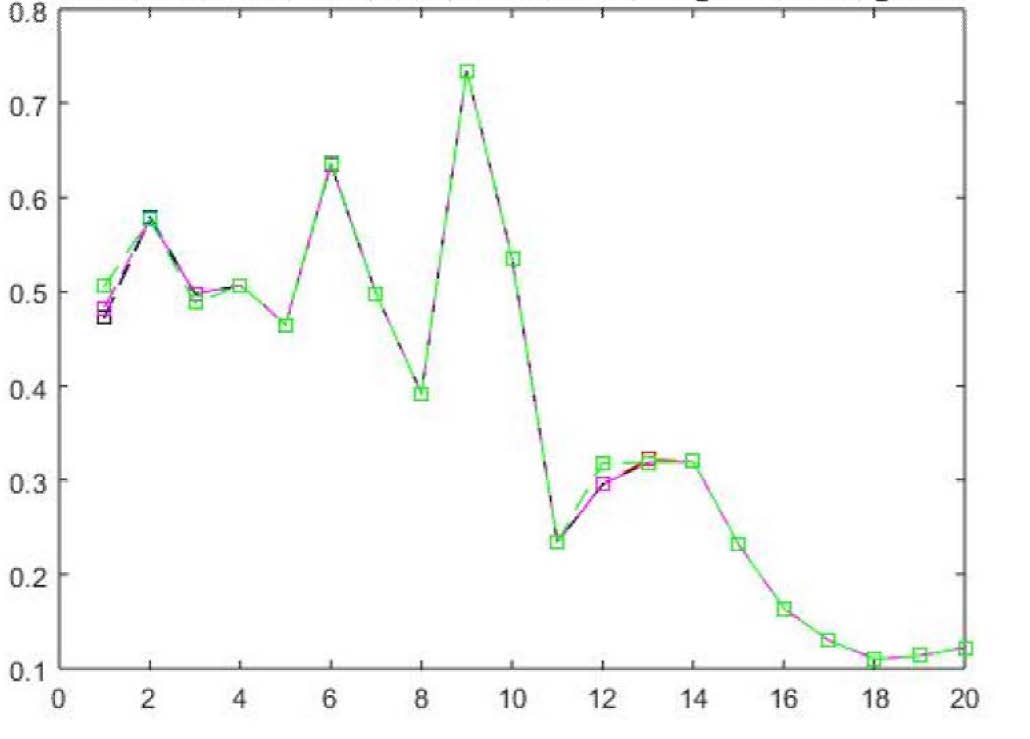}
\caption{ \emph{Intrinsic local linear and NW type local    Fr\'echet curve  regression predictors}. Empirical mean inter-curve variability of the 5-fold cross-validation geodesic  absolute  curve errors for local linear Fr\'echet curve predictor (left-hand side) and NW-type  local curve predictor (right-hand side), for the months  February 1980, March 1980 and April 1980 from top to bottom. Bandwidth parameter tested $BW1= 0.2000$ (red),  $ BW2= 0.2250$ (blue)  $BW3=0.2500$ (black)   $BW4=0.2750$ (magenta)   $BW5= 0.3000$ (green)}
\label{Fig:2.8}
\end{figure}

\clearpage

\subsection*{Acknowledgements}

\medskip

 \noindent This work has been supported in part by projects MCIN/ AEI/PID2022-142900NB-I00, and CEX2020-001105-M MCIN/AEI/10.13039/501100011033).

\end{document}